\documentclass{amsart}
\pdfoutput=1
\usepackage[headings]{fullpage}
\usepackage{amssymb,epic,eepic,epsfig,amsbsy,amsmath,amscd,color,amsthm}
\usepackage[all,cmtip]{xy}
\usepackage{amssymb,amsmath,amscd,color}
\usepackage{graphicx}
\usepackage{texdraw}
\usepackage{url}
\usepackage{mathrsfs}

\usepackage[utf8]{inputenc}
\usepackage{fontenc}
\usepackage{amsfonts}

\usepackage{tikz}
\tikzstyle cross=[preaction={draw=white, -, line width=6pt}]
\tikzstyle normal=[thick]
\usepackage{braids}
\usetikzlibrary{arrows,decorations.markings}

\usepackage{relsize,exscale}

\usetikzlibrary{arrows,decorations.markings}
\usetikzlibrary{cd}
\newcommand{\midarrow}{\tikz \draw[->] (0,0) -- +(.1,0);}

\usepackage{calc}
                        \theoremstyle{plain}

\newtheorem{theorem}{Theorem}[section]

\newtheorem{thm}{Theorem}
\newtheorem{lemma}[theorem]{Lemma}

\newtheorem{coro}[theorem]{Corollary}

\newtheorem{prop}[theorem]{Proposition}

\newtheorem{defn}[theorem]{Definition}

\theoremstyle{definition}

\newtheorem{rmk}[theorem]{Remark}
\newtheorem{example}[theorem]{Example}
\newtheorem*{Not}{Notations}
\newtheorem*{recall}{Recalls}

\theoremstyle{definition}

\def\BC{\mathbb C}
\def\BN{\mathbb N}
\def\BZ{\mathbb Z}
\def\BR{\mathbb R}

\def\BQ{\mathbb Q}

\def\CA{\mathcal A}
\def\CB{\mathcal B}

\def\CF{\mathcal F}

\def\CH{\mathcal H}

\def\CL{\mathcal L}

\def\CR{\mathcal R}

\def\CU{\mathcal U}

\def\Id{\mathrm{Id}}

\def\be { \begin{equation} }
\def\ee { \end{equation} }

\def\bpm { \begin{pmatrix} }
\def\epm { \end{pmatrix} }

\newcommand{\slt}{{\mathfrak{sl}(2)}}
\newcommand{\Uq}{{U_q\slt}}
\newcommand{\UqhL}{{U^{\frac{L}{2}}_q\slt}}

\newcommand{\End}{\operatorname{End}}
\newcommand{\Hom}{\operatorname{Hom}}

\newcommand{\sign}{\operatorname{sign}}

\newcommand{\mt}{\operatorname{\mathtt{t}}}

\newcommand{\RR}{\operatorname{R}}

\newcommand{\perm}{\operatorname{perm}}
\newcommand{\tens}{\operatorname{tens}}

\newcommand{\Conf}{\operatorname{Conf}}

\newcommand{\Rhom}{\operatorname{R}^{\mathrm{hom}}}

\newcommand{\Bn}{\mathcal{B}_n}
\newcommand{\PBn}{\mathcal{PB}_n}

\newcommand{\Sk}{\mathfrak{S}}

\newcommand{\Homeo}{\operatorname{Homeo}}
\newcommand{\Mod}{\operatorname{Mod}}

\newcommand{\Ker}{\operatorname{Ker}}

\newcommand{\qbin}[2]{\left[\begin{array}{c}
      #1 \\
      #2 \end{array}\right]}

\newcommand{\bapp}{\left. \begin{array}{rcl}}
\newcommand{\eapp}{\end{array} \right.}
\newcommand{\bfct}{\left\lbrace \begin{array}{rcl}}
\newcommand{\efct}{\end{array} \right.}

\newcommand{\XR}{X^{\BR}}
\newcommand{\XRm}{X^{\BR,-}}

\newcommand{\Hlf}{\operatorname{H} ^{\mathrm{lf}}}
\newcommand{\Habs}{\operatorname{\CH}^{\mathrm{abs}}}
\newcommand{\Hrelm}{\operatorname{\CH}^{\mathrm{rel }-}}

\newcommand{\Crelm}{\operatorname{C}^{\mathrm{rel }-}}

\newcommand{\Laurent}{\CR}
\newcommand{\Laurentmax}{\Laurent_{\text{max}}}

\newcommand{\EZnr}{E^0_{n,r}}

\title{A homological model for $\Uq$ Verma-modules and their braid representations}
\author{Jules Martel}
\date{}

\begin{document}

\maketitle

\begin{abstract}
We extend Lawrence's representations of the braid groups to relative homology modules, and we show that they are free modules over a ring of Laurent polynomials. We define homological operators and we show that they actually provide a representation for an integral version for $\Uq$. We suggest an isomorphism between a given basis of homological modules and the standard basis of tensor products of Verma modules, and we show it to preserve the integral ring of coefficients, the action of $\Uq$, the braid group representation and its grading. This recovers an integral version for Kohno's theorem relating absolute Lawrence representations with quantum braid representation on highest weight vectors. It is an extension of the latter theorem as we get rid of generic conditions on parameters, and as we recover the entire product of Verma-modules as a braid group and a $\Uq$-module. 
\end{abstract}

\tableofcontents

\section{Introduction}
We give two definitions for the braid group on $n$ strands. 
\begin{defn}\label{Artinpres}
Let $n\in \BN$.
\begin{itemize}
\item The {\em braid group} on $n$ strands $\Bn$ is the group generated by $n-1$ elements satisfying the so called {\em ``braid relations"}:
$$\Bn = \left\langle \sigma_1,\ldots,\sigma_{n-1} \Big| \begin{array}{ll} \sigma_i \sigma_j = \sigma_j \sigma_i & \text{ if } |i-j| \ge 2 \\ 
\sigma_i \sigma_{i+1} \sigma_i = \sigma_{i+1} \sigma_i \sigma_{i+1} & \text{ for } i=1,\ldots, n-2 \end{array} \right\rangle$$
\item The braid group on $n$ strands is the mapping class group of the punctured disk $D_n$ (defined in Section \ref{configandhomology}).
\[
\Bn = \Mod(D_n).
\]
\end{itemize}
\end{defn}

These two definitions provide two different directions to build representations. Quantum representations and homological representations. In this work we relate both theories. Quantum representations are built from the category of modules on a given quantum group which are purely algebraic tools, so that their topological meaning is quite mysterious. We study the quantum group arising from the quantized deformation of $\slt$, namely $\Uq$, and the question of its topological content is then natural. The goal is to relate Verma modules for $\Uq$ to homological modules, in the sense that we want to preserve both the action of $\Uq$, that of the braid group, and an integral structure of coefficients.

\subsection{Homological representations}

In \cite{Law}, R. Lawrence builds graded homological representations of braid groups relying on the fact that braids are associated with homeomorphisms of the punctured disk. Indeed, this generalizes to configuration spaces of $r$ points in the punctured disk denoted $X_r$ and defined in Definition \ref{configspaceofpoints}. This becomes a linear representation while lifted to homology, namely to modules denoted $\Habs_r$ and defined precisely in Definition \ref{defofH}. R. Lawrence builds a family of graded representations for the braid groups over $\Habs_r$ ($r \in \BN$ is the grading) with local coefficients in a ring of Laurent polynomials $\Laurentmax$ over the configuration space of points inside the punctured disk (\cite{Law}). She has developed this idea around 1990 in her thesis, by the time it was already for the purpose of finding topological information in the Jones polynomial, an invariant of knots defined out of quantum representations of braid groups. S. Bigelow used her ideas to recover a definition of the Jones polynomial from these homology modules in \cite{Big3}. This provides a formula for the Jones polynomial as a pairing between homology classes allowing it to be compared with homological invariants in \cite{DWa} and \cite{Man} for instance. 


Lawrence's representations notoriety comes from D. Krammer and S. Bigelow's works, showing their faithfulness at the second level of the grading (\cite{Kra}, \cite{Big1}), the one we refer to as the {\em BKL representation}. It is the first known finite dimensional and faithful linear representation of braid groups. In \cite{P-P}, L. Paoluzzi and L. Paris show that the BKL representation only recovers a sub-representation of the entire homological representation with coefficients in the ring of Laurent polynomials.

\subsection{Quantum representations}

On the quantum side, one can build braid representations over tensor products of {\em Verma modules} for $\Uq$. Namely, let $V$ be the Verma module of $\Uq$ (we won't put the parameter it depends on in notations, so as to simplify them), for $n \in \BN$, the module $V^{\otimes n}$ is endowed with a quantum action of the braid group $\Bn$. Let $r \in \BN$, $W_{n,r}$ be the subspace of $V^{\otimes n}$ generated by vectors of subweight $r$ and $Y_{n,r}$ be the one generated by highest weight vectors of $W_{n,r}$. Spaces $W_{n,r}$ and $Y_{n,r}$ are subrepresentations of $\Bn$, and $V^{\otimes n} = \bigoplus_{r \in \BN} W_{n,r}$. All these definitions are rigorously given in Section \ref{GoodVerma}. 

In \cite{JK}, C. Jackson and T. Kerler establish explicitly an isomorphism between the BKL representation $\Habs_2$ and that on {\em highest weight vectors} and {subweights $2$} denoted $Y_{n,2}$. In \cite{Koh}, T. Kohno shows Lawrence's representations are isomorphic to those from KZ monodromy restricted to highest weight vectors (Kohno's theorem, 2012), themselves previously shown to be isomorphic to the braid representations on highest weight vectors $Y_{n,r}$ in \cite{Drin} and \cite{K0}. This establishes a direct and deep relation between Lawrence's representations and $\Uq$ R-matrix that is summed up in \cite[Theorem~4.5]{Itogarside}. Homological and quantum representations depend on $(n+1)$ variables. One can treat them as parameters, or can take as a ground ring of coefficients Laurent polynomials in these variables denoted $\Laurentmax$ in this work (considering {\em integral versions} for quantum modules). Yet Kohno's isomorphism (between $\Bn$ representations $Y_{n,r}$ and $\Habs_r$) holds for a generic set of parameters (it is not a morphism on the ring of Laurent polynomials, but on $\BC$ when quantum parameters are evaluated at ``generic'' values) and does not recover the whole product of Verma modules, but only the braid group action over the $Y_{n,r}$ for $r\in \BN$. In \cite{FW}, G. Felder and C. Wieczerkowski build an action of the quantum group $\Uq$ on some module generated by topological objects of the punctured disk - {\em $r$-loops} - together with a natural action of the braid groups which commutes with the quantum one. The homological interpretations of this module remain conjectures (\cite[Conjecture~6.1,~6.2]{FW}) as well as its links with Lawrence's theory. Finally, in \cite{SchVar}, V. Schechtman and A. Varchenko obtain representations of quantum groups on some local system homology on configuration spaces of points. 
We sum-up the brief history of Lawrence's representations in three results.

\begin{theorem}\label{existingtheorems}
\begin{itemize}
\item[(i)] For all $r\in \BN$, $\Habs_r$ is a representation of $\Bn$. (\cite{Law})
\item[(ii)] The representation $\Habs_2$ is faithful. (\cite{Big1,Kra}).
\item[(iii)] There exists an isomorphism of $\Bn$-representations between $\Habs_r$ and quantum module $Y_{n,r}$. (\cite{JK} case $r=2$, \cite{Kohmulti} for generic values of parameters $q,\alpha_k$). 
\end{itemize}
\end{theorem}

\subsection{Results of the paper}

The present work extends Lawrence's representations via relative homology, it clarifies and generalizes their links with quantum representations of braid groups obtained on tensor products of $\Uq$ Verma's. Inspired by \cite{FW}, we extend Lawrence modules to relative homology modules denoted $\Hrelm_r$ and defined in Definition \ref{defofH}. We endow these complexes with a homological action of the quantum group $\UqhL$ (an integral version for $\Uq$ defined in Section \ref{halfLusztigversion}) via homological actions of its generators (defined in Section \ref{homologicalactionofUq}), that leads to the following result.

\begin{theorem}[Theorem \ref{homologicalHabiro}, Section \ref{sectionhomologicalhabiro}]
The module $\CH = \bigoplus_{r\in \BN}  \Hrelm_r$ over the ring of Laurent polynomials $\Laurentmax$ is a representation of $\UqhL$. 
\end{theorem}

In Proposition \ref{HrelTrick}, we show that modules $\Hrelm_r$ are free modules on the ring of Laurent polynomials $\Laurentmax$, and that a basis (said ``integral'') is given by the family of  {\em multi-arcs}, see Corollary \ref{Arcsbasis}. This helps us recognizing this $\UqhL$ representation as a tensor product of Verma modules, what we sum-up in the following statement. 

\begin{theorem}[Theorem \ref{monoidality}, Section \ref{sectionmonoidality}]
For all $n \in \BN$, there exists a morphism of $\UqhL$-modules :
\[
V^{\otimes n} \to \CH
\] 
such that the standard integral basis of $V^{\otimes n}$ is sent to the {\em multi-arcs} basis. The integer $n$ corresponds to the number of punctures of the disk $D_n$ used to define the configuration space $X_r$.
\end{theorem}

Finally, we extend the natural Lawrence action of braid groups over these homological modules, and we show that it is the R-matrix representation obtained using $\UqhL$ Verma modules.

\begin{theorem}[Theorem \ref{homoquantumbraided}, Section \ref{sectionbraidhabiro}] \label{thmtressesintroeng}
For all $n \in \BN$ and all $r \in \BN$, the morphism :
\[
W_{n,r} \to \Hrelm_r
\]
induced by the previous theorem is an isomorphism of $\Bn$ - representations, so much that the morphism:
\[
V^{\otimes n} \to \CH = \bigoplus_{r\in \BN} \Hrelm_r
\]
from previous theorem is a morphism of $\UqhL$-modules {\em and} of $\Bn$-modules.
\end{theorem}

We provide an integral basis for homology (i.e. basis as module on an integral ring of Laurent polynomials). The $\UqhL$-action and the $\Bn$-action preserve this structure, so does the isomorphism to the tensor product of Verma modules. This is an improvement regarding previous models, and is hopeful for topological quantum invariants built from these braid representations that needs parameters to be evaluated. For instance, $q$ being a root of unity is required to study several quantum invariants, and was not recovered by generic conditions of Kohno's Theorem (Theorem \ref{existingtheorems} (iii)). 

We show that the long exact sequence of relative homology becomes, in this model, a short one:
\[
0 \to \Habs_r \to \Hrelm_r \to H_{r-1}(X_r^-) \to 0,
\]
(where $X_r^-$ is defined in Definition \ref{defofH}) so that $\Hrelm_r$ extend Lawrence's representations. This work allows then an extension of Kohno's theorem beyond highest weight vectors, and to recover homologically the entire tensor product of $\Uq$ Verma modules. Lawrence's representations are sub-representations of it so that Kohno's theorem is a corollary of this work. Generic hypotheses are clarified and become algebraic thanks to the fact that all isomorphisms preserve the integral structure of coefficients, and the links between an integral basis (multi-arcs) and the multifork basis from Kohno's theorem are explicit. All of this is summed-up in Corollary \ref{kohnotheorem} and Proposition \ref{genericFork} in Section \ref{integralKohno}.

The obtained homological representations are a generalization of Lawrence's representations so they are generically faithful. They allow a homological recovering of several properties of the category of $\Uq$-modules.

We illustrate the weight structure of tensor product of Verma modules in the following diagram, at level $r$ of the grading:
\begin{center}
\begin{tikzpicture}
\tikzset{node distance=1.3cm, auto}

\node (A1) {$\cdots$ };
\node (A2) [right of=A1,,xshift=5cm] {$\cdots$};
\node (B1) [below of=A1] {$W_{n,r}$ };
\node (B2) [right of=B1,,xshift=5cm] {$\Hrelm_r$};
\node (C1) [below of=B1] {$W_{n,r+1}$};
\node (C2) [right of=C1,,xshift=5cm] {$\Hrelm_{r+1}$};
\node (D1) [below of=C1] {$\cdots$};
\node (D2) [right of=D1,,xshift=5cm] {$\cdots$};
\draw[<->] (B1) to node {  } (B2);
\draw[<->] (C1) to node {  } (C2);
\draw[->] (A1) to[bend left=40] node[midway,right] {\small $F$} (B1);
\draw[->] (B1) to[bend left=30] node[midway,left] {\small $E$} (A1);
\draw[->] (B1) to[bend left=30] node[midway,right] {\small $F$} (C1);
\draw[->] (C1) to[bend left=30] node[midway,left] {\small $E$} (B1);
\draw[->] (C1) to[bend left=30] node[midway,right] {\small $F$} (D1);
\draw[->] (D1) to[bend left=30] node[midway,left] {\small $E$} (C1);
\draw[->] (A2) to[bend left=40] node[midway,right] {\small $F$} (B2);
\draw[->] (B2) to[bend left=30] node[midway,left] {\small $E$} (A2);
\draw[->] (B2) to[bend left=30] node[midway,right] {\small $F$} (C2);
\draw[->] (C2) to[bend left=30] node[midway,left] {\small $E$} (B2);
\draw[->] (C2) to[bend left=30] node[midway,right] {\small $F$} (D2);
\draw[->] (D2) to[bend left=30] node[midway,left] {\small $E$} (C2);
%
\end{tikzpicture} 
\end{center}
Horizontal arrows correspond to isomorphisms of braid representations from Theorem \ref{thmtressesintroeng}, while vertical arrows correspond to $\Uq$ generators action $E,F$: the quantum ones on the left side and the homological ones (homological definitions are given in this work) on the right side, that rules the weight structure on Verma modules. The direct sum of all spaces aligned vertically on the left gives the tensor product of Verma modules $V^{\otimes n}$ , while the one of all spaces aligned on the right corresponds to the homological module $\CH$. The homological interpretation of $\Uq$ generators follows, together with the ones of relations they satisfy and the R-matrix built using these generators. 

\subsection{Plan of the paper}

%
In Section \ref{configandhomology} we define topological spaces and homology modules used to build homological representations. In Section \ref{structureofhomo} we give examples of homology classes in $\Hrelm_r$, representing them by multi-arcs diagrams, then we study the structure of the homology complexes of interest. Namely, we prove the crucial Proposition \ref{HrelTrick}, stating that modules $\Hrelm_r$ are free over the ring of Laurent polynomials $\Laurentmax$, and that it is the only non vanishing module of the entire homology complex. In Section \ref{computationrules} we state all the rules we need to do computation in $\Hrelm_r$, and we use them to show that the family of multi-arcs is a basis of $\Hrelm_r$ as a module over $\Laurentmax$ in Proposition \ref{Arcsbasis}. In Section \ref{quantumalgebra} we recall definitions and notations for quantum algebra. Namely we define an integral version (i.e. as a free $\Laurentmax$-module) of $\Uq$ denoted $\UqhL$, and its version for Verma modules. We then present the braid representations defined over tensor product of Verma modules, and how to get finite dimensional representation out of them in Remark \ref{GoodsubrepKJ}. Finally main results of this paper can be found in Section \ref{homologicalmodel}. In Subsection \ref{homologicalactionofUq} we define homological operators corresponding to generators of $\UqhL$ and we prove Theorem \ref{homologicalHabiro} stating that it provides a representation of $\UqhL$. In Subsection \ref{computationUqaction} we compute the homological action of $\UqhL$ in the multi arcs basis and we prove Theorem \ref{monoidality} saying that this homological representation is isomorphic to a tensor product of Verma modules. In Subsection \ref{homologicalbraid} we recall how to build a homological action of braid groups over homological modules. Then we prove Theorem \ref{homoquantumbraided} saying that the isomorphism of $\UqhL$-modules relating homological modules with Verma modules is also an isomorphism of $\Bn$-representations and that it preserves their grading. In Section \ref{integralKohno} we show that Theorem \ref{homoquantumbraided} recovers Kohno's theorem (Theorem \ref{existingtheorems}, (iii)) in an integral version, and exhibits previous generic conditions on parameters required. In Section \ref{FWconjectures} we give partially positive answers to Conjecture 6.2 of \cite{FW}. Section \ref{appendix} is an Appendix recalling some definitions of homology theories we use, namely the locally finite (Borel-Moore) version of singular homology and the local ring of coefficients set-up. 

{\bf Acknowledgment} This work was achieved during the PhD of the author that was held in the {\em Institut de Mathématiques de Toulouse}, in {\em Université Paul Sabatier, Toulouse 3}. The author thanks very much his advisor Francesco Costantino for asking this problem, and is very grateful for all discussions, fruitful remarks and help that led to this paper. The author is also very grateful to anonymous referees for their relevant and important remarks and corrections. The author thanks L.-H. Robert and E. Wagner for useful comments. 

\section{Configuration space and homology}\label{configandhomology}


\begin{defn}\label{configspaceofpoints}
Let $r\in \BN$, $n \in \BN$, $D$ be the unit disk, and $\left\lbrace w_1 , \ldots , w_n \rbrace\right. \in D^n$ points lying on the real line in the interior of $D$. Let $D_n = D \setminus \left\lbrace w_1 , \ldots , w_n \rbrace\right.$ be the unit disk with $n$ punctures. Let:
\[
\Conf_r(D_n):= \left\lbrace (z_1 , \ldots , z_r ) \in (D_n)^r \text{ s.t. } \begin{array}{c} z_i \neq z_j \forall i,j  \end{array} \right\rbrace
\]
be the configuration space of points in the punctured disk $D_n$. 
We define the following space:
\begin{eqnarray}\label{NotConfig}
X_r(w_1 , \ldots , w_n) & = &  \Conf_r(D_n) \Big/ \Sk_r. 
\end{eqnarray}
to be the space of {\em unordered} configurations of $r$ points inside $D_n$, where the permutation group $\Sk_r$ acts by permutation on coordinates.
\end{defn} 

When no confusion arises in what follows, we omit the dependence in $w_1 , \ldots , w_n$ to simplify notations. All the following computations rely on a choice of base point that we fix from now on.

\begin{Not}[Base point]\label{basepoint}
Let ${\pmb \xi^r}= \lbrace  \xi_1 , \ldots , \xi_r \rbrace$ be the base point of $X_r$
chosen so that $ \xi_i \in \partial D_n$ $\forall i$ with negative imaginary parts, and so that:
\[
\Re(w_0)< \Re(\xi_r)< \Re(\xi_{r-1})< \ldots <\Re(\xi_1) <\Re(w_1).  
\]
We illustrate the disk with chosen points in the following figure.

\begin{equation*}
\begin{tikzpicture}[scale=0.7]
\node (w0) at (-3,0) {};
\node (w1) at (1,0) {};
\node (w2) at (2,0) {};
\node[gray] at (2.8,0) {\ldots};
\node (wn) at (3.4,0) {};


\draw[thick,gray] (4,2) -- (-3,2) -- (-3,-2) -- (4,-2) -- (4,2);


\node[below,red] at (-1,-2) {$\xi_r$};
\node[below,red] at (-0.3,-2) {$\xi_{r-1}$};
\node[below=3pt,red] at (0.3,-2) {\small $\ldots$};
\node[below,red] at (0.8,-2) {$\xi_{1}$};

\node at (-1,-2)[red,circle,fill,inner sep=1pt]{};
\node at (-0.3,-2) [red,circle,fill,inner sep=1pt]{};
\node at (0.8,-2) [red,circle,fill,inner sep=1pt]{};

%

\node[gray] at (w0)[left=5pt] {$w_0$};
\node[gray] at (w1)[above=5pt] {$w_1$};
\node[gray] at (w2)[above=5pt] {$w_2$};
\node[gray] at (wn)[above=5pt] {$w_n$};
\foreach \n in {w1,w2,wn}
  \node at (\n)[gray,circle,fill,inner sep=3pt]{};
\node at (w0)[gray,circle,fill,inner sep=3pt]{};
\end{tikzpicture} .
\end{equation*}

We draw a square boundary for the disk, in order for the reader not to confuse it with arcs we will be drawing inside.

\end{Not}

We give a presentation of $\pi_1(X_r, {\pmb \xi^r})$ as a braid subgroup ({\em the mixed braid group}), which can be deduced from the one given in the introduction of \cite{Z1}, and will be explain with drawings. 

\begin{rmk}\label{pi_1X_r}
The group $\pi_1(X_r, {\pmb \xi^r})$ is isomorphic to the subgroup of $\CB_{r+n}$ generated by:
\[
\langle \sigma_1 , \ldots , \sigma_{r-1}, B_{r,1} , \ldots , B_{r,n} \rangle 
\]
where the $\sigma_i$ ($i=1,\ldots ,r-1$) are standard generators of $\CB_{r+n}$, and $B_{r,k}$ (for $k=1,\ldots ,n$) is the following pure braid:
\[
B_{r,k} = \sigma_{r} \cdots \sigma_{r+k-2} \sigma_{r+k-1}^2 \sigma_{r+k-2}^{-1} \cdots \sigma_{r}^{-1} .
\]
\end{rmk}

To see the correspondence between loops in $X_r$ and generators of the above braid subgroup we draw two examples.

\begin{example}
Two types of braid generators for $\pi_1(X_r, {\pmb \xi^r})$ are given in Remark \ref{pi_1X_r}, which correspond to two types of loops generating $\pi_1(X_r, {\pmb \xi^r})$. We give examples for both kinds.

\begin{itemize}
\item The braid $\sigma_1$ corresponds to a loop swapping $\xi_r$ and $\xi_{r-1}$ letting other base point coordinates fixed. This can be seen by drawing the movie of the loop in Figure \ref{pi1sigma}.
\begin{figure}[h!]
\begin{center}
\def\svgwidth{0.5\columnwidth}
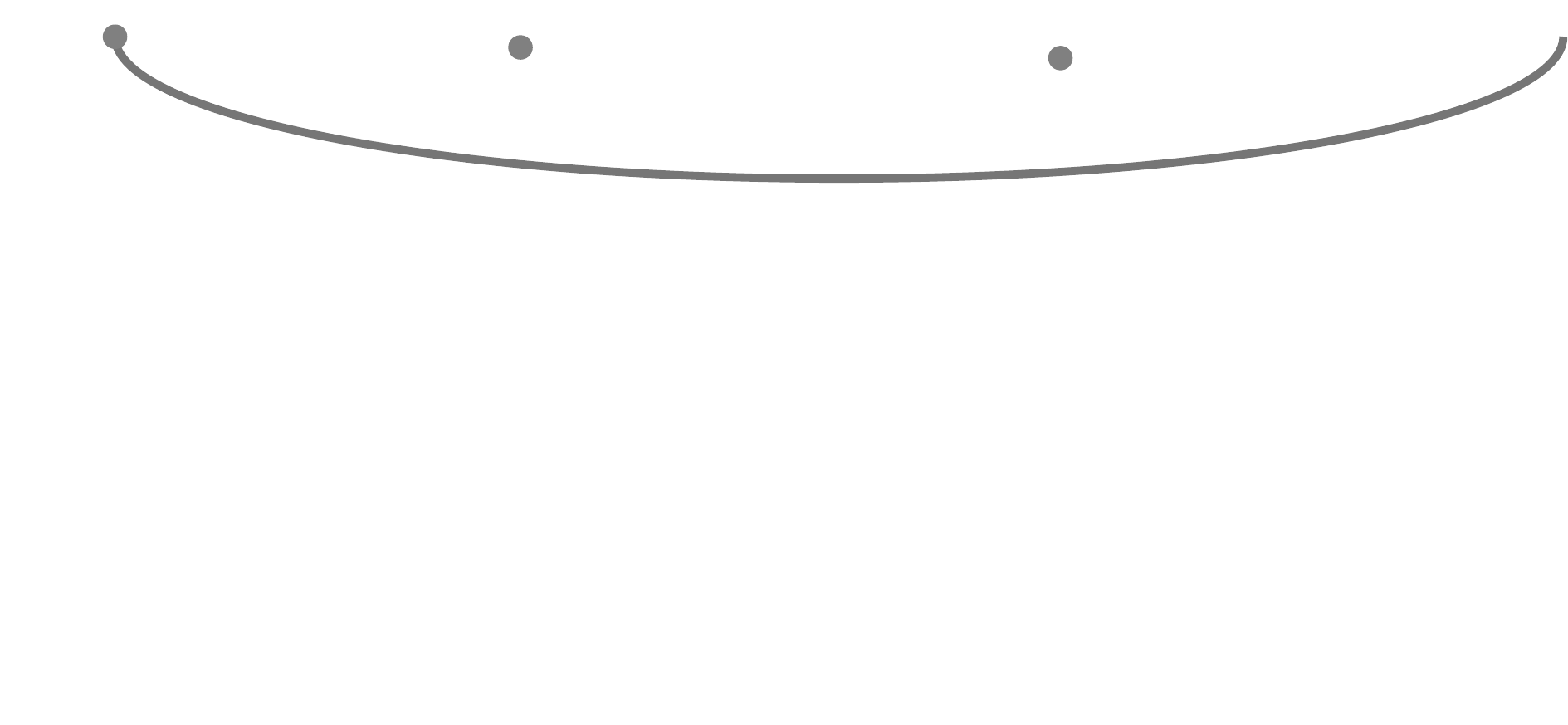
\caption{Generator $\sigma_1$. \label{pi1sigma}}
\end{center}
\end{figure}

\item The braid $B_{r,k}$ for $k \in \lbrace 1 , \ldots , n \rbrace$ corresponds to $\xi_1$ running once around $w_k$ before going back keeping other base point coordinates fixed. The correspondence in terms of standard braid generators can be seen by drawing the movie of this loop in Figure \ref{pi1B}.

\begin{figure}[h!]
\begin{center}
\def\svgwidth{0.5\columnwidth}
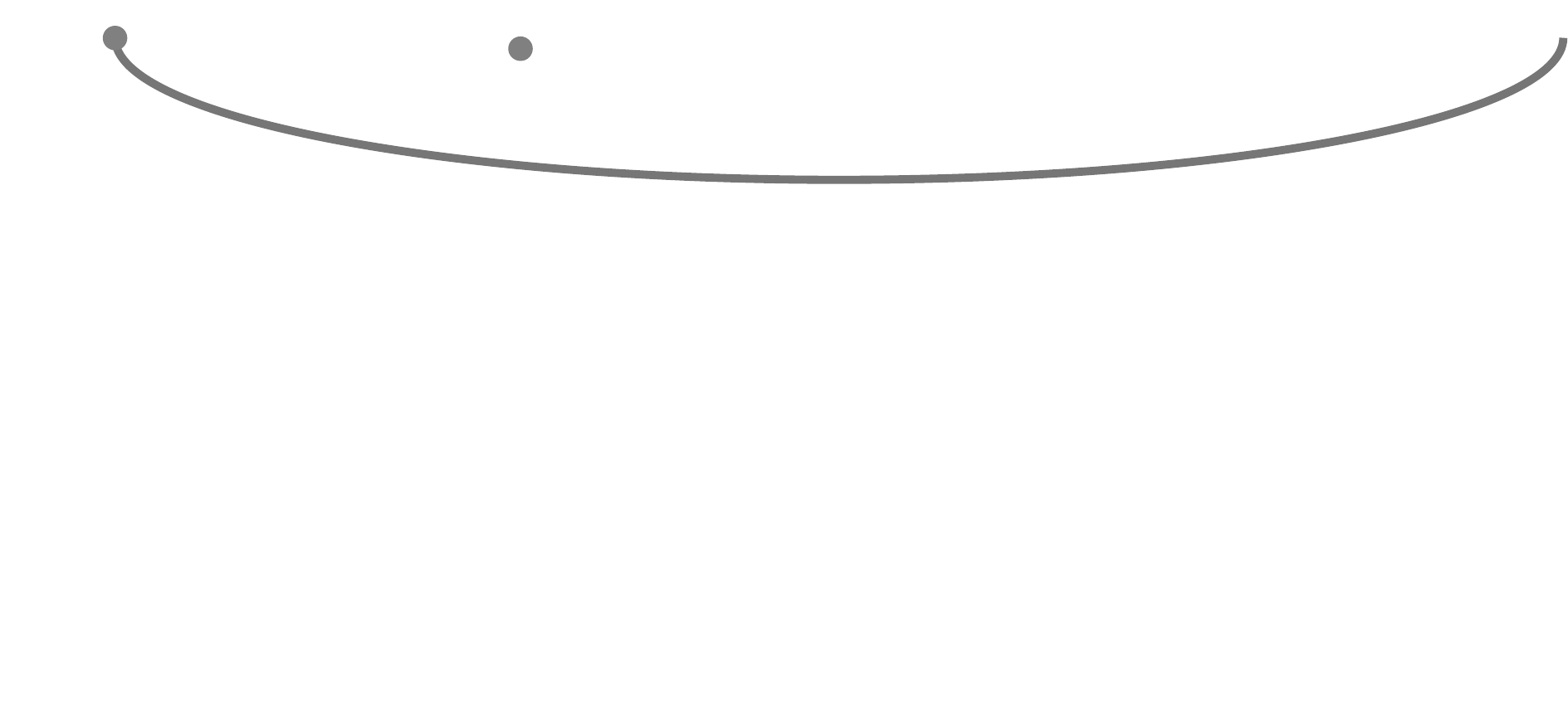
\caption{Generator $B_{r,k}$ \label{pi1B}}
\end{center}
\end{figure}
\end{itemize}
\end{example}

\begin{rmk}
In all what follows and in the above example, braids representing loops in $X_r$ are read from top to bottom. 
\end{rmk}

Using this set up, we define the local system of interest. 

\begin{defn}[Local system $L_r$.]\label{localsystXr}
Let $L_r(w_1,\ldots ,w_n)$ be the local system defined by the following algebra morphism:
\[
\rho_r : \bfct
\BZ\left[ \pi_1(X_r, {\pmb \xi^r}) \right] & \to & \BZ \left[ s_i^{\pm1} , t^{\pm 1}  \right]_{i=1,\ldots ,n}\\
\sigma_i & \mapsto & t \\
B_{r,k} & \mapsto & s_k^2 . \\
\efct
\]
In what follows we will use the notation $q^{\alpha_k}:=s_k$ for all $k = 1 , \ldots, n$. Using this notation, the morphism becomes:
\[
\rho_r : \bfct
\BZ\left[ \pi_1(X_r, {\pmb \xi^r}) \right] & \to & \BZ \left[ q^{\pm \alpha_i} , t^{\pm 1}  \right]_{i=1,\ldots ,n}\\
\sigma_i & \mapsto & t \\
B_{r,k} & \mapsto & q^{2\alpha_k} . \\
\efct
\]
When no confusion is possible we will omit the dependence in $(w_1,\ldots,w_n)$ in the notations to simplify them.
\end{defn}

\begin{rmk}
As it is a one dimensional local system it is abelian in the sense that:
\[
\rho_r ( s_1 s_1) = \rho_r(s_1) \rho_r(s_2) = \rho_r(s_2) \rho_r(s_1)
\]
for $s_1,s_2 \in \pi_1(X_r, {\pmb \xi^r})$. Moreover this local system corresponds to the maximal abelian cover of $X_r$, see Section 2 of \cite{Koh}.
\end{rmk}

We will use homology modules with coefficients in this local system, so that we fix notations from now on.

\begin{defn}\label{defofH}
Let $w_0 = -1$ be the leftmost point on the boundary of the disk, we define the following set:
\[
X_r^-(w_1 , \ldots , w_n) = \left\lbrace \left\lbrace z_1 , \ldots , z_r \right\rbrace \in X_r(w_1 , \ldots , w_n) \text{ s.t. } \exists i, z_i=w_0 \right\rbrace .
\]
Let $r \in \BN$ and $\Laurentmax = \BZ \left[ q^{\pm \alpha_i} , t^{\pm 1}  \right]_{i=1,\ldots ,n}$. We let {\em $\Hlf$} designates the homology of locally finite chains, and we use the following notations for homology with local coefficients in the ring $\Laurentmax$:
\[
\begin{array}{ccc}
\Habs_r  :=  \Hlf _r \left( X_r ; L_r \right) &
\text{ and } & 
\Hrelm_r := \Hlf _r \left( X_r, X^{-}_r ; L_r \right).
\end{array}
\] 
The second one is the homology of the pair $\left( X_r, X^{-}_r  \right)$. See the Appendix, Section \ref{appendix}, for recalls about these homology theories (locally finite/Borel-Moore, with local coefficients).
\end{defn}

\begin{rmk}
All the local system construction of homology classes (see the Appendix) depends on a choice for a lift of base point ${\pmb \xi}$ that we make here, namely let $\widehat{\pmb \xi}$ be a lift of ${\pmb \xi}$ in the cover corresponding to the local system $L_r$. For a different choice $\widehat{\pmb \xi'}$ of lift, all the classes are multiplied by the same (invertible) monomial $\rho_r(\widehat{\pmb \xi} \to \widehat{\pmb \xi'})$ of $\Laurentmax$, namely the local coefficient of a path relating $\widehat{\pmb \xi}$ and $\widehat{\pmb \xi'}$.
\end{rmk}

We recall signature and permutation of braids that will be needed.

\begin{Not}
We will call {\em signature} of a braid, the signature of the permutation induced by the braid, and we will use the following notation for morphisms:
\[
\sign: \Bn \xrightarrow{\perm} \Sk_n \xrightarrow{\sign} \lbrace \pm 1 \rbrace
\]
\end{Not}

\section{Structure of the homology}\label{structureofhomo}

\subsection{Examples of classes}\label{Examplesofclasses}

\begin{defn}\label{Eznrdef}
We define the set of partitions of $n$ in $r$ integers as follows:
\[
\EZnr = \lbrace (k_0, \ldots , k_{n-1}) \in \BN^n \text{ s.t. } \sum k_i=r  \rbrace .
\]
\end{defn}

We now define two families of topological objects indexed by $\EZnr$, that will correspond to classes in $\Hrelm_r$.  

\begin{Not}
We draw topological objects inside the punctured disk, without drawing the boundary of the disk entirely, for an easier reading. The gray color is used to draw the punctured disk. Red arcs are going from a coordinate of the base point ${\pmb \xi}$ of $X_r$ lying in its boundary to a dashed black arc. Dashed black arcs are oriented, from left to right if nothing is specified and if no confusion arises.  Finally, for all the following objects, the red arcs will end up going like in the following picture inside the dashed box, so that all families of red arcs are attached to the base point $\lbrace \xi_1, \ldots , \xi_r \rbrace$ of $X_r$ (here, $r'=r-k_{0}$).

\begin{equation*}
\begin{tikzpicture}[scale = 0.8]
\node (w0) at (-5,1) {};
\node (w1) at (-1,1) {};
\node (w2) at (1,1) {};
\node[gray] at (2.0,1.0) {\ldots};
\node (wn1) at (3,1) {};
\node (wn) at (5,1) {};


\draw[dashed,gray] (-5,0) -- (5,0) -- (5,-3);
\draw[gray] (-5,1.5) -- (-5,-3) -- (6,-3) node[right] {$\partial D_n$};

\node[above,red] at (-3.5,0) {$k_0$};
\node[above,red] at (0,0) {$k_1$};
\node[above,red] at (4,0) {$k_{n-1}$};

\draw[red] (-4,0.3) -- (-4,0) -- (-4.8,-3) node[below] {\small $\xi_r$};
\draw[red] (-3,0.3) -- (-3,0) -- (-4.3,-3) node[below] {\small $\xi_{r'}$};
\draw[red] (-0.5,0.3) -- (-0.5,0) -- (-4,-3);
\draw[red] (0.5,0.3) -- (0.5,0) -- (-3.8,-3);
\draw[red] (3.5,0.3) -- (3.5,0)-- (-3.2,-3) node[below] {\small $\xi_{k_{n-1}}$};
\draw[red] (4.5,0.3) -- (4.5,0)-- (-2.5,-3) node[below] {\small $\xi_1$};

\node[red] at (-3.8,-1) {$\ldots$};
\node[red] at (-1.2,-0.9) {$\ldots$};
\node[red] at (0.25,-0.9) {$\ldots$};
\node[red] at (2.9,-0.5) {$\ldots$};

\node[gray] at (w0)[left=5pt] {$w_0$};
\node[gray] at (w1)[above=5pt] {$w_1$};
\node[gray] at (w2)[above=5pt] {$w_2$};
\node[gray] at (wn1)[above=5pt] {$w_{n-1}$};
\node[gray] at (wn)[above=5pt] {$w_n$};
\foreach \n in {w1,w2,wn1,wn}
  \node at (\n)[gray,circle,fill,inner sep=3pt]{};
\node at (w0)[gray,circle,fill,inner sep=3pt]{};
\end{tikzpicture}
\end{equation*}
\begin{itemize}
\item[{\bf Code sequences}] Let $ {\bf k} = (k_0 , \ldots , k_{n-1}) \in \EZnr$ we define the {\em code sequence} $U_{\bf k}=U(k_0,\ldots,k_{n-1})$ to be the following drawing.
\begin{equation*}
\begin{tikzpicture}
\node (w0) at (-5,0) {};
\node (w1) at (-3,0) {};
\node (w2) at (-1,0) {};
\node[gray] at (0.0,0.0) {\ldots};
\node (wn1) at (1,0) {};
\node (wn) at (3,0) {};
  
\draw[dashed] (w0) -- (w1) node[midway] {$k_0$};
\draw[dashed] (w1) -- (w2) node[midway] {$k_1$};
\draw[dashed] (wn1) -- (wn) node[midway] {$k_{n-1}$};
%

\node[gray] at (w0)[left=5pt] {$w_0$};
\node[gray] at (w1)[above=5pt] {$w_1$};
\node[gray] at (w2)[above=5pt] {$w_2$};
\node[gray] at (wn1)[above=5pt] {$w_{n-1}$};
\node[gray] at (wn)[above=5pt] {$w_n$};
\foreach \n in {w1,w2,wn1,wn}
  \node at (\n)[gray,circle,fill,inner sep=3pt]{};
\node at (w0)[gray,circle,fill,inner sep=3pt]{};

\draw[double,thick,red] (-4,-0.1) -- (-4,-2);
\draw[double,thick,red] (2,-0.1) -- (2,-2);
\draw[double,thick,red] (-2,-0.1) -- (-2,-2);

\draw[dashed,gray] (-5,-2) -- (3,-2);
\draw[dashed,gray] (3,-2) -- (3,-3);

\node[gray,circle,fill,inner sep=0.8pt] at (-4.8,-3) {};
\node[below,gray] at (-4.8,-3) {$\xi_r$};
\node[below=5pt,gray] at (-4.2,-3) {$\ldots$};
\node[gray,circle,fill,inner sep=0.8pt] at (-3.5,-3) {};
\node[below,gray] at (-3.5,-3) {$\xi_1$};

\draw[red] (-4.8,-3) -- (-4,-2);
\draw[red] (-3.5,-3) -- (2,-2);
\node[red] at (-3,-2.4) {$\ldots$};

\draw[gray] (-5,0) -- (-5,1);
\draw[gray] (-5,0) -- (-5,-3);
\draw[gray] (-5,-3) -- (4,-3) node[right] {$\partial D_n$};
\end{tikzpicture}
\end{equation*}
The indexes $k_i$'s stand to illustrate the fact that $k_i$ configuration points are embedded in the corresponding dashed segment, as we explain in what follows. We have attached to an indexed $k_i$ dashed arc a red arc called a {\em $(k_i)$-handle}. It is represented by a little red tube which is a simpler representation used to represent $k_i$ parallel red arcs that are called {\em handles}. We let $\CU= \left\lbrace U(k_0, \ldots , k_{n-1}) \right\rbrace_{{\bf k} \in \EZnr}$. The definition of these objects comes from \cite{Big0}.

\item[\bf Multi-Arcs.] By analogy, for ${\bf k} \in \EZnr$ we define a multi-arc $A'_{\bf k} = A'(k_0, \ldots , k_{n-1} )$ to be the following picture:
\begin{equation*}
\begin{tikzpicture}[decoration={
    markings,
    mark=at position 0.5 with {\arrow{>}}}
    ]
\node (w0) at (-5,0) {};
\node (w1) at (-3,0) {};
\node (w2) at (-1,0) {};
\node[gray] at (0.0,0.0) {\ldots};
\node (wn1) at (1,0) {};
\node (wn) at (3,0) {};

\draw[dashed] (w0) -- (w1) node[midway] (k0) {$k_0$};
\draw[dashed] (w0) to[bend right=20] node[near end] (k1) {$k_1$} (w2);
\draw[dashed] (w0) to[bend right=40] node[pos=0.85] (k2) {$k_{n-2}$} (wn1);
\draw[dashed] (w0) to[bend right=60] node[pos=0.85] (k3) {$k_{n-1}$} (wn);

\node[gray] at (w0)[left=5pt] {$w_0$};
\node[gray] at (w1)[above=5pt] {$w_1$};
\node[gray] at (w2)[above=5pt] {$w_2$};
\node[gray] at (wn1)[above=5pt] {$w_{n-1}$};
\node[gray] at (wn)[above=5pt] {$w_n$};
\foreach \n in {w1,w2,wn1,wn}
  \node at (\n)[gray,circle,fill,inner sep=3pt]{};
\node at (w0)[gray,circle,fill,inner sep=3pt]{};

\draw[double,thick,red] (-4,-0.10) -- (-4,-3);
\draw[double,thick,red] (k1) -- (-2.1,-3);
\draw[double,thick,red] (k2) -- (0.05,-3);
\draw[double,thick,red] (k3) -- (2.1,-3);

\draw[dashed,gray] (-5,-3) -- (3.5,-3);
\draw[dashed,gray] (3.5,-3) -- (3.5,-4);

\node[gray,circle,fill,inner sep=0.8pt] at (-4.8,-4) {};
\node[below,gray] at (-4.8,-4) {$\xi_r$};
\node[below=5pt,gray] at (-4.2,-4) {$\ldots$};
\node[gray,circle,fill,inner sep=0.8pt] at (-3.5,-4) {};
\node[below,gray] at (-3.5,-4) {$\xi_1$};


\draw[red] (-4.8,-4) -- (-4,-3);
\draw[red] (-3.5,-4) -- (2.1,-3);
\node[red] at (-2,-3.4) {$\ldots$};

\draw[gray] (-5,0) -- (-5,1);
\draw[gray] (-5,0) -- (-5,-4);
\draw[gray] (-5,-4) -- (4,-4) node[right] {$\partial D_n$};

\end{tikzpicture}
\end{equation*}
As for code sequences, there is a $(k_i)$-handle arriving to a dashed arc indexed by $k_i$, this will be used to define the associated homology class. We call $\CA'= \left\lbrace A'(k_0, \ldots , k_{n-1}) \right\rbrace_{{\bf k} \in \EZnr}$ the family of all standard multi-arcs. This family of objects is new in the literature. 
\end{itemize}
\end{Not}

We provide a natural way to assign a class in $\Hrelm_r$ to these drawings. Let $X$ be the letter $U$ or $A'$ to treat both cases at the same time. Let ${\bf k} \in \EZnr$ and for all $i = 1 , \ldots ,n$, let:
\[
\phi_i : I_i \to D_n
\]
be the embedding of the dashed black arc number $i$ of $X(k_0 , \ldots , k_{n-1})$ indexed by $k_{i-1}$, where $I_i$ is a unit interval.
Let $\Delta^k$ be the standard (open) $k$ simplex:
\[
\Delta^k = \lbrace 0 < t_1 < \cdots < t_k < 1 \rbrace 
\]
for $k \in \BN$.
For all $i$, we consider the map $\phi^{k_{i-1}}$:
\[
\phi^{k_{i-1}}: \bfct
\Delta^{k_{i-1}} & \to & X_{k_{i-1}} \\
(t_1, \ldots , t_{k_{i-1}} ) & \mapsto & \lbrace \phi_i(t_1) , \ldots, \phi_i(t_{k_{i-1}}) \rbrace
\efct
\]
which is a singular locally finite $(k_{i-1})$-chain and moreover a cycle in $X_{k_{i-1}}$. One can think of the image of the simplex $\Delta^{k_{i-1}}$ to be the space of configurations of $k_{i-1}$ points inside the dashed arc. It provides a locally finite cycle as going to a face of the simplex corresponds to going to a collision between either two configuration points, either a configuration point with a puncture. Namely, points in the boundary of the simplex are removed points of the configuration space $X_r$, these simplices are closed submanifold going to infinity, so that they are locally finite cycles, see the Appendix. There is a cycle associated with each dashed arc, so that by considering the product of maps $(\phi^{k_{0}},\ldots,\phi^{k_{n-1}})\in \Conf_r(D_n)$ which is naturally sent to $X_r$, one  generalizes this fact by associating an $r$-cycle of $X_r$ with each object $X(k_0 , \ldots , k_{n-1})$, see following Remark \ref{chainwithdisjointsupport}. This shows how the union of dashed arcs defines a class in the homology with coefficient in $\BZ$.

To get a class in the local system homology, one has to choose a lift of the chain to the maximal abelian cover $L_r$ associated with the morphism $\rho_r$. The way to do so is using the red handles of $X(k_0 , \ldots , k_{n-1})$ with which is naturally associated a path:
\[
{\bf h}=\lbrace h_1,\ldots,h_r \rbrace: I \to X_r
\]
joining the base point $\pmb{\xi}$ and the $r$-chain assigned to the union of dashed arcs. At the cover level ($\widehat{X_r}$) there is a unique lift $\widehat{{\bf h}}$ of ${\bf h}$ that starts at $\widehat{{\pmb \xi}}$. The lift of $X(k_0, \ldots , k_{n-1})$ passing by $\widehat{\pmb \xi} (1)$ defines a cycle in $\Crelm_r$, and we still call $X(k_0 , \ldots , k_{n-1})$ the associated class in $\Hrelm_r$ as we will only use this class out of the original object. 

\begin{rmk}
If $\phi_i$ and $\phi_i'$ are two parametrizations of the dashed arc $D^{k_{i-1}}$, then $\phi_i$ and $\phi_i'$ are homotopic, so are the associated maps $\phi^{k_{i-1}}$ and ${\phi'}^{k_{i-1}}$. Then, the homology classes associated with $\phi^{k_{i-1}}$ and ${\phi'}^{k_{i-1}}$ are equal and this guarantees that objects are well defined. 
\end{rmk}

\begin{rmk}\label{chainwithdisjointsupport}
If $\phi^{k_1}$ and $\phi^{k_2}$ corresponds to chains with disjoint supports, there exists an associated chain $\lbrace  \phi^{k_1}, \phi^{k_2} \rbrace \in X_{k_1+k_2}$. 
\end{rmk}

\begin{rmk}\label{factorstoConf}
The map $\phi^{k}$ for any $k$ factors through $\Conf_k(D_n)$, namely it is the composition of a map to $\Conf_k(D_n)$ and the quotient to $X_r$, as follows:
\[
\phi^{k}: \Delta^{k} \to \Conf_k(D_n) \to X_k. 
\]
In what follows it will sometimes be more convenient to think about the image of $\phi^{k}$ as a submanifold of $\Conf_k(D_n)$ before considering the quotient by permutations.
\end{rmk}

\begin{example}
By analogy, there is a natural class in $\Hrelm_r$ associated with the following diagram:
\begin{equation*}
\vcenter{\hbox{\begin{tikzpicture}[decoration={
    markings,
    mark=at position 0.5 with {\arrow{>}}}
    ]
\node (w0) at (-5,0) {};
\node (w1) at (-3,0) {};
\node (w2) at (-1,0) {};
\node[gray] at (0.0,0.0) {\ldots};
\node (wn1) at (1,0) {};
\node (wn) at (3,0) {};

\coordinate (x0) at (-4.9,-3) {};
\coordinate (x1) at (-4.85,-3.2) {};
\coordinate (xr1) at (-3.8,-3) {};
\coordinate (xr) at (-3.2,-3) {};

\draw[dashed,postaction={decorate}] (w0) to[bend right=30] node[pos=0.25,below] {} node[pos=0.7] (k0) {} node[above,pos=0.7] {$r-1$} (wn1);
\draw[postaction={decorate}] (w1) to[bend left=20] node[pos=0.25,above] {} node[pos=0.75,below left] (ff) {} (wn1);

\node[gray] at (w0)[left=5pt] {$w_0$};
\node[gray] at (w1)[above=5pt] {$w_1$};
\node[gray] at (w2)[left=5pt] {$w_2$};
\node[gray] at (wn1)[above=5pt] {$w_{n-1}$};
\node[gray] at (wn)[left=5pt] {$w_n$};
\foreach \n in {w1,w2,wn1,wn}
  \node at (\n)[gray,circle,fill,inner sep=3pt]{};
\node at (w0)[gray,circle,fill,inner sep=3pt]{};

\draw[double,red,thick] (k0) -- (-4.5,-2) -- (-4.5,-2.5);
\draw[red] (-4.5,-2.5) -- (x0);
\draw[red] (-4.5,-2.5) -- (xr1);
\draw[red] (ff) to[bend left=40] (3.5,0) to[bend left=10] (xr);

\foreach \n in {x0,xr1,xr}
  \node at (\n)[gray,circle,fill,inner sep=1pt]{};

\node[below,gray] at (x0) {$\xi_r$};
\node[below right,gray] at (x1) {$\ldots$};
\node[below,gray] at (xr1) {$\xi_2$};
\node[below,gray] at (xr) {$\xi_1$};

\draw[gray] (-5,0) -- (-5,1);
\draw[gray] (-5,0) -- (-5,-3);
\draw[gray] (-5,-3) -- (4,-3) node[right] {$\partial D_n$};

\end{tikzpicture}}} \in \Hrelm_r
\end{equation*}
{\em When we draw a plain arc, it corresponds to the image of a $1$-dimensional simplex, and one configuration point embedded,} while a dashed arc indexed by $(r-1)$ corresponds to an $(r-1)$-simplex so to $(r-1)$ configuration points embedded. Red handles are considered, defining a cycle with local coefficients as above.
\end{example}

\subsection{Structural result}

We now study the algebraic structure of $\Hrelm_r$. 
 
\begin{prop}\label{HrelTrick}
For $r\in \BN$, the module $\Hrelm_r$ is a free $\Laurentmax$-module of dimension ${n+r-1}\choose{r}$, generated by the family $\CU$ of code sequences. Moreover, it is the only non vanishing module of the complex $\Hlf_{\bullet} \left( X_r,X_r^-;L_r  \right)$. 
\end{prop}
\begin{proof}
All over the proof, the local ring of coefficients will remain $L_r$ so that we omit it in the notations. Let $\XR_r$ be the set $\lbrace x_1 , \ldots , x_r \rbrace \in X_r$ such that $x_1 , \ldots , x_r$ lie in the segment $\left[ w_0 , w_n \right)$. Set $\XRm_r = \XR_r \cap X_r^-$. We use these simpler spaces to compute the homology, thanks to the following lemma that can be seen as a Bigelow interpretation of the Salvetti retract complex associated with hyperplanes arrangement \cite{Sal}. This method is adapted from Lemma 3.1 of \cite{Big0}. 

\begin{lemma}[Bigelow's trick]\label{bigtrick}
The following map:
\begin{align}
\Hlf_{\bullet} \left( \XR_r,\XRm_r ; L_r  \right) \to \Hlf_{\bullet} \left( X_r,X_r^-;L_r  \right)
\end{align}
induced by inclusion is an isomorphism.
\end{lemma}
\begin{proof}[Proof of Lemma \ref{bigtrick}]
Let $\epsilon >0$ and $A_{\epsilon}$ be the set of $\lbrace x_1 , \ldots , x_r \rbrace \in X_r$ such that $|x_i - x_j| \ge \epsilon$ and $|x_i - w_k| \ge \epsilon$ for all distinct $i,j=1,\ldots , r$ and $k = 1 , \ldots , n$. This family of compact sets yields a basis of compact sets for $X_r$ so that it suffices to show that for all sufficiently small $\epsilon$ the map:
\[
H_{\bullet} \left( \XR_r, \left(\XR_r \setminus A_\epsilon \right) \cup \XRm_r  \right) \to H_{\bullet} \left( X_r, \left(X_r \setminus A_\epsilon \right) \cup X_r^- \right)
\]
induced by inclusion is an isomorphism. This is sufficient by means of the inductive limit over compact sets definition of Borel-Moore homology, see Remark \ref{BMtoLF} in the Appendix.

Let $D_n' \subset D_n$ be a closed $(\epsilon/2)$-neighborhood of the interval $\left[ w_0 , w_n \right)$. Let $X'_r$ be the configuration space of $r$ points in $D_n'$, and $X'^-_r=X'_r \cap X^-_r$ be the ones with a coordinate in $w_0$. We have that the map:
\begin{equation}\label{compressingtrick}
H_{\bullet} \left( X'_r, \left(X'_r \setminus A_\epsilon \right) \cup X'^-_r  \right) \to H_{\bullet} \left( X_r, \left(X_r \setminus A_\epsilon \right) \cup X_r^- \right)
\end{equation}
induced by inclusion is an isomorphism. 
To see this, note that the obvious homotopy shrinking $X_r$ to $X_r'$ is a homotopy of the pairs involved. In other words, points in $X_r \setminus A_{\epsilon}$ corresponding to close points stay in it because the homotopy is a contraction. We will refer to this process - proving that (\ref{compressingtrick}) is an isomorphism - as the {\em compressing trick} later on. 

Let $V$ be the set of $\lbrace x_1 , \ldots , x_r \rbrace \in X_r$ with either $\Re(x_i) = \Re(x_j)$ for some $i,j \in \lbrace 1, \ldots , r \rbrace$ or $\Re(x_i) = w_k$ for some $i \in \lbrace 1, \ldots , r \rbrace$ and $k \in \lbrace 1, \ldots , n \rbrace$. Let $U=X'_r \setminus V$. Note that $V$ is a closed subset contained in $X'_r \setminus A_\epsilon $ which is the interior of  $\left(X'_r \setminus A_\epsilon \right) \cup X'^-_r $. This shows that $V$ satisfies the required hypothesis to perform the excision of the pair, so that the following map:
\[
H_{\bullet} \left( U, \left(U \setminus A_\epsilon \right) \cup \left( X'^-_r \cap U \right)  \right) \to H_{\bullet} \left( X'_r, \left(X'_r \setminus A_\epsilon \right) \cup X'^-_r \right)
\]
induced by inclusion is an isomorphism by the excision theorem. 

Finally there is an obvious {\em vertical line} deformation retraction that sends $U$ to $\XR_r$ taking $\lbrace x_1 , \ldots , x_r \rbrace$ to $\lbrace \Re(x_1) , \ldots , \Re(x_r) \rbrace$. This is again a contraction homotopy so that $U \setminus A_{\epsilon}$ is preserved and $X'_r \cap U$ is sent to $\XRm_r$. This retraction guarantees that the map:
\[
H_{\bullet} \left( \XR_r, \left(\XR_r \setminus A_\epsilon \right) \cup  \XRm_r \right)   \to H_{\bullet} \left( U, \left(U \setminus A_\epsilon \right) \cup \left( X'^-_r \cap U \right) \right)
\]
induced by inclusion is an isomorphism, and concludes the proof of Lemma \ref{bigtrick}.
\end{proof}

To end the proof of the proposition, it remains to compute the complex $\Hlf_{\bullet} \left( \XR_r,\XRm_r ; L_r  \right)$. Let $A^{\BR}_{\epsilon} \in \XR_r$ be the set of configurations $\lbrace x_1 , \ldots , x_r \rbrace$ of $\XR_r$ such that $|x_i - x_j| \ge \epsilon$ and $|x_i - w_k| \ge \epsilon$ where $i,j=1, \ldots , r$ and $k = 1, \ldots, n$. Let $A^{\BR,w_0}_{\epsilon}$ be $A^{\BR}_{\epsilon}$ with the additional condition that $|x_i - w_0| \ge \epsilon$ for $i= 1 , \ldots , r$. We are going to show that for sufficiently small $\epsilon$, the following complex:
\[
H_{\bullet} \left( \XR_r,\left( \XR_r \setminus A^{\BR}_{\epsilon} \right) \cup \XRm_r ; L_r  \right)
\]
is isomorphic to the Borel-Moore one of a disjoint union of open simplexes defined by code sequences. This will end the computation of $\Hlf_{\bullet} \left( \XR_r,\XRm_r ; L_r  \right)$ by definition of Borel-Moore homology. To do so, first we remark that the following spaces are homotopically equivalent:
\begin{align*}
\left( \XR_r \setminus A^{\BR}_{\epsilon} \right) \cup \XRm_r = & \left\lbrace \begin{array}{ll} & |x_i - x_j| < \epsilon \text{ for } i,j=1 , \ldots , r \\ \lbrace x_1 , \ldots , x_r \rbrace \in \XR_r \text{ s.t. } & \text{ or } |x_i - w_k| < \epsilon \text{ for } k=1,\ldots ,n \\ &  \text{ or } x_i=w_0 \end{array} \right\rbrace
\end{align*}
\begin{align*}
\simeq \left\lbrace \begin{array}{ll} & |x_i - x_j| < \epsilon \text{ for } i,j=1 , \ldots , r \\ \lbrace x_1 , \ldots , x_r \rbrace \in \XR_r \text{ s.t. } & \text{ or } |x_i - w_k| < \epsilon \text{ for } k=1,\ldots ,n \\ &  \text{ or } |x_i-w_0|<\epsilon \end{array} \right\rbrace &
= \XR_r \setminus A^{\BR,w_0}_{\epsilon} .
\end{align*}
This shows that the two following complexes are isomorphic:
\[
H_{\bullet} \left( \XR_r,\left( \XR_r \setminus A^{\BR}_{\epsilon} \right) \cup \XRm_r ; L_r  \right) \simeq H_{\bullet} \left( \XR_r, \XR_r \setminus A^{\BR,w_0}_{\epsilon} ; L_r  \right) .
\]
Then one remarks that $\XRm_r$ is closed in $A^{\BR,w_0}_{\epsilon}$ so that we can perform the excision and that the map:
\[
H_{\bullet} \left( \XR_r \setminus \XRm_r, \left( \XR_r \setminus A^{\BR,w_0}_{\epsilon} \right) \setminus \XRm_r ; L_r  \right) \to H_{\bullet} \left( \XR_r, \XR_r \setminus A^{\BR,w_0}_{\epsilon} ; L_r  \right) 
\]
induced by inclusion is an isomorphism. Let $\XR_r(w_0) \subset \XR_r$ be the space of configurations without any coordinate in $w_0$. The space $\XR_r(w_0)$ is exactly the space of configurations of $r$ points in $\left( w_0 , w_n \right)$ such that every coordinate is different from $w_k$ for $k=0 , \ldots , n$. For sufficiently small $\epsilon$, we have shown that the two complexes:
\[
H_{\bullet} \left( \XR_r,\left( \XR_r \setminus A^{\BR}_{\epsilon} \right) \cup \XRm_r ; L_r  \right) \simeq H_{\bullet} \left( \XR_r(w_0), \XR_r(w_0) \setminus A^{\BR,w_0}_{\epsilon} ; L_r  \right)
\]
are isomorphic. Then, as the family of $A^{\BR,w_0}_{\epsilon}$ is a compact set basis for $\XR_r(w_0)$, we end up with the complexes:
\[
\Hlf_{\bullet} \left( \XR_r,\XRm_r ; L_r  \right) \simeq \Hlf_{\bullet} \left( \XR_r(w_0) ; L_r  \right)
\]
being isomorphic. To conclude the computation we take Bigelow's decomposition of $\XR_r(w_0)$ using code sequences as follows and as it was done in \cite{Big0}.
For ${\bf k} \in E^0_{n,r}$, the set of all $\lbrace x_1 , \ldots , x_r \rbrace \in X_r$ such that $x_1 , \ldots , x_r \in \left( w_0 , w_n  \right)$ and:
\[
\sharp \left( \lbrace x_1 , \ldots, x_r \rbrace \cap \left( w_i , w_{i+1}  \right) \right) = k_i
\]
for $i = 0 , \ldots , n-1$ is exactly $U(k_0,\ldots , k_{n-1})$, and one remarks that:
\[
\XR_r(w_0) = \bigsqcup_{{\bf k} \in E^0_{n,r}} U_{\bf k}. 
\]
From this disjoint union of open simplexes, we deduce that $\Hlf_{r} \left( \XR_r(w_0) ; L_r  \right)$ is the direct sum of $\sharp E^0_{n,r} =$ $ {n+r-1}\choose{r}$ copies of $\Laurentmax$ while all other $\Hlf_{k} \left( \XR_r(w_0) ; L_r  \right)$ for $k \neq r$ vanishes. The complex $\Hlf_{\bullet} \left( \XR_r,\XRm_r ; L_r  \right)$ has the same decomposition, which concludes the proof. 

\end{proof}

Bigelow's trick was initially used to show the following. 
\begin{prop}[Lemma 3.1 \cite{Big0}]\label{relativetoabolute}
The morphism:
\[
\Hlf_{\bullet} \left( \XR_r(w_0) ; L_r  \right) \to \Hlf_{\bullet} \left( X_r(w_0) ; L_r  \right)
\]
induced by inclusion is an isomorphism of complexes. 
\end{prop}

From this and from the proof of Proposition \ref{HrelTrick}, one gets the following corollary.

\begin{coro}\label{HrelStruc}
\begin{itemize}
\item The morphism: $\Hlf_{\bullet} \left( X_r(w_0) ; L_r  \right) \to \Hlf_{\bullet} \left( X_r,X^-_r ; L_r  \right)$ induced by inclusion is an isomorphism.  
\item The family $\CU = \left(U_{\bf k} \right)_{{\bf k } \in E^0_{n,r}}$ yields a basis of $\Hrelm_r$ as an $\Laurentmax$-module. 
\end{itemize}
\end{coro}

We conclude this part with two remarks about the proof of Proposition \ref{HrelTrick}.

\begin{rmk}
\begin{itemize}
\item The proof of Proposition \ref{HrelTrick} is constructive in the sense that it provides a process to express homology classes in the $\CU$ basis. This will be used in next sections. 
\item All along the proof of Proposition \ref{HrelTrick}, the local system does not change, no morphism of the latter is needed. The proof relies only on topological operations such as excisions and homotopy equivalences. In some sense the proof is rigid regarding the local ring of coefficients, and should be adaptable with another one.
\end{itemize}
\end{rmk}

\section{Computation rules}\label{computationrules}


\subsection{Homology techniques}

\begin{rmk}[Handle rule]\label{handlerule}
Let $B$ be a singular locally finite $r$-cycle of $C_r(X_r,X^-_r,\BZ)$. We've seen a process to choose a lift of $B$ to the homology with local coefficients in $L_r$, using a handle which is a path joining ${\pmb \xi}$ to $x \in B$. Let $\alpha$ and $\beta$ be two different paths joining ${\pmb \xi}$ and $B$. Let $\widehat{B}^\alpha$ and $\widehat{B}^\beta$ be the lifts of $B$ chosen using $\alpha$ and $\beta$ respectively. By the {\em handle rule} we have the following relation in $\Hrelm_r$:
\begin{equation*}
\widehat{B}^\alpha = \rho_r(\alpha \beta^{-1}) \widehat{B}^\beta
\end{equation*}
where $\rho_r$ is the representation of $\pi_1(X_r, {\pmb \xi^r})$ used to construct $L_r$ in Definition \ref{localsystXr}. This expresses how the local system coordinate of a homological class is translated after a change of handle. 
\end{rmk}

\begin{rmk}\label{example0}
One must be careful while permuting red arcs of a multi-arcs or code sequence-like class (see Section \ref{Examplesofclasses}). Indeed, as the parametrization of the underlying simplex is ruled by the order of relating arcs to the base point: such a permutation of red handles multiplies the class by its signature. We show one example. 

We have the following equality between these two classes in $\Hrelm_2$: 
\[
\left(\vcenter{\hbox{ \begin{tikzpicture}[decoration={
    markings,
    mark=at position 0.5 with {\arrow{>}}}
    ]
\coordinate (w0) at (-1,0) {};
\coordinate (w1) at (1,0) {};
\coordinate (x0) at (-1,-1) {};
\coordinate (x1) at (1,-1) {};

\draw[postaction={decorate}] (w0) -- node[above,pos=0.3] (k0) {} (w1);
\draw[postaction={decorate}] (w0) to[bend left=30] node[pos=0.8] (k1) {} (w1);

\draw[red] (k0) -- node[midway] (a) {} (k0|-x0);
\draw[red] (k1) arc (180:90:0.4) arc (90:-90:0.4) -- (k1|-x1);

\node[right,red] at (a) {$\alpha$};

\node[gray] at (w0)[left=5pt] {$w_i$};
\node[gray] at (w1)[right=5pt] {$w_j$};
\foreach \n in {w0,w1}
  \node at (\n)[gray,circle,fill,inner sep=3pt]{};

\end{tikzpicture} }}\right)
= \sign(\alpha \beta^{-1}) \rho_r(\alpha \beta^{-1}) 
\left(\vcenter{\hbox{ \begin{tikzpicture}[decoration={
    markings,
    mark=at position 0.5 with {\arrow{>}}}
    ]
\coordinate (w0) at (-1,0) {};
\coordinate (w1) at (1,0) {};
\coordinate (x0) at (-1,-1) {};
\coordinate (x1) at (1,-1) {};

\draw[postaction={decorate}] (w0) -- node[above,pos=0.75] (k1) {} (w1);
\draw[postaction={decorate}] (w0) to[bend left=30] node[above,pos=0.3] (k0) {} (w1);

\draw[red] (k0) -- node[pos=0.66] (b) {} (k0|-x0);
\draw[red] (k1) -- (k1|-x1);

\node[right,red] at (b) {$\beta$};

\node[gray] at (w0)[left=5pt] {$w_i$};
\node[gray] at (w1)[right=5pt] {$w_j$};
\foreach \n in {w0,w1}
  \node at (\n)[gray,circle,fill,inner sep=3pt]{};

\end{tikzpicture} }}\right) 
\]
with $\rho_r(\alpha \beta^{-1})=t^{-1}q^{-2\alpha_j}$ and $\sign(\alpha \beta^{-1})=-1$. Indeed, we suppose that the drawing is empty everywhere outside the parentheses besides the red handles $\alpha$ and $\beta$  that join the base point ${\pmb \xi}$ in the boundary. We suppose also that $\alpha$ and $\beta$ follow exactly same paths outside parentheses. This allows us to draw the braid $\alpha \beta^{-1}$ in Figure \ref{exemple0}.

\begin{figure}[h!]
\begin{center}
\def\svgwidth{0.4\columnwidth}
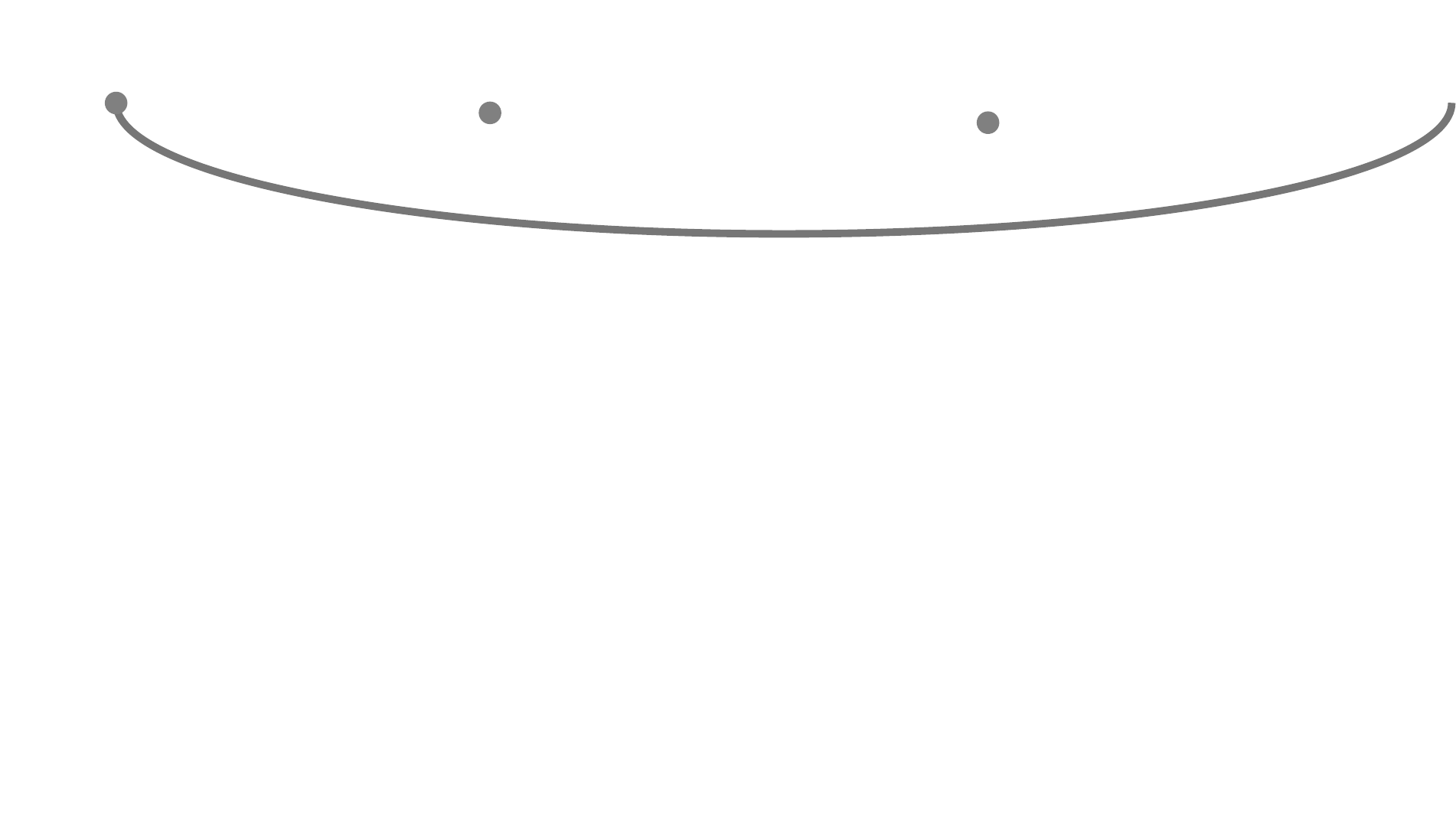
\caption{The braid $\alpha \beta^{-1}$ \label{exemple0}}
\end{center}
\end{figure}

The figure continues outside of the box, but as the path to the base point is the same for  $\alpha$ and $\beta$ the path upper box is the inverse of the lower one. As the local system is abelian, the out box parts of the braid won't contribute to $\rho_r(\alpha \beta^{-1})$ (nor to $\sign(\alpha \beta^{-1})$). Considering the definition of $\rho_r$ (Definition \ref{localsystXr}), one sees that the local system coordinate of the above path is $t^{-1}q^{-2\alpha_j}$ so is the one of $\alpha \beta^{-1}$.

One should also notice that the signature of the braid is $(-1)^{e(t)}$, where $e(t)$ is the exponent of $t$ in the local coefficient. This last remark is equivalent to the following very useful remark for what follows:
\[
\sign(\alpha \beta^{-1}) \rho_r(\alpha \beta^{-1}) = \rho_r(\alpha \beta^{-1})_{|t=\mt}
\]
where $\mt := -t$. This will be extensively use in all following computations. 
 
\end{rmk}

We reformulate the compressing trick used in the proof of Proposition \ref{HrelTrick} in a local version.

\begin{prop}[Compressing trick]\label{compress2}
Let $D_p \subset D_n$ (and $D_p^0 \subset D_n$ respectively) be a topological punctured disk with punctures $w_{n_1} ,\ldots ,w_{n_p}$ and $n_i \in \lbrace 1,\ldots ,n \rbrace$ for $i=1,\ldots ,p$ (resp. $D_p^0$ contains also $w_0$). Let $X_r(D_p)$ (resp. $X_r(D_p^0)$) be the space of configuration of $r$ points inside $D_p$ (resp $D_p^0$). Let $D_p'$ (resp ${D'}_p^0$) be an $\epsilon$-neighborhood of a segment in $D_p$ (resp $D_p^0$) joining the points $w_{n_1} , \ldots , w_{n_p}$ (resp. having an end in $w_0$) and contained in the real axis, with $\epsilon$ small enough to have $D_p' \subset D_p$. Then the morphisms:
\[
H_{\bullet} \left( X_r(D'_p) \right) \to H_{\bullet} \left( X_r(D_p) \right)
\]
and
\[
H_{\bullet} \left( X_r({D'}_p^0) , X_r({D'}_p^0)^-  \right) \to H_{\bullet} \left( X_r(D_p^0),  X_r(D_p^0)^- \right)
\]
induced by inclusion are isomorphisms (the module $X_r({D'}_p^0)^-$ stands for  configurations with one point in $w_0$). All the homology modules are Borel-Moore ones (or equivalently of locally finite chains) and considered with coefficients in the local system $L_r$ restricted to the space of interest, so that we omit it in the notations.
\end{prop}
\begin{proof}
The proof is exactly the same as the one of (\ref{compressingtrick}) being an isomorphism, in the proof of Lemma \ref{bigtrick}, but performed inside $D_p$ (resp. $D_p^0$).
\end{proof}

\begin{prop}[Combing process.]\label{combing}
Let $M=M(D_1^{k_1} , \ldots , D_d^{k_d})$ be a class associated with a drawing made of disjoint dashed arcs $D_1$ indexed by $k_1$, $D_2$ indexed by $k_2$ and so on, all of them related to the base point ${\pmb \xi}$ by red handles. Suppose the $(k_1)$-handle reaches $D^{k_1}_1$ in a point $x$. Let $D^{k_1}_1 = {D}_1^- \cup_{x} {D}_1^+$ be a subdivision of arc $D_1$ following its orientation. Let $D$ be an arc joining $x$ to some $w \in \lbrace w_0 , \ldots ,w_n \rbrace$, and such that $D$ is disjoint from all the $D^{k_1}_i$'s. Let $l \in \lbrace 0 , \ldots , k_1 \rbrace$, and $M^l$ be the following class obtained from $M$ by modifying its drawing as follows:
\[
M^l = M\left( \left(D_1^- \star D \right)^{l} ,\left(   D^{-1} \star D_1^+ \right)^{k_1-l} ,D_2^{k_2}, \ldots , D_d^{k_d} \right)
\]
so that the initial arc $D_1$ is divided into two, one indexed by $l$ the other one by $k_1-l$. Handles are preserved from $M$, except for the $(k_1)$-handle that is divided into two: one $(l)$-handle joining $\left(D_1^- \star D \right)^{l}$ in $x$ and one $(k_1-l)$-handle joining $\left( \left(D_1^+ \star D \right)^{-1} \right)^{k_1-l}$ in $x$. There is the following homological relation:
\[
M= \sum_{l=0}^{k_1} M^l .
\]
See Examples \ref{breakingplain} and \ref{breakingdashed} of such combing, which should help the understanding of the statement. 
\end{prop}
\begin{proof}
Suppose the class $M=M(D_1^{k_1})$ is made of only one dashed arc. Let $\phi^{k_1}$:
\[
\phi^{k_1} : \bfct
\Delta^{k_1} & \to & X_{k_1} \\
(t_1 , \ldots , t_{k_1} ) & \mapsto & \lbrace \phi(t_i), i=1,\ldots , k_1 \rbrace
\efct
\]
be the chain naturally associated with the indexed $k_1$ dashed arc of the considered class, where $\phi$ is a parametrization of $D^{k_1}$. We subdivide the simplex: for $l \in \lbrace 0 , \ldots , k_1 \rbrace$ let $ \Delta^{k_1,l}$ be defined as follows:
\[
\Delta^{k_1,l} = \lbrace (t_1 , \ldots , t_{k_1}) \in \Delta^k \text{ s.t. } t_{l} < \phi^{-1}(x) < t_{l+1} \rbrace
\]
which image by $\phi^{k_1}$ corresponds to configurations for which the handle together with $D$ arrive between images of $t_l$ and $t_{l+1}$. Let $\phi^{k_1,l}$ be the restriction of $\phi^{k_1}$ to $\Delta^{k_1,l}$. Let:
\[
h_t : I \to D_n
\]
be an isotopy (rel. endpoints) sending the arc $D^{k_1}$ to the right one of Figure \ref{ht1} (arcs oriented from left to right). 

\begin{figure}[h!]
\begin{center}
\def\svgwidth{0.7\columnwidth}
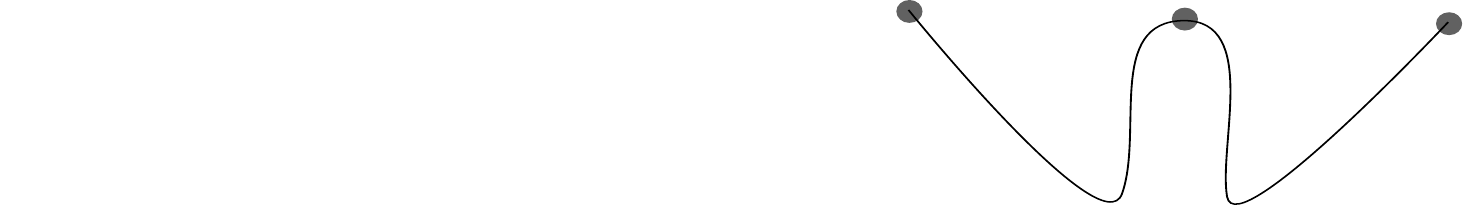
\caption{The isotopy $h_t$. \label{ht1}}
\end{center}
\end{figure}

For all $t$ in $I$, let $\phi_t^{k_1}$ be the following map:
\[
\phi_t^{k_1} : \bfct
\Delta^{k_1} & \to & X_{k_1} \\
(t_1 , \ldots , t_{k_1} ) & \mapsto & \lbrace h_t \circ \phi(t_i), i=1,\ldots , k_1 \rbrace
\efct
\]
and let $\phi_t^{k_1,l}$ be the following map:
\[
\phi_t^{k_1,l} : \bfct
\Delta^{k_1,l} & \to & X_{k_1} \\
(t_1 , \ldots , t_{k_1} ) & \mapsto & \lbrace h_t \circ \phi(t_i), i=1,\ldots , k_1 \rbrace ,
\efct
\]
namely the restriction to $\Delta^{k_1,l}$.
Let $\left[ \phi_t^{k_1} \right]$ and $\left[ \phi_t^{k_1,l} \right]$ be the corresponding chains. One remarks that $\phi_0^{k_1,l} = \phi^{k_1,l}$ and $\phi_0^{k_1} = \phi^{k_1}$. In terms of chains we have the following equality holding for all $t \in I$:
\[
\left[ \phi_t^{k_1} \right] = \sum_l  \left[ \phi_t^{k_1,l} \right],
\]
this is because $\lbrace \Delta^{k_1,l}, l = 0 , \ldots , k_1 \rbrace$ is a subdivision of $\Delta^{k_1}$. 
For $t=0$ this chain is $\left[  \phi^{k_1} \right]$ while for $t=1$, terms of the sum are Borel-Moore cycles homologous to $M^l$. It shows that $\left[  \phi^{k_1} \right]$ and $\sum_l M^l$ are homotopic so that the relation $M= \sum_{l=0}^{k_1} M^l$ holds in $\Hlf_r(X_r,X_r^-, \BZ)$. Then the lifting process is unchanged as handles are preserved. It proves the proposition for a class composed by one dashed arc, and it generalizes to all classes with disjoint dashed arcs, as only the first component is involved in the combing.

%

\end{proof}

Two examples of combings that will be used many times. 

\begin{example}[Breaking a plain arc]\label{breakingplain}
By considering a path joining the red handle to $w_i$ one can check the following relations between homology class (all arcs oriented from left to right):
\begin{align*}
\left(\vcenter{\hbox{
\begin{tikzpicture}[decoration={
    markings,
    mark=at position 0.5 with {\arrow{>}}}
    ]
\node (w0) at (-1,0) {};
\node[gray] at (-0.2,0) {$\ldots$}; 
\node (w1) at (0.5,0) {};
\node (wn) at (1.5,0) {};
\coordinate (a) at (-2,-1);
\draw (w0) to[bend right=30pt] node[pos=0.25,below] {\tiny $P_1^-$} node[pos=0.75,below] {\tiny $P_1^+$} coordinate[midway] (solo) (wn);
\node[gray] at (w0)[left=5pt] {$w_0$};
\node[gray] at (wn)[above=5pt] {$w_{i+1}$};
\node[gray] at (w1)[above=5pt] {$w_{i}$};
\foreach \n in {w0,wn,w1}
  \node at (\n)[gray,circle,fill,inner sep=3pt]{};
\draw[red] (solo) -- (solo|-a);
\draw[gray,thick] (-1,0.3) -- (w0) -- (-1,-1.5);
\end{tikzpicture}
}} \right)& = 
\left(\vcenter{\hbox{
\begin{tikzpicture}[decoration={
    markings,
    mark=at position 0.5 with {\arrow{>}}}
    ]
\node (w0) at (-1,0) {};
\node[gray] at (-0.2,0) {$\ldots$}; 
\node (w1) at (0.5,0) {};
\node (wn) at (1.5,0) {};
\coordinate (a) at (-2,-1);
\draw (w0) to[bend right=30pt] node[pos=0.4,below] {\tiny $P_1^-$} (solo) to node[midway,right] {\tiny $P$} (w1);
\node[gray] at (w0)[left=5pt] {$w_0$};
\node[gray] at (wn)[above=5pt] {$w_{i+1}$};
\node[gray] at (w1)[above=5pt] {$w_{i}$};
\foreach \n in {w0,wn,w1}
  \node at (\n)[gray,circle,fill,inner sep=3pt]{};
\draw[red] (solo) -- (solo|-a);
\draw[gray,thick] (-1,0.3) -- (w0) -- (-1,-1.5);
\end{tikzpicture}
}} \right)
+\left(\vcenter{\hbox{
\begin{tikzpicture}[decoration={
    markings,
    mark=at position 0.5 with {\arrow{>}}}
    ]
\node (w0) at (-1,0) {};
\node[gray] at (-0.2,0) {$\ldots$}; 
\node (w1) at (0.5,0) {};
\node (wn) at (1.5,0) {};
\coordinate (a) at (-2,-1);
\draw (w1) to node[pos=0.35,left] {\tiny $P^{-1}$} (solo) to[bend right=30pt] node[pos=0.6,below] {\tiny $P_1^{+}$} (wn);
\node[gray] at (w0)[left=5pt] {$w_0$};
\node[gray] at (wn)[above=5pt] {$w_{i+1}$};
\node[gray] at (w1)[above=5pt] {$w_{i}$};
\foreach \n in {w0,wn,w1}
  \node at (\n)[gray,circle,fill,inner sep=3pt]{};
\draw[red] (solo) -- (solo|-a);
\draw[gray,thick] (-1,0.3) -- (w0) -- (-1,-1.5);
\end{tikzpicture}
}} \right) \\
& = \left(\vcenter{\hbox{
\begin{tikzpicture}[decoration={
    markings,
    mark=at position 0.5 with {\arrow{>}}}
    ]
\node (w0) at (-1,0) {};
\node[gray] at (-0.2,0) {$\ldots$}; 
\node (w1) at (0.5,0) {};
\node (wn) at (1.5,0) {};
\coordinate (a) at (-2,-1);
\draw (w0) to[bend right=30pt] node[above,pos=0.6] (solo) {} (w1);
\node[gray] at (w0)[left=5pt] {$w_0$};
\node[gray] at (wn)[above=5pt] {$w_{i+1}$};
\node[gray] at (w1)[above=5pt] {$w_{i}$};
\foreach \n in {w0,wn,w1}
  \node at (\n)[gray,circle,fill,inner sep=3pt]{};
\draw[red] (solo) -- (solo|-a);
\draw[gray,thick] (-1,0.3) -- (w0) -- (-1,-1.5);
\end{tikzpicture}
}} \right)
+\left(\vcenter{\hbox{
\begin{tikzpicture}[decoration={
    markings,
    mark=at position 0.5 with {\arrow{>}}}
    ]
\node (w0) at (-1,0) {};
\node[gray] at (-0.2,0) {$\ldots$}; 
\node (w1) at (0.5,0) {};
\node (wn) at (1.5,0) {};
\coordinate (a) at (-2,-1);
\draw (w1) to node[above,midway] (solo) {} (wn);
\node[gray] at (w0)[left=5pt] {$w_0$};
\node[gray] at (wn)[above=5pt] {$w_{i+1}$};
\node[gray] at (w1)[above=5pt] {$w_{i}$};
\foreach \n in {w0,wn,w1}
  \node at (\n)[gray,circle,fill,inner sep=3pt]{};
\draw[red] (solo) -- (solo|-a);
\draw[gray,thick] (-1,0.3) -- (w0) -- (-1,-1.5);
\end{tikzpicture}
}} \right)
\end{align*}
where drawings are the same outside boxes. To obtain the second line we have applied small isotopies not changing the homology class. One remarks that before the small isotopies being applied, handles are unchanged. 
\end{example}

\begin{example}[Breaking a dashed arc]\label{breakingdashed}
By considering a path joining the red handle to $w_i$ one can check the following relations between homology class:
\begin{equation*}
\left(\vcenter{\hbox{
\begin{tikzpicture}[decoration={
    markings,
    mark=at position 0.5 with {\arrow{>}}}
    ]
\node (w0) at (-1.5,0) {};
\node (wj) at (1.5,0) {};
\node (wi) at (0.5,0) {};
\coordinate (a) at (-1,-1.5);

\draw[dashed] (w0) to[bend right=40] node[above=0.1pt,pos=0.5] (k0) {$k$} (wj);

\node[gray] at (w0)[left=5pt] {$w_0$};
\node[gray] at (wj)[right=5pt] {$w_j$};
\node[gray] at (wi)[left=5pt] {$w_i$};
\foreach \n in {w0,wi,wj}
  \node at (\n)[gray,circle,fill,inner sep=3pt]{};

\draw[double,red,thick] (k0) -- (k0|-a);
\draw[gray,thick] (-1.5,0.3) -- (w0) -- (-1.5,-1.5);

\end{tikzpicture}
}}\right)
=
\sum_{l=0}^k
\left(\vcenter{\hbox{
\begin{tikzpicture}[decoration={
    markings,
    mark=at position 0.5 with {\arrow{>}}}
    ]
\node (w0) at (-1.5,0) {};
\node (wj) at (1.5,0) {};
\node (wi) at (0.5,0) {};
\coordinate (a) at (-1,-1.5);

\draw[dashed] (w0) to[bend right=40] node[above=0.1pt,pos=0.5] (k0) {$l$} (wi);
\draw[dashed] (wi) to node[pos=0.5,above] (k1) {\small $k-l$} (wj);

\node[gray] at (w0)[left=5pt] {$w_0$};
\node[gray] at (wj)[above right] {$w_{i+1}$};
\node[gray] at (wi)[above left] {$w_i$};
\foreach \n in {w0,wi,wj}
  \node at (\n)[gray,circle,fill,inner sep=3pt]{};

\draw[double,red,thick] (k0) -- (k0|-a);
\draw[double,red,thick] (k1) -- (k1|-a);
\draw[gray,thick] (-1.5,0.3) -- (w0) -- (-1.5,-1.5);

\end{tikzpicture}
}} \right).
\end{equation*}
where drawings are the same outside boxes. 
\end{example}

\subsection{Diagram rules}

We use homology techniques presented in the previous section to set diagram rules between homology classes. These rules expressed with coefficients in the ring $\Laurentmax$ involve quantum numbers that we introduce now.

\begin{defn}\label{quantumt}
Let $i$ be a positive integer. We define the following elements of $\BZ \left[ t^{\pm 1} \right] \subset \Laurentmax$.
\begin{equation*}
(i)_t := (1+t+ \cdots + t^{i-1}) = \frac{1-t^i}{1-t} , \text{  } (k)_t! := \prod_{i=1}^k (i)_t, \text{ and }  {{k}\choose{l}}_t := \frac{(k)_t!}{(k-l)_t! (l)_t!} 
\end{equation*}
\end{defn}

\begin{Not}
In what follows we will use extensively the variable $-t$ instead of $t$, so that we fix a notation for it:
\[
\boxed{ \mt:=-t }
\]
\end{Not}

\begin{Not}
Since we work with Borel-Moore homology with local coefficients, one can think of it as the following complex:
\[
H_{\bullet} \left( X_r, \left(X_r \setminus A_\epsilon \right) \cup X_r^- ; L_r \right)
\]
for a small $\epsilon$, with $A_{\epsilon}$ defined as in the proof of Proposition \ref{HrelTrick}. A dashed arc indexed by $k>1$ corresponds to an embedding of $k$ points (a $k$-simplex) inside the arc. 

As the order of points does not matter - working in $X_r$ - one can think of the dashed arc as in Figure \ref{dashedmodel}.

\begin{figure}[h!]
\begin{center}
\def\svgwidth{0.7\columnwidth}
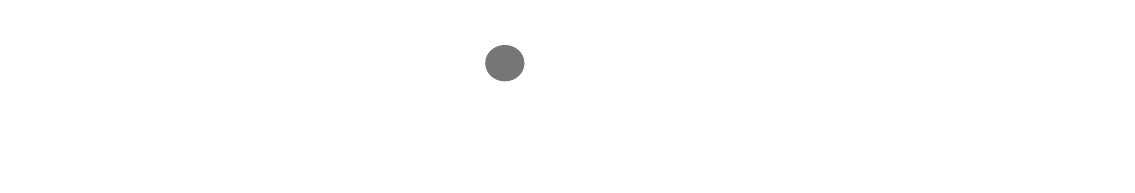
\caption{Dashed arc model. \label{dashedmodel}}
\end{center}
\end{figure}

On the left side we see a standard piece of an element of $\CU$ and on the right side, one can think of this element as the image of one point by the following embedding:
\[
\Delta^k \to \left[ w_i, w_j \right]
\]
where $M_i$ is the image of $t_i$, the $i^{th}$ coordinate of $\Delta^k$. The $M_i$'s are represented by gray boxes to keep in mind that we work relatively to $X_r \setminus A_{\epsilon}$. Every point is lifted to the maximal abelian cover ($\widehat{X_r}$) using the red handle reaching it. A first diffeomorphism of $D_n$ has been applied, allowing one to imagine this picture with $w_i$ facing $w_j$. This diffeomorphism does not change homology classes. 
\end{Not}

The above picture will be useful to deal with the proof of the following crucial homological relations showing a first appearance of quantum numbers.

\begin{lemma}\label{plaintodashed}
Let $k>1$ be an integer. The following equalities hold in $\Hrelm_{\bullet}$:
\begin{equation*}
\left(\vcenter{\hbox{
\begin{tikzpicture}[decoration={
    markings,
    mark=at position 0.5 with {\arrow{>}}}
    ]
\node (w0) at (-1,0) {};
\node (w1) at (1,0) {};
\coordinate (a) at (-1,-1);

\draw[dashed] (w0) -- (w1) node[pos=0.66] (k0) {$k$};
\draw[postaction={decorate}] (w0) to[bend left=40] node[pos=0.3,above] (k1) {} (w1);

\node[gray] at (w0)[left=5pt] {$w_i$};
\node[gray] at (w1)[right=5pt] {$w_j$};
\foreach \n in {w0,w1}
  \node at (\n)[gray,circle,fill,inner sep=3pt]{};

\draw[double,red,thick] (k0) -- (k0|-a);
\draw[red] (k1) -- (k1|-a);

\end{tikzpicture}
}}\right)
=
(k+1)_{\mt} \left(\vcenter{\hbox{
\begin{tikzpicture}[decoration={
    markings,
    mark=at position 0.5 with {\arrow{>}}}
    ]
\node (w0) at (-1.5,0) {};
\node (w1) at (1.5,0) {};
\coordinate (a) at (-1,-1);

\draw[dashed] (w0) -- (w1) node[pos=0.5,above] (k0) {$(k+1)$};

\node[gray] at (w0)[left=5pt] {$w_i$};
\node[gray] at (w1)[right=5pt] {$w_j$};
\foreach \n in {w0,w1}
  \node at (\n)[gray,circle,fill,inner sep=3pt]{};

\draw[double,red,thick] (k0) -- (k0|-a);

\end{tikzpicture}
}}\right)
\end{equation*}
\begin{equation*}
\left(\vcenter{\hbox{
\begin{tikzpicture}[decoration={
    markings,
    mark=at position 0.5 with {\arrow{>}}}
    ]
\node (w0) at (-1,0) {};
\node (w1) at (1,0) {};
\coordinate (a) at (-1,-1);

\draw[dashed] (w0) -- (w1) node[pos=0.33] (k0) {$k$};
\draw[postaction={decorate}] (w0) to[bend left=40] node[pos=0.66,above] (k1) {} (w1);

\node[gray] at (w0)[left=5pt] {$w_i$};
\node[gray] at (w1)[right=5pt] {$w_j$};
\foreach \n in {w0,w1}
  \node at (\n)[gray,circle,fill,inner sep=3pt]{};

\draw[double,red,thick] (k0) -- (k0|-a);
\draw[red] (k1) -- (k1|-a);

\end{tikzpicture}
}}\right)
=
(k+1)_{\mt^{-1}} \left(\vcenter{\hbox{
\begin{tikzpicture}[decoration={
    markings,
    mark=at position 0.5 with {\arrow{>}}}
    ]
\node (w0) at (-1.5,0) {};
\node (w1) at (1.5,0) {};
\coordinate (a) at (-1,-1);

\draw[dashed] (w0) -- (w1) node[pos=0.5,above] (k0) {$(k+1)$};

\node[gray] at (w0)[left=5pt] {$w_i$};
\node[gray] at (w1)[right=5pt] {$w_j$};
\foreach \n in {w0,w1}
  \node at (\n)[gray,circle,fill,inner sep=3pt]{};

\draw[double,red,thick] (k0) -- (k0|-a);

\end{tikzpicture}
}}\right)
\end{equation*}
\begin{equation*}
\left(\vcenter{\hbox{
\begin{tikzpicture}[decoration={
    markings,
    mark=at position 0.5 with {\arrow{>}}}
    ]
\node (w0) at (-1,0) {};
\node (w1) at (1,0) {};
\coordinate (a) at (-1,-1);

\draw[dashed] (w0) -- (w1) node[above,pos=0.33] (k0) {$k$};
\draw[postaction={decorate}] (w0) to[bend right=40] node[pos=0.66,above] (k1) {} (w1);

\node[gray] at (w0)[left=5pt] {$w_i$};
\node[gray] at (w1)[right=5pt] {$w_j$};
\foreach \n in {w0,w1}
  \node at (\n)[gray,circle,fill,inner sep=3pt]{};

\draw[double,red,thick] (k0) -- (k0|-a);
\draw[red] (k1) -- (k1|-a);

\end{tikzpicture}
}}\right)
=
(k+1)_{\mt} \left(\vcenter{\hbox{
\begin{tikzpicture}[decoration={
    markings,
    mark=at position 0.5 with {\arrow{>}}}
    ]
\node (w0) at (-1.5,0) {};
\node (w1) at (1.5,0) {};
\coordinate (a) at (-1,-1);

\draw[dashed] (w0) -- (w1) node[pos=0.5,above] (k0) {$(k+1)$};

\node[gray] at (w0)[left=5pt] {$w_i$};
\node[gray] at (w1)[right=5pt] {$w_j$};
\foreach \n in {w0,w1}
  \node at (\n)[gray,circle,fill,inner sep=3pt]{};

\draw[double,red,thick] (k0) -- (k0|-a);

\end{tikzpicture}
}}\right)
\end{equation*}
\begin{equation*}
\left(\vcenter{\hbox{
\begin{tikzpicture}[decoration={
    markings,
    mark=at position 0.5 with {\arrow{>}}}
    ]
\node (w0) at (-1,0) {};
\node (w1) at (1,0) {};
\coordinate (a) at (-1,-1);

\draw[dashed] (w0) -- (w1) node[above,pos=0.66] (k0) {$k$};
\draw[postaction={decorate}] (w0) to[bend right=40] node[pos=0.33,above] (k1) {} (w1);

\node[gray] at (w0)[left=5pt] {$w_i$};
\node[gray] at (w1)[right=5pt] {$w_j$};
\foreach \n in {w0,w1}
  \node at (\n)[gray,circle,fill,inner sep=3pt]{};

\draw[double,red,thick] (k0) -- (k0|-a);
\draw[red] (k1) -- (k1|-a);

\end{tikzpicture}
}}\right)
=
(k+1)_{\mt^{-1}} \left(\vcenter{\hbox{
\begin{tikzpicture}[decoration={
    markings,
    mark=at position 0.5 with {\arrow{>}}}
    ]
\node (w0) at (-1.5,0) {};
\node (w1) at (1.5,0) {};
\coordinate (a) at (-1,-1);

\draw[dashed] (w0) -- (w1) node[pos=0.5,above] (k0) {$(k+1)$};

\node[gray] at (w0)[left=5pt] {$w_i$};
\node[gray] at (w1)[right=5pt] {$w_j$};
\foreach \n in {w0,w1}
  \node at (\n)[gray,circle,fill,inner sep=3pt]{};

\draw[double,red,thick] (k0) -- (k0|-a);

\end{tikzpicture}
}}\right)
\end{equation*}
where we suppose that the classes are the same everywhere outside parentheses, red handles joining same base points and following same paths. 
\end{lemma}
\begin{proof}
We prove the first equality - last three correspond to symmetric situations so they are proved similarly.
The idea of the proof is an application of the compressing trick from Proposition \ref{compress2}, which consists in applying a homotopy compressing the disk until points cannot approach each other vertically anymore without meeting. Namely, let $D$ be the disk depicted in the parentheses. While compressing $D$ to an open $\frac{\epsilon}{2}$-neighborhood $D'$ of $\left(w_i , w_j \right)$, the plain arc from the top will approach the dashed arc. As we work in Borel-Moore homology, so relatively to $X_r \setminus A_{\epsilon}$ for a small $\epsilon$, at some points, the point lying on the plain arc will cut the dashed arc to put its $\epsilon$-neighborhood in. As there are $k$ points lying on the dashed arc, there are $k+1$ possibilities of cuts (between $\left( w_0,M_1 \right), \left( M_1,M_2 \right), \ldots , \left( M_{k-1},M_k \right)$ or $\left( M_{k} , w_j \right)$). The situation may be summed up as the equality of Figure \ref{forkdecomposeddd}. In the figure, we distinguish the point $M$ from the plain arc coming between $M_{i-1}$ and $M_i$ in the sum.

\begin{figure}[h]
\begin{center}
\def\svgwidth{1.1\columnwidth}
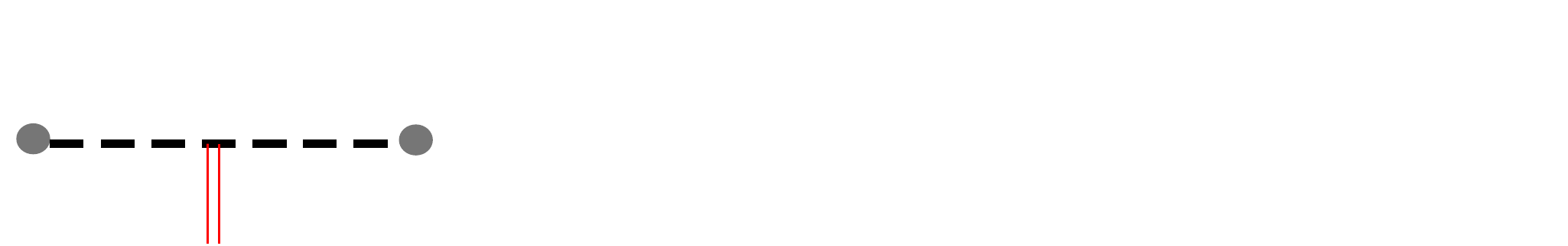
\caption{Homological relation. \label{forkdecomposeddd}}
\end{center}
\end{figure}

To be more precise, let $\phi^k$ be the chain:
\[
\Delta^k \to X_k
\] 
associated with the indexed $k$ dashed arc. Let $\psi$:
$
I \to D_n
$
be the one associated with the plain one. Then:
\[
\Psi = \lbrace  \psi, \phi^k \rbrace : I\times \Delta^k \to X_{k+1}
\]
is the chain associated with the left object of the equality we are studying. For $i=1 , \ldots , k+1$, let $\Delta_i$ be:
\[
\Delta_i = \lbrace (t,t_1 , \ldots , t_k) \in I \times \Delta^k \text{ s.t. } t_{i-1} < t < t_i \rbrace
\]
and $\Psi_i$ be the restriction of $\Psi$ to $\Delta_i$. In terms of chains we have the equality:
\[
\left[ \Psi \right]= \sum_i \left[ \Psi_i \right],
\]
as the set $\lbrace \Delta_i , i=1 , \ldots , k+1 \rbrace$ is a subdivision of $I \times \Delta^k$.
Every $\Delta_i$ is naturally homeomorphic to the standard simplex $\Delta^{k+1}$, but it involves a permutation of coordinates in the parametrization of the simplex:
\[
\tau_i: (t,t_1,\ldots , t_k) \mapsto (t_1 , \ldots , t_{i-1} , t , t_i , \ldots , t_k)
\]
where $\tau_i$ can be seen as an element of $\Sk_{k+1}$. 
By homotoping the plain arc to the dashed one, one obtains a homotopy from $\Psi_i$ to $\phi^{k+1} \circ \tau_i$, for all $i\in \lbrace 1 , \ldots , k+1 \rbrace$, and considered as chains of $\Conf_k(D_n)$. Then:
\[
\left[ \Psi \right]= \sum_{i=1}^{k+1} \sign(\tau_i) \left[ \phi^{k+1} \right] = \sum_{i=1}^{k+1} (-1)^{i-1} \left[ \phi^{k+1} \right] .
\]
This shows that the relation:

\begin{equation}\label{dansZ}
\left(\vcenter{\hbox{
\begin{tikzpicture}[decoration={
    markings,
    mark=at position 0.5 with {\arrow{>}}}
    ]
\node (w0) at (-1,0) {};
\node (w1) at (1,0) {};
\coordinate (a) at (-1,-1);

\draw[dashed] (w0) -- (w1) node[pos=0.5] (k0) {$k$};
\draw[postaction={decorate}] (w0) to[bend left=40] node[pos=0.3,above] (k1) {} (w1);

\node[gray] at (w0)[left=5pt] {$w_i$};
\node[gray] at (w1)[right=5pt] {$w_j$};
\foreach \n in {w0,w1}
  \node at (\n)[gray,circle,fill,inner sep=3pt]{};


\end{tikzpicture}
}}\right)
=
\sum_{i=1}^{k+1} (-1)^{i-1} \left(\vcenter{\hbox{
\begin{tikzpicture}[decoration={
    markings,
    mark=at position 0.5 with {\arrow{>}}}
    ]
\node (w0) at (-1.5,0) {};
\node (w1) at (1.5,0) {};
\coordinate (a) at (-1,-1);

\draw[dashed] (w0) -- (w1) node[pos=0.5,above] (k0) {$(k+1)$};

\node[gray] at (w0)[left=5pt] {$w_i$};
\node[gray] at (w1)[right=5pt] {$w_j$};
\foreach \n in {w0,w1}
  \node at (\n)[gray,circle,fill,inner sep=3pt]{};
%

\end{tikzpicture}
}}\right)
\end{equation}
holds in $H(X_r,X_r^-,\BZ)$. This can be seen as Figure \ref{forkdecomposeddd} without handles. (A drawing without handles corresponds to an unlifted homology class.)

Now it's just a matter of reorganizing the handles in the elements of the sum in Figure \ref{forkdecomposeddd} to get a dashed arc model. Using the handle rule, one can check that for $i \in \lbrace 1, \ldots , k+1 \rbrace$ we have the equality of Figure \ref{tipure} in the local system homology.

\begin{figure}[h]
\begin{center}
\def\svgwidth{0.9\columnwidth}
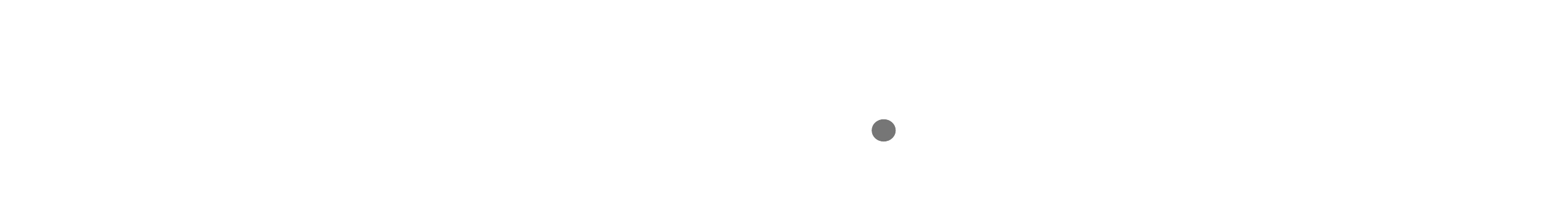
\caption{Local system relation. \label{tipure}}
\end{center}
\end{figure}

To see this, we draw the braid associated with this change of handle, in Figure \ref{handlerulecompress}, so that one verifies its local coordinate to be $t^{i-1}$ (as $(i-1)$ red strands are passing successively in front of the $i^{th}$ one).
\begin{figure}[!h!]
\begin{center}
\def\svgwidth{0.45\columnwidth}
\def\svgscale{0.2}
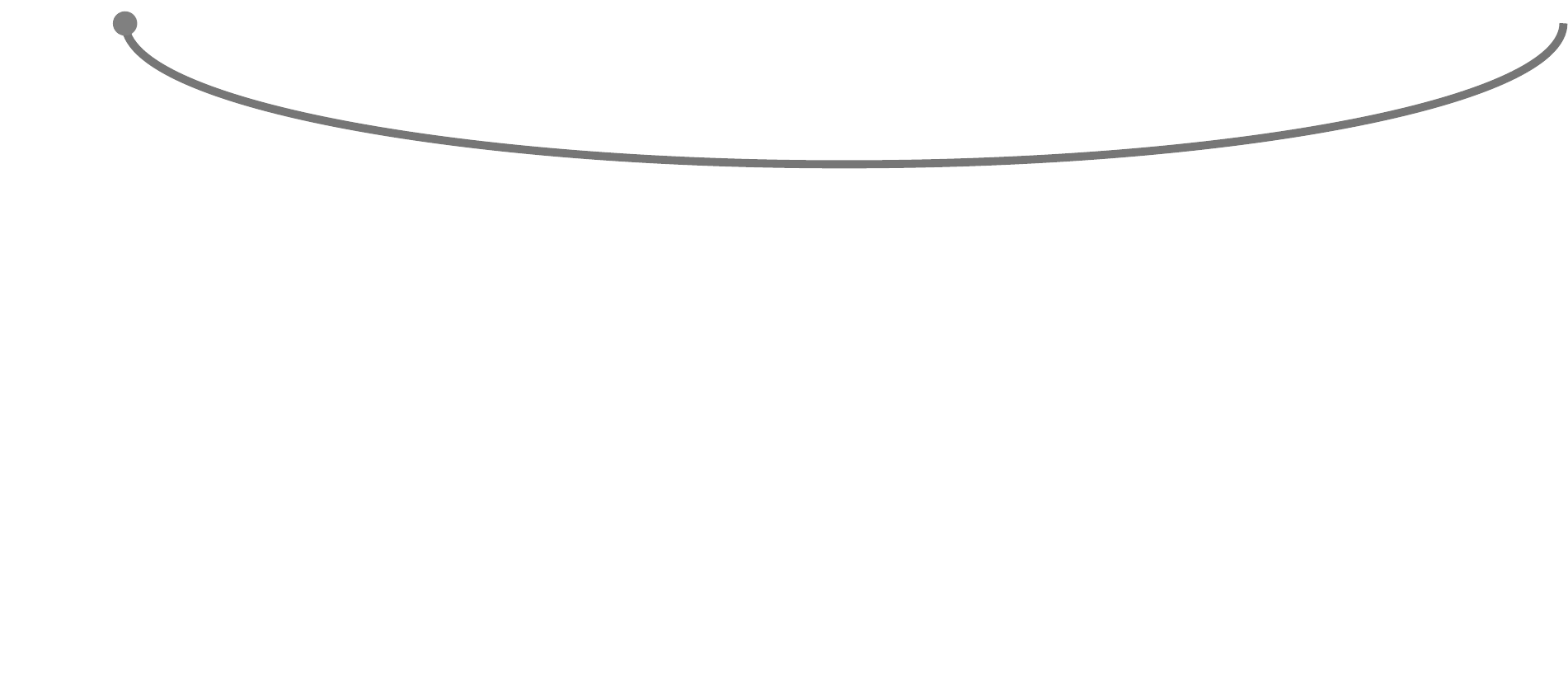
\caption{Handle rule. \label{handlerulecompress}}
\end{center}
\end{figure}

 Again, in the picture, one has to imagine that the red handles are going back to the base point before and after this box following same paths so that these parts do not contribute to the local system coefficient. Now, we can conclude the following equality:
 
\begin{equation*}
\left(\vcenter{\hbox{
\begin{tikzpicture}[decoration={
    markings,
    mark=at position 0.5 with {\arrow{>}}}
    ]
\node (w0) at (-1,0) {};
\node (w1) at (1,0) {};
\coordinate (a) at (-1,-1);

\draw[dashed] (w0) -- (w1) node[pos=0.66] (k0) {$k$};
\draw[postaction={decorate}] (w0) to[bend left=40] node[pos=0.3,above] (k1) {} (w1);

\node[gray] at (w0)[left=5pt] {$w_i$};
\node[gray] at (w1)[right=5pt] {$w_j$};
\foreach \n in {w0,w1}
  \node at (\n)[gray,circle,fill,inner sep=3pt]{};

\draw[double,red,thick] (k0) -- (k0|-a);
\draw[red] (k1) -- (k1|-a);

\end{tikzpicture}
}}\right)
=
\sum_{i=1}^{k+1} (-1)^{i-1} t^{i-1} \left(\vcenter{\hbox{
\begin{tikzpicture}[decoration={
    markings,
    mark=at position 0.5 with {\arrow{>}}}
    ]
\node (w0) at (-1.5,0) {};
\node (w1) at (1.5,0) {};
\coordinate (a) at (-1,-1);

\draw[dashed] (w0) -- (w1) node[pos=0.5,above] (k0) {$(k+1)$};

\node[gray] at (w0)[left=5pt] {$w_i$};
\node[gray] at (w1)[right=5pt] {$w_j$};
\foreach \n in {w0,w1}
  \node at (\n)[gray,circle,fill,inner sep=3pt]{};

\draw[double,red,thick] (k0) -- (k0|-a);

\end{tikzpicture}
}}\right)
\end{equation*} 
where the term $t^{i-1}$ comes from what we've just remarked, and $(-1)^{i-1}$ comes from \ref{dansZ}. This concludes the proof of the  equality we were looking for remembering notation $\mt=-t$. 

\end{proof}
We study one simple example for illustrating the above lemma.

\begin{example}
Let's study the case $k=1$ of the above lemma that gives the following equality:
\begin{equation}\label{theequation}
\left(\vcenter{\hbox{
\begin{tikzpicture}[decoration={
    markings,
    mark=at position 0.5 with {\arrow{>}}}
    ]
\node (w0) at (-1,0) {};
\node (w1) at (1,0) {};
\coordinate (a) at (-1,-1);

\draw[postaction={decorate}] (w0) -- (w1);
\draw[postaction={decorate}] (w0) to[bend left=40] node[pos=0.3,above] (k1) {} (w1);

\node[gray] at (w0)[left=5pt] {$w_1$};
\node[gray] at (w1)[right=5pt] {$w_2$};
\foreach \n in {w0,w1}
  \node at (\n)[gray,circle,fill,inner sep=3pt]{};

\draw[red] (k0) -- (k0|-a);
\draw[red] (k1) -- (k1|-a);

\end{tikzpicture}
}}\right)
=
(1-t) \left(\vcenter{\hbox{
\begin{tikzpicture}[decoration={
    markings,
    mark=at position 0.5 with {\arrow{>}}}
    ]
\node (w0) at (-1.5,0) {};
\node (w1) at (1.5,0) {};
\coordinate (a) at (-1,-1);

\draw[dashed] (w0) -- (w1) node[pos=0.5,above] (k0) {$2$};

\node[gray] at (w0)[left=5pt] {$w_1$};
\node[gray] at (w1)[right=5pt] {$w_2$};
\foreach \n in {w0,w1}
  \node at (\n)[gray,circle,fill,inner sep=3pt]{};

\draw[double,red,thick] (k0) -- (k0|-a);

\end{tikzpicture}
}}\right).
\end{equation}
Let's first consider the case $t=q^{\alpha_1} = q^{\alpha_2} =1$ consisting in working with $\BZ$-homology, namely without considering the cover. As $t=1$, the above relation becomes:
\begin{equation*}
\left(\vcenter{\hbox{
\begin{tikzpicture}[decoration={
    markings,
    mark=at position 0.5 with {\arrow{>}}}
    ]
\node (w0) at (-1,0) {};
\node (w1) at (1,0) {};
\coordinate (a) at (-1,-1);

\draw[postaction={decorate}] (w0) -- (w1);
\draw[postaction={decorate}] (w0) to[bend left=40] node[pos=0.3,above] (k1) {} (w1);

\node[gray] at (w0)[left=5pt] {$w_1$};
\node[gray] at (w1)[right=5pt] {$w_2$};
\foreach \n in {w0,w1}
  \node at (\n)[gray,circle,fill,inner sep=3pt]{};


\end{tikzpicture}
}}\right)
=0 .
\end{equation*}

The left term corresponds to an embedding of a square in $\Conf_2(D_2)$ (considering only two punctures $w_1$ and $w_2$ in the disk), itself defining a cycle in $X_2$ and a class in $\Hlf_2(X_2; \BZ)$. This square is represented in Figure \ref{carre1} where points $a,b,c,d$ are respectively $(w_1,w_1),(w_1,w_2),(w_2,w_2)$ and $(w_2,w_1)$ and gray tubes are parts of hyperplanes (equations written in the figure). The arrows recalls the orientation of embedded intervals. 
\begin{figure}[h!]
\begin{center}
\def\svgwidth{0.7\columnwidth}
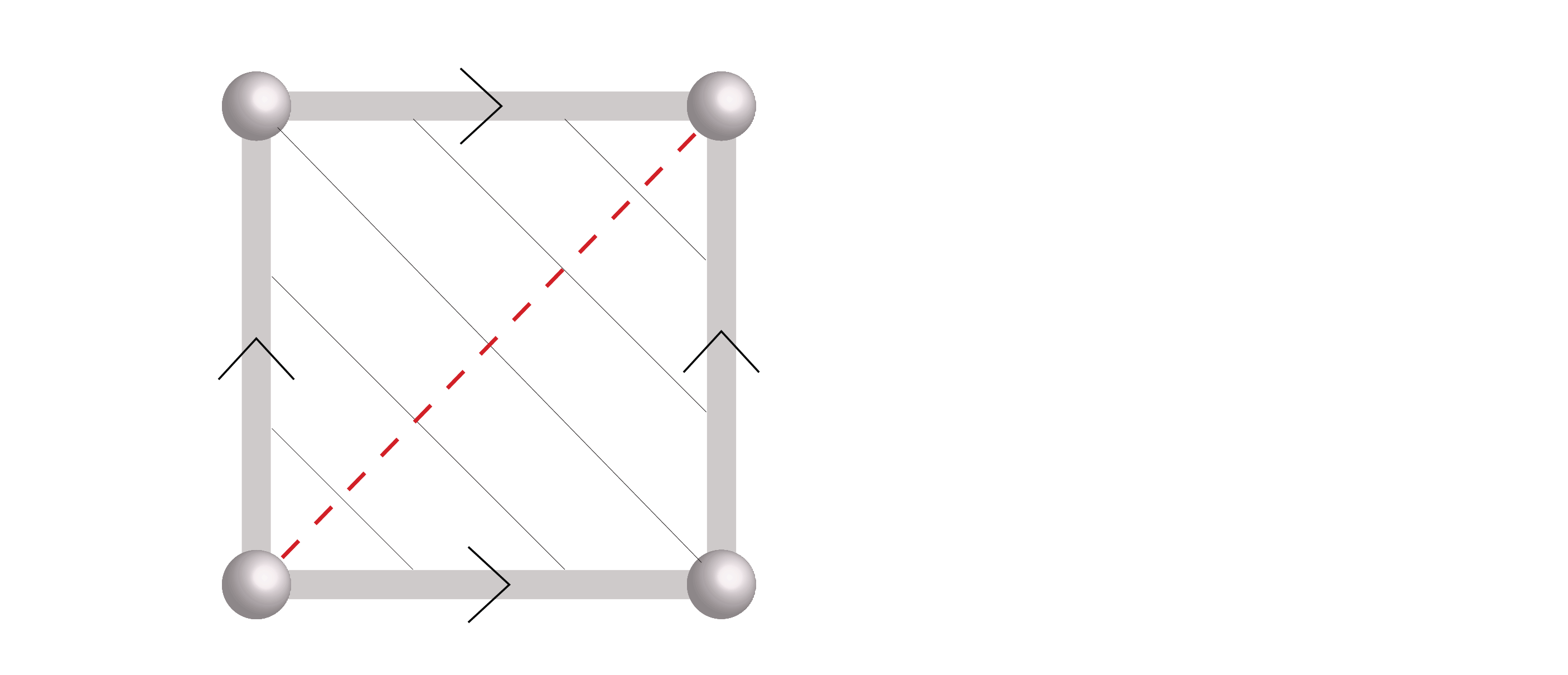
\caption{The square corresponding to the left term. \label{carre1}}
\end{center}
\end{figure}

The red dashed diagonal stands to help the reader figuring out how to decompose the square as the gluing of two simplices. These two simplices are identified after making the quotient by the permutation group but with opposite orientations which justifies that this chain is zero, so it is as a homology class. Now we remove conditions $t=q^{\alpha_1} = q^{\alpha_2} =1$, and we go back to Equation \ref{theequation}. Then now $a,b,c,d$ and hyperplanes from Figure \ref{carre1} are lifted to $\widehat{Conf_2(D_2)}$, the cover associated with $\rho_2$. In Figure \ref{carre2}, the square from Figure \ref{carre1} is seen from another side so to make the hyperplane $\lbrace z_1 = z_2 \rbrace$ appeared. More precisely, Figure \ref{carre1} can be think as a top view of Figure \ref{carre2}. 

\begin{figure}[h!]
\begin{center}
\def\svgwidth{0.7\columnwidth}
\begingroup%
  \makeatletter%
  \providecommand\color[2][]{%
    \errmessage{(Inkscape) Color is used for the text in Inkscape, but the package 'color.sty' is not loaded}%
    \renewcommand\color[2][]{}%
  }%
  \providecommand\transparent[1]{%
    \errmessage{(Inkscape) Transparency is used (non-zero) for the text in Inkscape, but the package 'transparent.sty' is not loaded}%
    \renewcommand\transparent[1]{}%
  }%
  \providecommand\rotatebox[2]{#2}%
  \newcommand*\fsize{\dimexpr\f@size pt\relax}%
  \newcommand*\lineheight[1]{\fontsize{\fsize}{#1\fsize}\selectfont}%
  \ifx\svgwidth\undefined%
    \setlength{\unitlength}{866.80341461bp}%
    \ifx\svgscale\undefined%
      \relax%
    \else%
      \setlength{\unitlength}{\unitlength * \real{\svgscale}}%
    \fi%
  \else%
    \setlength{\unitlength}{\svgwidth}%
  \fi%
  \global\let\svgwidth\undefined%
  \global\let\svgscale\undefined%
  \makeatother%
  \begin{picture}(1,0.3216801)%
    \lineheight{1}%
    \setlength\tabcolsep{0pt}%
    \put(0,0){\includegraphics[width=\unitlength,page=1]{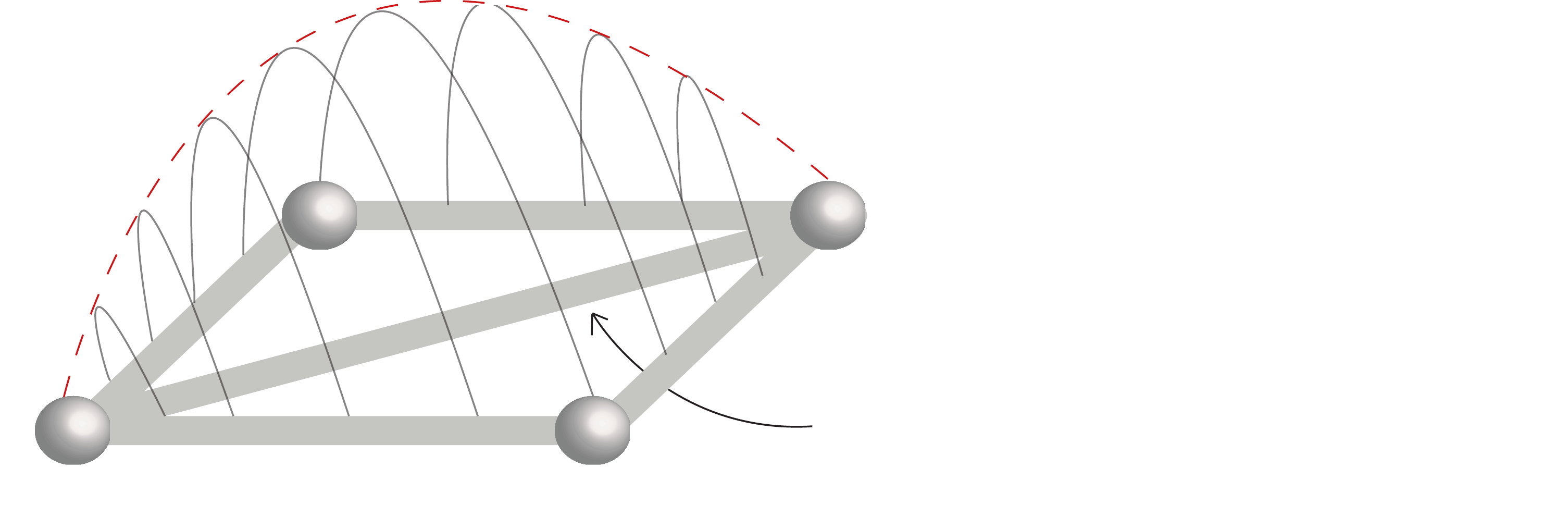}}%
    \put(0.52822011,0.04801065){\color[rgb]{0,0,0}\makebox(0,0)[lt]{\lineheight{1.25}\smash{\begin{tabular}[t]{l}$\lbrace z_1 = z_2 \rbrace$\end{tabular}}}}%
    \put(-0.00287289,0.01236195){\color[rgb]{0,0,0}\makebox(0,0)[lt]{\lineheight{1.25}\smash{\begin{tabular}[t]{l}$a$\end{tabular}}}}%
    \put(0.57332666,0.19278809){\color[rgb]{0,0,0}\makebox(0,0)[lt]{\lineheight{1.25}\smash{\begin{tabular}[t]{l}$c$\end{tabular}}}}%
    \put(0.40817856,0.00508671){\color[rgb]{0,0,0}\makebox(0,0)[lt]{\lineheight{1.25}\smash{\begin{tabular}[t]{l}$d$\end{tabular}}}}%
    \put(0.40817856,0.00508671){\color[rgb]{0,0,0}\makebox(0,0)[lt]{\lineheight{1.25}\smash{\begin{tabular}[t]{l}$d$\end{tabular}}}}%
  \end{picture}%
\endgroup%

\caption{The square corresponding to the left term. \label{carre2}}
\end{center}
\end{figure}

The isotopy crashing the plain arc on the dashed arc in the proof of the above Lemma \ref{plaintodashed} corresponds to an isotopy crashing the dashed surface from Figure \ref{carre2} on the ground square (containing the hyperplane $\lbrace z_1 = z_2 \rbrace$). Figure \ref{carre3} shows a movie of such an isotopy viewed from the top. 
\begin{figure}[h!]
\begin{center}
\def\svgwidth{0.7\columnwidth}
\begingroup%
  \makeatletter%
  \providecommand\color[2][]{%
    \errmessage{(Inkscape) Color is used for the text in Inkscape, but the package 'color.sty' is not loaded}%
    \renewcommand\color[2][]{}%
  }%
  \providecommand\transparent[1]{%
    \errmessage{(Inkscape) Transparency is used (non-zero) for the text in Inkscape, but the package 'transparent.sty' is not loaded}%
    \renewcommand\transparent[1]{}%
  }%
  \providecommand\rotatebox[2]{#2}%
  \newcommand*\fsize{\dimexpr\f@size pt\relax}%
  \newcommand*\lineheight[1]{\fontsize{\fsize}{#1\fsize}\selectfont}%
  \ifx\svgwidth\undefined%
    \setlength{\unitlength}{1126.12375738bp}%
    \ifx\svgscale\undefined%
      \relax%
    \else%
      \setlength{\unitlength}{\unitlength * \real{\svgscale}}%
    \fi%
  \else%
    \setlength{\unitlength}{\svgwidth}%
  \fi%
  \global\let\svgwidth\undefined%
  \global\let\svgscale\undefined%
  \makeatother%
  \begin{picture}(1,0.30925721)%
    \lineheight{1}%
    \setlength\tabcolsep{0pt}%
    \put(0,0){\includegraphics[width=\unitlength,page=1]{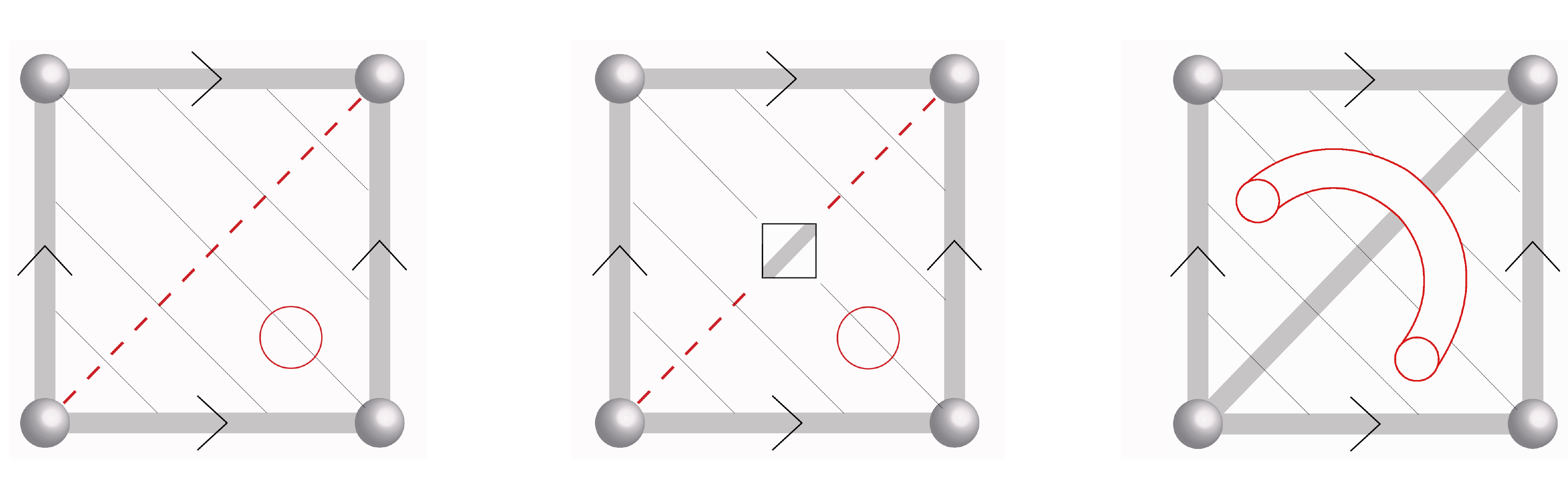}}%
    \put(0.28447424,0.14567425){\color[rgb]{0,0,0}\makebox(0,0)[lt]{\lineheight{1.25}\smash{\begin{tabular}[t]{l}$\to$\end{tabular}}}}%
    \put(0.66460973,0.13696279){\color[rgb]{0,0,0}\makebox(0,0)[lt]{\lineheight{1.25}\smash{\begin{tabular}[t]{l}$\to$\end{tabular}}}}%
    \put(-0.00221133,0.00708318){\color[rgb]{0,0,0}\makebox(0,0)[lt]{\lineheight{1.25}\smash{\begin{tabular}[t]{l}$a$\end{tabular}}}}%
    \put(-0.00141938,0.28901701){\color[rgb]{0,0,0}\makebox(0,0)[lt]{\lineheight{1.25}\smash{\begin{tabular}[t]{l}$b$\end{tabular}}}}%
    \put(0.25992381,0.2842653){\color[rgb]{0,0,0}\makebox(0,0)[lt]{\lineheight{1.25}\smash{\begin{tabular}[t]{l}$c$\end{tabular}}}}%
    \put(0.25675599,0.00391536){\color[rgb]{0,0,0}\makebox(0,0)[lt]{\lineheight{1.25}\smash{\begin{tabular}[t]{l}$d$\end{tabular}}}}%
  \end{picture}%
\endgroup%

\caption{The movie of the isotopy. \label{carre3}}
\end{center}
\end{figure}

The red circle shows where the path corresponding to red handles arrives. The last step of the movie shows the sum of two simplices glued from either side of the hyperplane $\lbrace z_1 = z_2 \rbrace$. The red tube shows how the red handle path has to bypass this hyperplane (involving the $t$ coefficient). All this isotopy happens in $\widehat{\Conf_2(D_2)}$. Working in $X_2$, upper and lower simplices are identified (with different lifts and orientation), so that in $\widehat{X_2}$ the end of the isotopy is equal to:

\begin{figure}[h!]
\begin{center}
\def\svgwidth{0.3\columnwidth}
\begingroup%
  \makeatletter%
  \providecommand\color[2][]{%
    \errmessage{(Inkscape) Color is used for the text in Inkscape, but the package 'color.sty' is not loaded}%
    \renewcommand\color[2][]{}%
  }%
  \providecommand\transparent[1]{%
    \errmessage{(Inkscape) Transparency is used (non-zero) for the text in Inkscape, but the package 'transparent.sty' is not loaded}%
    \renewcommand\transparent[1]{}%
  }%
  \providecommand\rotatebox[2]{#2}%
  \newcommand*\fsize{\dimexpr\f@size pt\relax}%
  \newcommand*\lineheight[1]{\fontsize{\fsize}{#1\fsize}\selectfont}%
  \ifx\svgwidth\undefined%
    \setlength{\unitlength}{351.17782583bp}%
    \ifx\svgscale\undefined%
      \relax%
    \else%
      \setlength{\unitlength}{\unitlength * \real{\svgscale}}%
    \fi%
  \else%
    \setlength{\unitlength}{\svgwidth}%
  \fi%
  \global\let\svgwidth\undefined%
  \global\let\svgscale\undefined%
  \makeatother%
  \begin{picture}(1,0.76816636)%
    \lineheight{1}%
    \setlength\tabcolsep{0pt}%
    \put(0,0){\includegraphics[width=\unitlength,page=1]{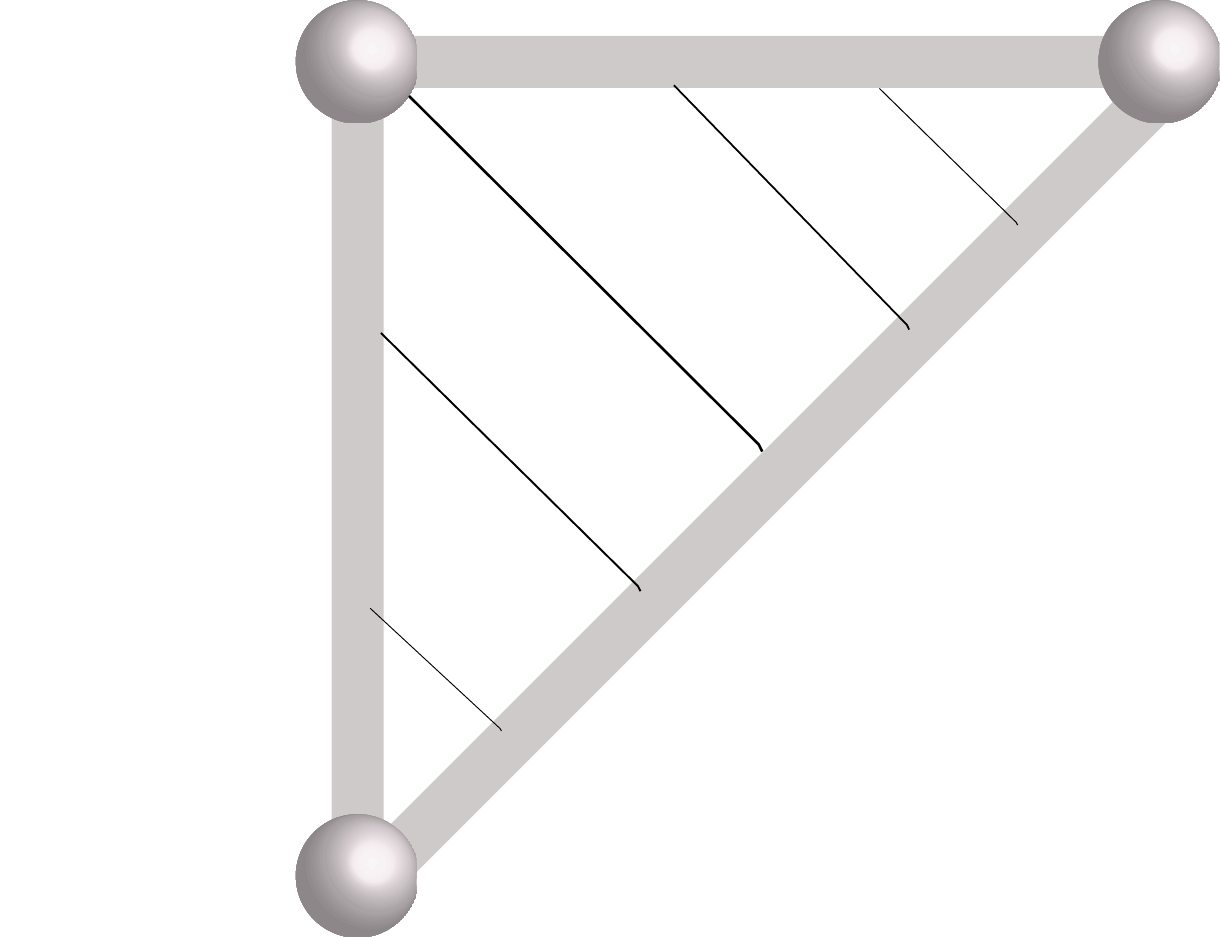}}%
    \put(-0.00709109,0.39501728){\color[rgb]{0,0,0}\makebox(0,0)[lt]{\lineheight{1.25}\smash{\begin{tabular}[t]{l}$(1-t)$\end{tabular}}}}%
  \end{picture}%
\endgroup%

\caption{The end of the isotopy considered in $\Habs_2$. \label{carre4}}
\end{center}
\end{figure}

as a class in $\Habs_2$. This is the right term of Equation \ref{theequation}. 

\end{example}

From Lemma \ref{plaintodashed} we deduce several corollaries. A first straightforward consequence of Lemma \ref{plaintodashed} is the following.

\begin{coro}\label{1forkto1code}
Let $k>1$ be an integer, the following equality holds in $\Hrelm_{\bullet}$:
\begin{equation*}
\left(\vcenter{\hbox{
\begin{tikzpicture}[decoration={
    markings,
    mark=at position 0.5 with {\arrow{>}}}
    ]
\node (w0) at (-1,0) {};
\node (w1) at (1,0) {};
\coordinate (a) at (-1,-1);

\draw[postaction={decorate}] (w0) to[bend left=40] node[pos=0.3,above] (k1) {} (w1);
\node at (w1)[above=4pt,left=23pt] {.};
\node at (w1)[left=23pt] {.};
\node at (w1)[below=4pt,left=23pt] {.};
\node at (w1)[above=20pt,left=20pt] {$k$};
\draw[postaction={decorate}] (w0) to[bend left=-40] node[pos=0.7,above] (kk) {} (w1);

\node[gray] at (w0)[left=5pt] {$w_i$};
\node[gray] at (w1)[right=5pt] {$w_j$};
\foreach \n in {w0,w1}
  \node at (\n)[gray,circle,fill,inner sep=3pt]{};

\draw[red] (k1) -- (k1|-a);
\draw[red] (kk) -- (kk|-a);

\end{tikzpicture}
}}\right)
=
(k)_{\mt}! \left(\vcenter{\hbox{
\begin{tikzpicture}[decoration={
    markings,
    mark=at position 0.5 with {\arrow{>}}}
    ]
\node (w0) at (-1.5,0) {};
\node (w1) at (1.5,0) {};
\coordinate (a) at (-1,-1);

\draw[dashed] (w0) -- (w1) node[pos=0.5,above] (k0) {$k$};

\node[gray] at (w0)[left=5pt] {$w_i$};
\node[gray] at (w1)[right=5pt] {$w_j$};
\foreach \n in {w0,w1}
  \node at (\n)[gray,circle,fill,inner sep=3pt]{};

\draw[double,red,thick] (k0) -- (k0|-a);

\end{tikzpicture}
}}\right)
\end{equation*}.
\end{coro}
\begin{proof}
The proof is made by recursion on $k$. The recursion property is given by Lemma \ref{plaintodashed}. 
\end{proof}

Lemma \ref{plaintodashed} allows also one to compute the fusion between two dashed arcs.

\begin{coro}\label{dashedtodashed}
For integers $k,l>1$, there is the following relation between classes:
\begin{equation*}
\left(\vcenter{\hbox{
\begin{tikzpicture}[decoration={
    markings,
    mark=at position 0.5 with {\arrow{>}}}
    ]
\node (w0) at (-1,0) {};
\node (w1) at (1,0) {};
\coordinate (a) at (-1,-1);

\draw[dashed] (w0) to[bend right=30](w1) node[pos=0.7] (k0) {$k$};
\draw[dashed] (w0) to[bend left=30] node[pos=0.3,above] (k1) {$l$} (w1);

\node[gray] at (w0)[left=5pt] {$w_i$};
\node[gray] at (w1)[right=5pt] {$w_j$};
\foreach \n in {w0,w1}
  \node at (\n)[gray,circle,fill,inner sep=3pt]{};

\draw[double,red,thick] (k0) -- (k0|-a);
\draw[double,red,thick] (k1) -- (k1|-a);

\end{tikzpicture}
}}\right)
=
{{k+l}\choose{l}}_{\mt}
\left(\vcenter{\hbox{
\begin{tikzpicture}[decoration={
    markings,
    mark=at position 0.5 with {\arrow{>}}}
    ]
\node (w0) at (-1.5,0) {};
\node (w1) at (1.5,0) {};
\coordinate (a) at (-1,-1);

\draw[dashed] (w0) -- (w1) node[pos=0.5,above] (k0) {$(k+l)$};

\node[gray] at (w0)[left=5pt] {$w_i$};
\node[gray] at (w1)[right=5pt] {$w_j$};
\foreach \n in {w0,w1}
  \node at (\n)[gray,circle,fill,inner sep=3pt]{};

\draw[double,red,thick] (k0) -- (k0|-a);

\end{tikzpicture}
}}\right).
\end{equation*} 
\end{coro}
\begin{proof}
The two following equalities are direct consequences of previous Corollary \ref{plaintodashed}.
\[
\begin{array}{rl}
(k)_{\mt}! (l)_{\mt}! \left(\vcenter{\hbox{
\begin{tikzpicture}[decoration={
    markings,
    mark=at position 0.5 with {\arrow{>}}}
    ]
\node (w0) at (-1,0) {};
\node (w1) at (1,0) {};
\coordinate (a) at (-1,-1);

\draw[dashed] (w0) to[bend right=30](w1) node[pos=0.7] (k0) {$k$};
\draw[dashed] (w0) to[bend left=30] node[pos=0.3,above] (k1) {$l$} (w1);

\node[gray] at (w0)[left=5pt] {$w_i$};
\node[gray] at (w1)[right=5pt] {$w_j$};
\foreach \n in {w0,w1}
  \node at (\n)[gray,circle,fill,inner sep=3pt]{};

\draw[double,red,thick] (k0) -- (k0|-a);
\draw[double,red,thick] (k1) -- (k1|-a);

\end{tikzpicture}
}}\right) &
=
\left(\vcenter{\hbox{
\begin{tikzpicture}[decoration={
    markings,
    mark=at position 0.5 with {\arrow{>}}}
    ]
\node (w0) at (-1,0) {};
\node (w1) at (1,0) {};
\coordinate (a) at (-1,-1);

\draw[postaction={decorate}] (w0) to[bend left=40] node[pos=0.3,above] (k1) {} (w1);
\node at (w1)[above=4pt,left=23pt] {.};
\node at (w1)[left=23pt] {.};
\node at (w1)[below=4pt,left=23pt] {.};
\node at (w1)[above=20pt,left=20pt] {$k+l$};
\draw[postaction={decorate}] (w0) to[bend left=-40] node[pos=0.7,above] (kk) {} (w1);

\node[gray] at (w0)[left=5pt] {$w_i$};
\node[gray] at (w1)[right=5pt] {$w_j$};
\foreach \n in {w0,w1}
  \node at (\n)[gray,circle,fill,inner sep=3pt]{};

\draw[red] (k1) -- (k1|-a);
\draw[red] (kk) -- (kk|-a);

\end{tikzpicture}
}}\right) \\
& = (k+l)_{\mt}! 
\left(\vcenter{\hbox{
\begin{tikzpicture}[decoration={
    markings,
    mark=at position 0.5 with {\arrow{>}}}
    ]
\node (w0) at (-1.5,0) {};
\node (w1) at (1.5,0) {};
\coordinate (a) at (-1,-1);

\draw[dashed] (w0) -- (w1) node[pos=0.5,above] (k0) {$(k+l)$};

\node[gray] at (w0)[left=5pt] {$w_i$};
\node[gray] at (w1)[right=5pt] {$w_j$};
\foreach \n in {w0,w1}
  \node at (\n)[gray,circle,fill,inner sep=3pt]{};

\draw[double,red,thick] (k0) -- (k0|-a);

\end{tikzpicture}
}}\right).
\end{array}
\]
One concludes using the integral equality:
\[
(k+l)_{\mt}! = (k)_{\mt}!(l)_{\mt}! {{k+l}\choose{l}}_{\mt}
\]
and simplification by $(k)_{\mt}!(l)_{\mt}!$.
\end{proof}

\subsection{Basis of multi-arcs}

We recall that $\CA'$ and $\CU$ are respectively families of code sequences and of multi-arcs, and that $\CU$ was shown to be a basis of $\Hrelm_r$ as an $\Laurentmax$-module, see Corollary \ref{HrelStruc}. We prove a proposition relating multi-arcs with code sequences. 

\begin{prop}\label{ArcstoCodes}
Let ${\bf k} \in \EZnr$. There is the following expression for the standard multi-arc in terms of code sequences.
\begin{align*}
A'(k_0, \ldots , k_{n-1}) = & \sum_{l_{n-1} = 0}^{k_{n-1}}  \sum_{l_{n-2} = 0}^{k_{n-2}+l_{n-1}} \cdots  \sum_{l_{1} = 0}^{k_{1}+l_2} \left(  \prod_{i=0}^{n-2} {{k_i+l_{i+1}}\choose{l_{i+1}}}_{\mt} U(k_0',k_1'',\ldots,k_{n-2}'',k_{n-1}') \right)\\
\end{align*}
where $k_0'=k_0+l_1,k_{n-1}'= k_{n-1}-l_{n-1}$ and $k_i''= k_{i}+l_{i+1}-l_{i}$ for $i=1, \ldots , n-2$. 
\end{prop}
\begin{proof}
Let ${\bf k} \in \EZnr$ and $A'$ its associated multi-arcs. We treat one by one the dashed arcs of $A'$ starting by the one ending at $w_n$ then the one ending at $w_{n-1}$ and so on. The first step is the following:
\begin{align*}
\left(\vcenter{\hbox{
\begin{tikzpicture}[decoration={
    markings,
    mark=at position 0.5 with {\arrow{>}}}
    ]
\node (w0) at (-1.5,0) {};
\node[gray] at (0,0) {$\ldots$};
\node (wn) at (2.5,0) {};
\node (wn1) at (1.5,0) {};
\coordinate (a) at (-2,-2);
\draw[dashed] (w0) to[bend right=50] node[pos=0.7] (k0) {\small $k_{n-1}$} (wn);
\draw[dashed] (w0) to[bend right=30] node[pos=0.3] (k1) {\small $k_{n-2}$} (wn1);
\node[gray] at (w0)[left=5pt] {$w_0$};
\node[gray] at (wn)[above=5pt] {$w_n$};
\node[gray] at (wn1)[above=5pt] {$w_{n-1}$};
\foreach \n in {w0,wn,wn1}
  \node at (\n)[gray,circle,fill,inner sep=3pt]{};
\draw[double,red,thick] (k0) -- (k0|-a);
\draw[double,red,thick] (k1) -- (k1|-a);
\draw[gray,thick] (-1.5,0.3) -- (w0) -- (-1.5,-2.5);
\end{tikzpicture}
}} \right)
& =
\sum_{l_{n-1}=0}^{k_{n-1}}
\left(\vcenter{\hbox{
\begin{tikzpicture}[decoration={
    markings,
    mark=at position 0.5 with {\arrow{>}}}
    ]
\node (w0) at (-1.5,0) {};
\node[gray] at (0,0) {$\ldots$};
\node (wn) at (2.5,0) {};
\node (wn1) at (1.5,0) {};
\coordinate (a) at (-2,-2);
\draw[dashed] (w0) to[bend right=80] node[pos=0.7] (k0) {\small $l_{n-1}$} (wn1);
\draw[dashed] (w0) to[bend right=30] node[pos=0.3,above] (k1) {\small $k_{n-2}$} (wn1);
\draw[dashed] (wn1) -- (wn) node[above,pos=0.5] (k2) {\small $k_{n-1}'$};
\node[gray] at (w0)[left=5pt] {$w_0$};
\node[gray] at (wn)[above=5pt] {$w_n$};
\node[gray] at (wn1)[above left] {$w_{n-1}$};
\foreach \n in {w0,wn,wn1}
  \node at (\n)[gray,circle,fill,inner sep=3pt]{};
\draw[double,red,thick] (k0) -- (k0|-a);
\draw[double,red,thick] (k1) -- (k1|-a);
\draw[double,red,thick] (k2) -- (k2|-a);
\draw[gray,thick] (-1.5,0.3) -- (w0) -- (-1.5,-2.5);
\end{tikzpicture}
}} \right)\\
& =
\sum_{l_{n-1}=0}^{k_{n-1}} {{k_{n-2}+l_{n-1}}\choose{l_{n-1}}}_{\mt}
\left(\vcenter{\hbox{
\begin{tikzpicture}[decoration={
    markings,
    mark=at position 0.5 with {\arrow{>}}}
    ]
\node (w0) at (-1.5,0) {};
\node[gray] at (0,0) {$\ldots$};
\node (wn) at (2.5,0) {};
\node (wn1) at (1.5,0) {};
\coordinate (a) at (-2,-2);
\draw[dashed] (w0) to[bend right=60] node[above=0.2pt,pos=0.5] (k0) {\tiny $(k_{n-2}+l_{n-1})$} (wn1);
\draw[dashed] (wn1) -- (wn) node[above=1pt,pos=0.5] (k2) {\tiny $k_{n-1}'$};
\node[gray] at (w0)[left=5pt] {$w_0$};
\node[gray] at (wn)[above=5pt] {$w_n$};
\node[gray] at (wn1)[above left] {$w_{n-1}$};
\foreach \n in {w0,wn,wn1}
  \node at (\n)[gray,circle,fill,inner sep=3pt]{};
\draw[double,red,thick] (k0) -- (k0|-a);
\draw[double,red,thick] (k2) -- (k2|-a);
\draw[gray,thick] (-1.5,0.3) -- (w0) -- (-1.5,-2.5);
\end{tikzpicture}
}}\right)
\end{align*}
with $k_{n-1}'=k_{n-1}-l_{n-1}$. The first equality is a breaking of dashed arc, see Example \ref{breakingdashed}. The second equality is a direct application of Corollary \ref{dashedtodashed}. The end of the proof is an iteration of this process. Next step is the following, with $k_{n-2}'=k_{n-2}+l_{n-1}$:
\begin{align*}
\left(\vcenter{\hbox{
\begin{tikzpicture}[decoration={
    markings,
    mark=at position 0.5 with {\arrow{>}}}
    ]
\node (w0) at (-1.5,0) {};
\node[gray] at (-0.5,0) {$\ldots$};
\node (wn) at (2.5,0) {};
\node (wn1) at (1.5,0) {};
\node (wn2) at (0.5,0) {};
\coordinate (a) at (-2,-2);
\draw[dashed] (w0) to[bend right=70] node[above,pos=0.7] (k0) {\small $k_{n-2}'$} (wn1);
\draw[dashed] (wn1) -- (wn) node[above,pos=0.5] (k2) {\tiny $k_{n-1}'$};
\draw[dashed] (w0) to[bend right=20] node[pos=0.4] (k3) {\small $k_{n-3}$} (wn2);
\node[gray] at (w0)[left=5pt] {$w_0$};
\node[gray] at (wn)[above=5pt] {$w_n$};
\node[gray] at (wn2)[above=5pt] {$w_{n-2}$};
\foreach \n in {w0,wn,wn1,wn2}
  \node at (\n)[gray,circle,fill,inner sep=3pt]{};
\draw[double,red,thick] (k0) -- (k0|-a);
\draw[double,red,thick] (k2) -- (k2|-a);
\draw[double,red,thick] (k3) -- (k3|-a);
\draw[gray,thick] (-1.5,0.3) -- (w0) -- (-1.5,-2.5);
\end{tikzpicture}
}}\right)
& = \sum_{l_{n-2}=0}^{k_{n-2}'}
\left(\vcenter{\hbox{
\begin{tikzpicture}[decoration={
    markings,
    mark=at position 0.5 with {\arrow{>}}}
    ]
\node (w0) at (-1.5,0) {};
\node[gray] at (-0.5,0) {$\ldots$};
\node (wn) at (2.5,0) {};
\node (wn1) at (1.5,0) {};
\node (wn2) at (0.5,0) {};
\coordinate (a) at (-2,-2);
\draw[dashed] (wn1) to node[above,pos=0.5] (k0) {\tiny $k_{n-2}''$} (wn2);
\draw[dashed] (wn1) -- (wn) node[above,pos=0.5] (k2) {\tiny $k_{n-1}'$};
\draw[dashed] (w0) to[bend right=20] node[pos=0.3] (k3) {\small $k_{n-3}$} (wn2);
\draw[dashed] (w0) to[bend right=90] node[above,pos=0.6] (k4) {\small $l_{n-2}$} (wn2);
\node[gray] at (w0)[left=5pt] {$w_0$};
\node[gray] at (wn)[above=5pt] {$w_n$};
\foreach \n in {w0,wn,wn1,wn2}
  \node at (\n)[gray,circle,fill,inner sep=3pt]{};
\draw[double,red,thick] (k0) -- (k0|-a);
\draw[double,red,thick] (k2) -- (k2|-a);
\draw[double,red,thick] (k3) -- (k3|-a);
\draw[double,red,thick] (k4) -- (k4|-a);
\draw[gray,thick] (-1.5,0.3) -- (w0) -- (-1.5,-2.5);
\end{tikzpicture}
}}\right) \\
& = \sum_{l_{n-2}=0}^{k_{n-2}'} {{k_{n-3}+l_{n-2}}\choose{l_{n-2}}}_{\mt}
\left(\vcenter{\hbox{
\begin{tikzpicture}[decoration={
    markings,
    mark=at position 0.5 with {\arrow{>}}}
    ]
\node (w0) at (-1.5,0) {};
\node[gray] at (-0.5,0) {$\ldots$};
\node (wn) at (2.5,0) {};
\node (wn1) at (1.5,0) {};
\node (wn2) at (0.5,0) {};
\coordinate (a) at (-2,-2);
\draw[dashed] (wn1) to node[above,pos=0.5] (k0) {\tiny $k_{n-2}''$} (wn2);
\draw[dashed] (wn1) -- (wn) node[above,pos=0.5] (k2) {\tiny $k_{n-1}'$};
\draw[dashed] (w0) to[bend right=80] node[above,pos=0.5] (k3) {\small $k_{n-3}+l_{n-2}$} (wn2);
\node[gray] at (w0)[left=5pt] {$w_0$};
\node[gray] at (wn)[above=5pt] {$w_n$};
\foreach \n in {w0,wn,wn1,wn2}
  \node at (\n)[gray,circle,fill,inner sep=3pt]{};
\draw[double,red,thick] (k0) -- (k0|-a);
\draw[double,red,thick] (k2) -- (k2|-a);
\draw[double,red,thick] (k3) -- (k3|-a);
\draw[gray,thick] (-1.5,0.3) -- (w0) -- (-1.5,-2.5);
\end{tikzpicture}
}} \right)
\end{align*}
where $k_{n-2}'' = k_{n-2}'-l_{n-2}$. A complete iteration of this process gives the formula of the proposition. 
\end{proof}

By looking at the diagonal terms of the matrix expressing mutli-arcs in the code sequence basis, one gets the following corollary.

\begin{coro}[Basis of multi-arcs]\label{Arcsbasis}
The family $\CA'$ of multi-arcs is a basis of $\Hrelm_r$ as an $\Laurentmax$-module.
\end{coro}
\begin{proof}
Let $\EZnr$ being given the lexical order. This yields an order on families $\CA'$ and $\CU$. One can see from Proposition \ref{ArcstoCodes} that with this order, the matrix expressing multi-arcs in the code sequence basis is upper-triangular. The determinant of this matrix is given by the product of diagonal terms. The diagonal terms are the binomial in the sum of the formula from Proposition \ref{ArcstoCodes} corresponding to $l_i=0$ for all $i \in \lbrace 1,\ldots , n-1 \rbrace$. In these cases, the binomials are equal to $1$ so that the determinant of the matrix is $1$. As $\CU$ is a basis and the change of basis determinant is invertible, the proof is complete. 
\end{proof}

The family of multi-arcs will play a central role in this work as it is a basis of the homology thanks to this last result.

\section{Quantum algebra}\label{quantumalgebra}

This section is independent from previous ones. We fix notations for quantum algebra objects that will be recovered by the above introduced homological modules. 


The most standard definition of the quantum algebra $\Uq$ is as a vector space over a rational field.

\begin{defn}\label{Uqnaif}
The algebra $\Uq$ is the algebra over $\BQ(q)$ generated by elements $E,F$ and $K^{\pm 1}$, satisfying the following relations:
\begin{align*}
KEK^{-1}=q^2E & \text{ , } KFK^{-1}=q^{-2}F \\
\left[E, F \right] = \frac{K-K^{-1}}{q-q^{-1}} & \text{ and }
KK^{-1}=K^{-1}K=1 .
\end{align*}
The algebra $\Uq$ is endowed with a coalgebra structure defined by $\Delta$ and $\epsilon$ as follows:
\[
\begin{array}{rl}
\Delta(E)= 1\otimes E+ E\otimes K, & \Delta(F)= K^{-1}\otimes F+ F\otimes 1 \\
\Delta(K) = K \otimes K, & \Delta(K^{-1}) = K^{-1}\otimes K^{-1} \\
\epsilon(E) = \epsilon(F) = 0, & \epsilon(K) = \epsilon(K^{-1}) = 1
\end{array}
\]
and an antipode is defined as follows:
\[
S(E) = EK^{-1}, S(F)=-KF,S(K)=K^{-1},S(K^{-1}) = K.
\]
This provides a {\em Hopf algebra} structure (neither commutative nor cocommutative), so that the category of modules over $\Uq$ is monoidal, namely there is a natural action over tensor products of modules given by the coproduct. 
\end{defn}

\begin{rmk}[Specialization issue]\label{specializationissue}
The {\em specialization} process of the parameter $q$ is algebraically the following. 
Let $\xi \in \BC$ be a complex number. By specialization of $q$ to the parameter $\xi$ one considers the morphism:
\[
eval: \bfct
\BQ(q) & \to & \BC \\
q & \mapsto & \xi
\efct
\]
and the following complex vector space:
\[
U_{\xi} = \BC \otimes_{eval} \Uq .
\]
By working with $q$ as a variable, one can find problems to evaluate if $\xi$ is not transcendental for instance. To define quantum topological invariants from $\Uq$-modules, we are sometimes interested in $q$ being a root of unity, for which the ground ring $\BQ(q)$ is not appropriate. 
\end{rmk}

The above remark justifies the definition of integral versions for $\Uq$, the aim of next subsection.

\begin{defn}[Integral version, {\cite[\S~9.2]{C-P}}]
Let $\Laurent_0 = \BZ\left[ q^{\pm 1} \right]$ be the ring of Laurent polynomials in the single variable $q$. An {\em integral version} for $\Uq$ is an $\Laurent_0$-subalgebra $U_{\Laurent_0}$ of $\Uq$ such that the natural map:
\[
U_{\Laurent_0} \otimes_{\Laurent_0} \BQ(q) \to \Uq
\]
is an isomorphism of $\BQ(q)$ algebras.

Then, for $\xi \in \BC^*$, the specialization of $U_{\Laurent_0}$ to $\xi$ means the following vector space:
\[
U_{\xi} = \BC \otimes_{eval} U_{\Laurent_0} \text{ with } eval : \Laurent_0 \to \BC. 
\]
\end{defn}

We introduce another version for quantum numbers. We will relate them with those from Definition \ref{quantumt} in Remark \ref{quantumtq}, later on.

\begin{defn}\label{quantumq}
Let $i$ be a positive integer. We define the following elements of $\BZ \left[ q^{\pm 1} \right]$.
\begin{equation*}
\left[ i \right]_q := \frac{q^i-q^{-i}}{q-q^{-1}} , \text{  } \left[ k \right]_q! := \prod_{i=1}^k \left[ i \right]_q , \text{  } \qbin{k}{l}_q := \frac{\left[ k \right]_q!}{\left[ k-l \right]_q! \left[ l \right]_q!} .
\end{equation*}
\end{defn}

\subsection{An integral version}\label{halfLusztigversion}

In this section, we define an integral version for $\Uq$ that will be central for the present work. This integral version is similar to the one introduced by Lusztig in \cite{Lus}. The difference is that we consider only the divided powers of $F$ as generators, not those of $E$. This version is introduced in \cite{Hab} and \cite{JK} (with subtle differences in the definitions of divided powers for $F$). We follow the one of \cite{JK}, so that we first define the divided powers, presenting a minor difference from the original ones of Lusztig. Let:
\[
F^{(n)} =  \frac{(q-q^{-1})^n}{\left[ n \right]_q!} F^n .
\]
 Let $\Laurent_0 = \BZ\left[ q^{\pm 1} \right]$ be the ring of integral Laurent polynomials in the variable $q$. 

\begin{defn}[Half integral algebra, \cite{Hab}, \cite{JK}]\label{Halflusztig}
Let $\UqhL$ be the $\Laurent_0$-subalgebra of $\Uq$ generated by $E$, $K^{\pm 1}$ and $F^{(n)}$ for $n\in \BN^*$. We call it a {\em half integral version} for $\Uq$, the word half to illustrate that we consider only half of divided powers as generators. 
\end{defn}

\begin{rmk}[Relations in $\UqhL$, {\cite[(16)~(17)]{JK}}]\label{relationsUqhL}
The relations among generators involving divided powers are the following ones:
\[
KF^{(n)}K^{-1} = q^{-2n}F^{(n)}
\]
\[
\left[ E, F^{(n+1)}  \right] = F^{(n)} \left( q^{-n} K - q^n K^{-1}  \right) \text{ and }
F^{(n)} F^{(m)} = \qbin{n+m}{n}_q F^{(n+m)} .
\]
Together with relations from Definition \ref{Uqnaif}, they complete the presentation of $\UqhL$. 

$\UqhL$ inherit a Hopf algebra structure, making its category of modules monoidal. The coproduct is given by:
\[
\Delta(K) = K \otimes K \text{ , } \Delta(E) = E \otimes K + 1 \otimes E , \text{ and } \Delta(F^{(n)}) = \sum_{j=0}^n q^{-j(n-j)}K^{j-n} F^{(j)} \otimes F^{(n-j)}. 
\]
\end{rmk}

\begin{prop}
The algebra $\UqhL$ admits the following set as an $\Laurent_0$-basis:
\[
\left\lbrace K^l E^m F^{(n)} , l \in \BZ, m,n \in \BN \right\rbrace .
\]
\end{prop}

\subsection{Verma modules and braiding}\label{VermaBraiding}

Now we define a special family of universal objects in the category of $\Uq$-modules, we express their presentation in the special case of $\UqhL$ and we give a braiding for this family of modules. Namely, the {\em Verma modules} are infinite dimensional modules which have a universal (among quantum groups) definition, and which depend on a parameter. Again, we work with this parameter as a variable with an integral ring, letting $\Laurent_1 := \BZ \left[ q^{\pm 1} , s^{\pm 1} \right]$. In \cite{JK}, they give an explicit presentation for the integral Verma-module of $\UqhL$, that we recall here.
%
\begin{defn}[Verma modules for $\UqhL$, {\cite[(18)]{JK}}]\label{GoodVerma}
Let $V^{s}$ be the Verma module of $\UqhL$. It is the infinite $\Laurent_1$-module, generated by vectors $\lbrace v_0, v_1 \ldots \rbrace$, and endowed with an action of $\UqhL$, generators acting as follows:
\[
K \cdot v_j = s q^{-2j} v_{j} \text{ , } E \cdot v_j = v_{j-1} \text{ and } F^{(n)} v_j = \left( \qbin{n+j}{j}_q \prod_{k=0}^{n-1} \left( sq^{-k-j} - s^{-1}q^{j+k} \right) \right) v_{j+n} .
\]
\end{defn}

\begin{rmk}[Weight vectors]\label{weightdenomination}
We will often make implicitly the change of variable $s := q^{\alpha}$ and denote $V^s$ by $V^{\alpha}$. This choice made to use a practical and usual denomination for eigenvalues of the $K$ action (which is diagonal in the given basis). Namely we say that vector $v_j$ is {\em of weight $\alpha - 2j$}, as $K \cdot v_j = q^{\alpha - 2j} v_j$. The notation with $s$ shows an integral Laurent polynomials structure strictly speaking. 
\end{rmk}

\begin{defn}[$R$-matrix, {\cite[(21)]{JK}}]\label{goodRmatrix}
Let $s=q^{\alpha}$ , $t=q^{\alpha'}$. The operator $q^{H \otimes H /2}$ is the following:
\[
q^{H \otimes H /2}:
\bfct
V^{s} \otimes V^{t} & \to & V^{s} \otimes V^{t}  \\
v_i \otimes v_j & \mapsto & q^{(\alpha - 2i)(\alpha'-2j)} v_i \otimes v_j 
\efct .
\]
We define the following R-matrix (see \cite[Part~2]{Kas} for the denomination):
\[
R := q^{H \otimes H/2} \sum_{n=0}^\infty q^{\frac{n(n-1)}{2}} E^n \otimes F^{(n)} .
\]
It is not yet a well defined object as the sum is infinite, but it will be well defined as an operator on Verma modules, see the following proposition (the sum always cuts off when applied to tensor product of Verma vectors). 
\end{defn}

\begin{prop}[{\cite[Theorem~7]{JK}}]\label{UqhLbraiding}
Let $V^s$ and $V^t$ be Verma modules of $\UqhL$ (with $s=q^{\alpha}$ and $t=q^{\alpha'}$). Let $\RR$ be the following operator:
\[
\RR:= q^{-\alpha \alpha' /2} T \circ R 
\]
where $T$ is the twist defined by $T(v\otimes w ) = w \otimes v$. Then $\RR$ provides a braiding for $\UqhL$ integral Verma modules. 
Namely, the morphism:
\[
Q: \bfct
\Laurent_1\left[ \Bn \right] & \to & \End_{\Laurent_1, \UqhL} \left({V^s}^{\otimes n}\right)  \\
\sigma_i & \mapsto & 1^{\otimes i-1} \otimes \RR \otimes 1^{\otimes n-i-2}
\efct
\]
is an $\Laurent_1$-algebra morphism. It provides a representation of $\Bn$ such that its action commutes with that of $\UqhL$. 
\end{prop}

\begin{rmk}\label{coloredquantum}
One can consider a braid action over $V^{s_1} \otimes \cdots \otimes V^{s_n}$ so that the morphism $Q$ is well defined but becomes multiplicative (i.e. algebra morphism) only when restricted to the pure braid group $\PBn$, so to be an endomorphism. 
\end{rmk}

\subsection{Finite dimensional braid representations}

Although braid group representations over products of Verma modules are infinite dimensional, we find finite dimensional subrepresentations using the comutativity of the braid action with the quantum structure.

\begin{rmk}\label{GoodsubrepKJ}
For $r \in \BN$:
\begin{itemize}
\item The subspace $W_{n,r} = \Ker( K - s^nq^{-2r})$ of $({V^s})^{\otimes n}$ provides a sub-representation of $\Bn$.
\item The subspace $Y_{n,r} = W_{n,r} \cap \Ker E \subset W_{n,r}$ provides a sub-representation of $\Bn$. 
\end{itemize}
We usually call $W_{n,r}$ the space of {\em subweight $r$} vectors, while $Y_{n,r}$ is called the space of {\em highest weight} vectors. 
\end{rmk}

Using the definition of the coproduct, the following remark is easily checked.

\begin{rmk}[Weight structure]\label{weightscheme}
The weight structure is managed by actions of generators: the action of $F^{(1)}$ sends an element in $W_{n,r}$ to an element in $W_{n,r+1}$ while the action of $E$ sends an element in $W_{n,r}$ to an element in $W_{n,r-1}$. Moreover, the tensor product of Verma modules is graded by weights:
\[
\left( V^{s}\right)^{\otimes n} = \bigoplus_{r\in \BN} W_{n,r} .
\]
\end{rmk}

\begin{theorem}[Irreducibility of highest weight modules, {\cite[Theorem~21]{JK}}]\label{IrreducibleJK}
The $\Bn$-representations $Y_{n,r}$ are irreducible over the fraction field $\BQ \left( q,s \right)$.
\end{theorem} 

\clearpage

\section{Homological model for $\UqhL$ Verma modules}\label{homologicalmodel}

In this section we recover quantum algebra representations in homological modules. 

\subsection{Homological action of $\UqhL$}\label{homologicalactionofUq}

We recall the scheme for the Verma module grading that is explained in Remark \ref{weightscheme}.
\[
\begin{tikzpicture}
\node (a) at (-1.2,0) {$W_{n,r}$};
\node (b) at (1.2,0) {$W_{n,r+1}$};
\draw[->] (a) to[bend right] node[below] {$F^{(1)}$} (b);
\draw[->] (b) to[bend right] node[above] {$E$} (a);
\end{tikzpicture}.
\]
The goal of this section is to construct homological operators $E, K^{\pm 1}$ and $F^{(k)}$ such that they mimic the weight structure existing on quantum Verma modules. Namely we want homological operators to fit with the following scheme:
\[
\begin{tikzpicture}
\node (a) at (-1.2,0) {$\Hrelm_r$};
\node (b) at (1.2,0) {$\Hrelm_{r+1}$};
\draw[->] (a) to[bend right] node[below] {$F^{(1)}$} (b);
\draw[->] (b) to[bend right] node[above] {$E$} (a);
\end{tikzpicture}.
\]

Definitions for homological operators were inspired by \cite{FW}. In their article, the authors define such operators acting upon a topological module built from configuration space $X_r$. The fact that their module has a homological definition remained conjectures, namely Conjecture 6.1 and 6.2 of \cite{FW}. 

The following remark, relating the two types of quantum numbers we have introduced in Definitions \ref{quantumt} and \ref{quantumq}, will be usefull for computations.

\begin{rmk}\label{quantumtq}
Let $\mt=q^{-2}$, the following relations hold in $\BZ \left[ q^{\pm 1} \right]$:
\[
(i)_{\mt}= q^{1-i} \left[ i \right]_q \text{ , } (k)_{\mt}!= q^{\frac{-k(k-1)}{2}} \left[ k \right]_q! \text{ , } {{k+l}\choose{l}}_{\mt} = q^{-kl}\qbin{k+l}{l}_q .
\]
\end{rmk}

\subsubsection{Action of $F^{(1)}$, and its divided powers.} 

We want the operator $F^{(1)} $ to go from $\Hrelm_r$ to $\Hrelm_{r+1}$, it has to increase by one the degree of a chain while passing from $X_r$ to $X_{r+1}$ for the topological space. By extension, we will build operators $F^{(k)}$, for $k> 1$ going from $\Hrelm_r$ to $\Hrelm_{r+k}$. We define them using the family $\CU$ shown to be an $\Laurentmax$-basis of the homology, although it is not difficult to define the operator without a basis, but it complicates notations. 

\begin{defn}[Divided powers of $F$]\label{actionofF}
We define the following family of homological operators:
\[
F^{(k)} : \left\lbrace \bapp
\Hrelm_r & \to & \Hrelm_{r+k} \\
 U(k_0, \ldots , k_{n-1}) & \mapsto & q^{k(1-k)/2} q^{k \sum_{i=1}^{n} \alpha_i} \left( \vcenter{\hbox{
\begin{tikzpicture}[scale=0.7, every node/.style={scale=0.7}]
\node (w0) at (-2,0) {};
\node (w1) at (-0.5,0) {};
\node[gray] at (0.0,0.0) {\ldots};
\node (wn1) at (0.5,0) {};
\node (wn) at (1.5,0) {};
\draw[dashed] (w0) -- (w1) node[above,midway] {\small $k_0$};
\draw[dashed] (wn1) -- (wn) node[above,midway] {\small $k_{n-1}$};
%
\node[gray] at (w0)[above left=2pt] {$w_0$};
\node[gray] at (w1)[above=2pt] {$w_1$};
\node[gray] at (wn1)[above=2pt] {};
\node[gray] at (wn)[above=2pt] { $w_n$};
\foreach \n in {w1,wn1,wn}
  \node at (\n)[gray,circle,fill,inner sep=2pt]{};
\node at (w0)[gray,circle,fill,inner sep=2pt]{};
\draw[thick,red] (-1.25,-0.1) -- (-1.25,-1.2) -- (-1.75,-2);
\draw[thick,red] (1,-0.1) -- (1,-1.2) -- (-1.65,-2);
\draw[dashed,postaction={decorate}] (w0) to[bend right=70] (2,0) to[bend right=70] node[midway, above] {$k$} (w0); 
\draw[red] (2,0) -- (2,-1.2) -- (-1.3,-2);
\draw[gray] (-2,0) -- (-2,1.5) -- (2.2,1.5) -- (2.2,-2) -- (-2,-2);
\draw[gray] (-2,0) -- (-2,-2);
\end{tikzpicture}
}} \right). \eapp \right.
\]
\end{defn}

\begin{rmk}
\begin{itemize}
\item In terms of homology class with coefficients in $\BZ$, involved by the union of dashed arcs corresponding to a product of simplexes, the operator $F^{(k)}$ simply adds an indexed $k$ dashed arc that rounds once along the boundary in counterclockwise direction. 
\item For the local coefficient definition, we chose to simplify the drawing by adding a straight handle, but it costs a coefficient $q^{k \sum_{i=1}^{n} \alpha_i}$ that one can removed using another more complicated family of handles. We will work with the simpler drawing and will add the coefficient ad-hoc in following computations so that we define an intermediate operator:
\[
(F')^{(k)} : \left\lbrace \bapp
\Hrelm_r & \to & \Hrelm_{r+k} \\
 U(k_0, \ldots , k_{n-1}) & \mapsto & \left( \vcenter{\hbox{
\begin{tikzpicture}[scale=0.7, every node/.style={scale=0.7}]
\node (w0) at (-2,0) {};
\node (w1) at (-0.5,0) {};
\node[gray] at (0.0,0.0) {\ldots};
\node (wn1) at (0.5,0) {};
\node (wn) at (1.5,0) {};
\draw[dashed] (w0) -- (w1) node[above,midway] {\small $k_0$};
\draw[dashed] (wn1) -- (wn) node[above,midway] {\small $k_{n-1}$};
%
\node[gray] at (w0)[above left=2pt] {$w_0$};
\node[gray] at (w1)[above=2pt] {$w_1$};
\node[gray] at (wn1)[above=2pt] {};
\node[gray] at (wn)[above=2pt] { $w_n$};
\foreach \n in {w1,wn1,wn}
  \node at (\n)[gray,circle,fill,inner sep=2pt]{};
\node at (w0)[gray,circle,fill,inner sep=2pt]{};
\draw[thick,red] (-1.25,-0.1) -- (-1.25,-1.2) -- (-1.75,-2);
\draw[thick,red] (1,-0.1) -- (1,-1.2) -- (-1.65,-2);
\draw[dashed,postaction={decorate}] (w0) to[bend right=70] (2,0) to[bend right=70] node[midway, above] {$k$} (w0); 
\draw[red] (2,0) -- (2,-1.2) -- (-1.3,-2);
\draw[gray] (-2,0) -- (-2,1.5) -- (2.2,1.5) -- (2.2,-2) -- (-2,-2);
\draw[gray] (-2,0) -- (-2,-2);
\end{tikzpicture}
}} \right), \eapp \right. 
\]
such that $F'^{(k)} = q^{k(1-k)/2} q^{k \sum_{i=1}^{n} \alpha_i} F^{(k)}$. 
\end{itemize} 
\end{rmk}

The following proposition justifies the {\em divided powers} denomination. 

\begin{prop}[Divided powers of $F$]\label{dividedpower}
There is the following relation between elements of $\Hom_{\Laurentmax} \left( \Hrelm_r , \Hrelm_{r+k} \right)$:
\[
(F^{(1)})^k = q^{k(k-1)/2} (k)_{\mt}! F^{(k)} .
\]
Let $\mt=q^{-2}$, then:
\[
(F^{(1)})^k= \left[ k \right]_q! F^{(k)}.
\]
\end{prop}
\begin{proof}
This is a direct consequence of the following equality of classes:
\begin{equation*}
\left( \vcenter{\hbox{\begin{tikzpicture}[scale=0.7,decoration={
    markings,
    mark=at position 0.5 with {\arrow{>}}}
    ]
\node[gray,left] (P) at (-2,0) {$w_0$};
\draw (-2,0) to[bend right=30] node[pos=0.7,above] (k) {}  (-0.5,-1.65);
\draw[postaction={decorate}] (-0.5,-1.65) to[bend left=-55](1.8,0);
\draw (-2,0) to[bend left=30] (-0.5,1.65);
\draw (-0.5,1.65) to[bend right=-55] (1.8,0);
\draw[red] (k) to (-1.6,-1.5);
\node[red,left] at (-1.6,-1.3) {};

\node at (2.1,0) {$\ldots$};
\node at (2.2,0.3) {$k$};

\draw (-2,0) to[bend right=30] node[pos=0.9,above] (k1) {}  (-0.5,-2);
\draw[postaction={decorate}] (-0.5,-2) to[bend left=-55](2.5,0);
\draw (-2,0) to[bend left=30] (-0.5,2);
\draw (-0.5,2) to[bend right=-55] (2.5,0);
\draw[red] (k1) to (-1.6,-2);

\draw[gray] (-2,0) -- (-2,0.7);
\draw[gray] (-2,0) -- (-2,-0.7);
\end{tikzpicture} }} \right)
= (k)_{\mt}!
\left( \vcenter{\hbox{\begin{tikzpicture}[scale=0.7,decoration={
    markings,
    mark=at position 0.5 with {\arrow{>}}}
    ]
\node[gray,left] (P) at (-2,0) {$w_0$};
\draw[dashed] (-2,0) to[bend right=30] node[pos=0.7,above] (k) {}  (-0.5,-1.65);
\draw[dashed,postaction={decorate}] (-0.5,-1.65) to[bend left=-55](1.8,0);
\draw[dashed] (-2,0) to[bend left=30] (-0.5,1.65);
\draw[dashed] (-0.5,1.65) to[bend right=-55] node[above,midway] {$k$} (1.8,0);
\draw[double,red] (k) to (-1.6,-1.3);
\node[red,left] at (-1.6,-1.3) {};

\draw[gray] (-2,0) -- (-2,0.7);
\draw[gray] (-2,0) -- (-2,-0.7);
\end{tikzpicture} }} \right).
\end{equation*}
which can be proved as Corollary \ref{1forkto1code}, and whatever stands inside the circles. On the right, there are $k$ parallel arcs rounding along the boundary in counterclockwise direction, while on the left there is one dashed arc rounding along the boundary. This shows that $F'^k = (k)_{\mt} ! F'^{(k)}$, and the first statement is immediate. To get the second equality, for $\mt=q^{-2}$ one uses directly Remark \ref{quantumtq}. 
\end{proof}

\subsubsection{Actions of $E$ and $K$} 

To define the action of $E\in \Hom_{\Laurentmax} \left( \Hrelm_r, \Hrelm_{r-1} \right)$ we need a way to remove one configuration point. This is the purpose of morphisms defined in the following definition.

\begin{defn}\label{addingpoint}
\begin{itemize}
\item Let $\psi^r$ be the following homeomorphism:
\[
\psi^r : \bfct
X_r \setminus X_r^- & \to & X_{r+1}^- \\
Z & \mapsto & Z \cup {w_0} \\
{\pmb \xi^r}  & \mapsto & \lbrace \xi_1 , \ldots , \xi_r , w_0 \rbrace 
\efct .
\]
\item $\psi^r$ induces:
\[
\psi^r_* : \pi_1(X_r \setminus X_r^-, \pmb \xi^r ) \to \pi_1(X_{r+1}^-,\lbrace \pmb \xi^r , w_0 \rbrace ) .
\]
We provide a natural way to transport the base point on the right to ${\pmb \xi^{r+1}}$, namely we move $w_0$ along $\partial D_n$ through a path $\varphi^r$ defined as follows:
\[
\varphi^r: \bfct
I & \to & X_{r+1} \\
t & \mapsto & \varphi^r(t) = \lbrace \varphi_1(t) , \xi_r, \ldots , \xi_1 \rbrace
\efct
\]
where $\varphi_1$ goes from  $\xi_{r+1}$ to $w_0$ along $\partial D_n$ in the clockwise direction, while other coordinates remain fixed in ${\pmb \xi^{r}}$. 
\item We let then $\Phi^r$ be the composition of the above $\psi^r_*$ with the morphism induced by the change of base point through conjugation by $\varphi^r$.
\[
\Phi^r: \bapp
\pi_1(X_r \setminus X_r^-, {\pmb \xi^r}) & \to & \pi_1(X_{r+1}, {\pmb \xi^{r+1}}) 
\eapp .
\]
\end{itemize} 

In what follows we will often omit the indices $r$ in $\varphi^r$, ${\pmb \xi^r}$ and $\Phi^r$, to simplify notations when no confusion is possible. 

\end{defn}

\begin{lemma} \label{localsystemiso}
The morphism $\Phi^r$ lifted to the local system level:
\[
\Phi^r: L_r\restriction_{X_r \setminus X_r^-} \to L_{r+1}\restriction_{X_{r+1}^-}
\]
is an isomorphism of local systems.
\end{lemma}
\begin{proof}
Underlying space are homeomorphic through $\psi$ (addition of $w_0$). Let $\rho_r$ be the representation of $\pi_1 (X_r , {\pmb \xi^r} )$ providing the local system $L_r$. The following diagram is commutative:
\begin{equation*}
\begin{tikzcd}
 \pi_1 (X_r \setminus X_r^- , {\pmb \xi^r} ) \arrow[r,"\Phi^r"] \arrow[d,"\rho_r"]
   & \pi_1 (X_{r+1} , {\pmb \xi^{r+1}} )  \arrow[d,"\rho_{r+1}"]\\ 
\BZ^{n+1} = \bigoplus_{i \in \lbrace 1 , \ldots ,n \rbrace} \BZ \langle q^{\alpha_i} \rangle \oplus \BZ \langle t \rangle \arrow[r,"\Id"]
   & \bigoplus_{i \in \lbrace 1 , \ldots ,n \rbrace} \BZ \langle q^{\alpha_i} \rangle \oplus \BZ \langle t \rangle
\end{tikzcd}
\end{equation*}
which proves the lemma. The commutation is easy to verify thinking of the representation of $\pi_1 (X_r \setminus X_r^- , {\pmb \xi^r} )$ given in  Remark \ref{pi_1X_r}. The morphism $\Phi^r$ simply adds a straight strands to the braid, not modifying its image by $\rho_{r}$. 
\end{proof}

\begin{rmk}\label{localsystemisoHomologic}
We formulate the above Lemma \ref{localsystemiso} for homologies. In other words, the choice of path $\varphi$ in Definition \ref{addingpoint} yields the following isomorphism:
\[
\Phi^r: \Hlf_r ({X_r \setminus X_r^-};L_r) \to \Hlf_r(X_{r+1}^-;L_{r+1}).
\]
Let an element in $\Hlf_r ({X_r \setminus X_r^-};L_r)$ being given by a couple $(\left[ \Delta \right], h)$ where $\left[ \Delta \right]$ is the class of a chain in $\Hlf_r ({X_r \setminus X_r^-};\BZ)$ and $h$ a path relating $\pmb{\xi^r}$ to $\Delta$ (case of interest). Then its image by $\Phi^r$ is determined by the couple $\left( \left[ \lbrace \Delta, w_0 \rbrace \right], \lbrace h,w_0 \rbrace \circ \varphi^r \right)$. One must pay attention to the fact that the isomorphism between homologies depends on the choice of $\varphi$, so does the operator $E$ defined in Definition \ref{actionofE} below. 
\end{rmk}

\begin{rmk} \label{identificationHr1}
This remark is a recall. We have the following equality:
\[
H_{r-1}(X_{r-1} \setminus X^-_{r-1}; L_{r-1}) = H_{r-1}(X_{r-1}(w_0);L_{r-1}) = \Hrelm_{r-1} . 
\]
where $X_r(w_0)$ is the space of configurations of $X_r$ without coordinate in $w_0$. The first equality is the fact that $X_{r-1} \setminus X^-_{r-1}$ and $X_{r-1}(w_0)$ are canonically homeomorphic. The second one is Corollary \ref{HrelStruc}. 
\end{rmk}

From this identification one is able to define an operator $E$ as in the following definition. 

\begin{defn}[Action of $E$] \label{actionofE}
Let $E$ be the operator defined as follows:
\begin{equation*}
\begin{tikzcd}
E: \Hrelm_r \arrow[r,"\partial_*"] & H_{r-1}(X_r^-;L_r) \arrow[r,"\left({\Phi^r}\right)^{-1}"] & H_{r-1}(X_{r-1} \setminus X^-_{r-1}; L_{r-1}) = \Hrelm_{r-1}. 
\end{tikzcd}
\end{equation*}
The arrow $\partial_*$ is the boundary map of the exact sequence of the pair $(X_r, X_r^-)$. The arrow $\left({\Phi^r}\right)^{-1}$ is the inverse isomorphism provided by Lemma \ref{localsystemiso} (see Remark \ref{localsystemisoHomologic}) and the last equality is the above Remark \ref{identificationHr1}.
\end{defn}

\begin{rmk}\label{Ecalculelebord}
The definition of $E$ is the boundary map of the relative exact sequence of the pair involved, the rest are just isomorphic identifications of homology modules. Namely, the operator $E$ reads the part of the boundary that lies in $X_r^-$. 
\end{rmk}
We give a first example of computation with a standard code sequence.

\begin{example}[Action of $E$ on a code sequence] \label{Ecodeseq}
Let ${\bf k} = (k_0 , \ldots , k_{n-1}) \in \EZnr$, and $U_{\bf k}$ its associated standard code sequence. One can check the following property:
\[
E \cdot U_{\bf k} = U(k_0 -1 , \ldots , k_{n-1}).
\]
Consider first $U(k_0 , 0 , \ldots , 0)$ and let $\phi^{k_0}$ be the chain associated with the indexed $k_0$ dashed arc. We recall our parametrization of the standard simplex:
\[
\Delta^{k_0} = \lbrace 0 \leq t_1 \leq \cdots \leq t_{k_0} \leq 1 \rbrace
\]
so that its only boundary part sent to configurations with one coordinate in $w_0$ is $\lbrace t_1 = 0 \rbrace \in \Delta^k$. We remark that $\phi^{k_0}$ restricted to $\lbrace t_1 = 0 \rbrace$ is $\phi^{k_0-1}$ (chain associated with same dashed arc but indexed by $k_0-1$). This by shifting left the parametrization: $(0,t_2,\ldots,t_{k_0}) \mapsto (t_2,\ldots,t_{k_0})$ (which does not involve any permutation which could change orientation). Consequently, one sees that the equality holds at the level of homology over $\BZ$. 

To deal with the handle rule lifting process, we remark that only the leftmost configuration point embedded in $U(k_0 , \ldots , k_{n-1})$ can join $w_0$. This is saying that the only part of the boundary of $U(k_0 , \ldots , k_{n-1})$ lying in $X_r^-$ corresponds to the leftmost point being in $w_0$. No local coefficient appears while applying $(\Phi^r)^{-1}$ (Lemma \ref{localsystemiso}) thanks to the fact that the handle joining the leftmost configuration point is the leftmost handle, and it joins $\xi_r$, namely the leftmost base point's coordinate. Another way to say this is by remarking that the path following the leftmost handle, then going to $w_0$ along $U_{\bf k}$ then back to $\xi_r$ along the boundary can be homotoped to $w_0$ without perturbing other handles. In other words, composition (of the path corresponding to red handles) with the inverse of path $\varphi^r$ (Definition \ref{addingpoint}) does not involve any change of local coordinate.  
\end{example}

The action of the operator $K$ is a diagonal action encoding the value of $r$.
\begin{defn}[Action of $K$] \label{actionofK}
For $r \in \BN$, the operator $K$ is the following diagonal action over $\Hrelm_r$:
\[
K = q^{ \sum_i \alpha_i} \mt^{r} \Id_{\Hrelm_r}.  
\]
We define the operator $K^{-1}$ to be the inverse of $K$. 
\end{defn}

\subsubsection{Homological $\Uq$ representation.}\label{sectionhomologicalhabiro}

Let $\CH = \bigoplus_{r\in \BN} \Hrelm_r$, the actions of $E, F^{(1)} $ and $K$ are endomorphisms of $\CH$. We have the following proposition.

\begin{prop}\label{relationeFK}
The operators $E,F^{(1)} $ and $K$ satisfy the following relations:
\[
KE = \mt^{-1} EK \text{ , } KF^{(1)}  =  \mt F^{(1)} K \text{ and } \left[ E,F^{(1)}  \right]  =  K - K^{-1} .
\]
\end{prop}
\begin{proof}
The first two relations are direct consequences of both facts that $F^{(1)} $ increases $r$ by one, $E$ decreases it by one, and of the (diagonal) definition of $K$. It remains to prove the last one. The proof can be performed without considering a basis of $\CH$, although we do it here using the basis of code sequences for an easier reading. Let $r \in \BN$, we recall that $\CU = \left( U_{\bf k} \right)_{{\bf k} \in \EZnr}$ is a basis of $\Hrelm_r$ as an $\Laurentmax$-module. Let ${\bf k } = (k_0 , \ldots , k_{n-1}) \in \EZnr$. First we compute the commutation between $E$ and $F'$ before renormalizing $F'$ to $F^{(1)} $. The class $F' \cdot \left( U_{\bf k} \right)$ corresponds to the following one:
\begin{equation*}
F'^{(1)} \cdot U_{\bf k} = \vcenter{\hbox{\begin{tikzpicture}[scale=0.7]
\node (w0) at (-5,0) {};
\node (w1) at (-3,0) {};
\node (w2) at (-1,0) {};
\node[gray] at (0.0,0.0) {\ldots};
\node (wn1) at (1,0) {};
\node (wn) at (2.5,0) {};  
\draw[dashed] (w0) -- (w1) node[midway,above] {$k_0$};
\draw[dashed] (w1) -- (w2) node[midway,above] {$k_1$};
\draw[dashed] (wn1) -- (wn) node[midway,above] {$k_{n-1}$};
%
\node[gray] at (w0)[left=5pt] {$w_0$};
\node[gray] at (w1)[above=5pt] {$w_1$};
\node[gray] at (w2)[above=5pt] {$w_2$};
\node[gray] at (wn1)[above=5pt] {$w_{n-1}$};
\node[gray] at (wn)[above=5pt] {$w_n$};
\foreach \n in {w1,w2,wn1,wn}
  \node at (\n)[gray,circle,fill,inner sep=3pt]{};
\node at (w0)[gray,circle,fill,inner sep=3pt]{};
\draw[double,thick,red] (-4,-0.1) -- (-4,-2);
\draw[double,thick,red] (1.75,-0.1) -- (1.75,-2);
\draw[double,thick,red] (-2,-0.1) -- (-2,-2);
\draw[dashed,gray] (-5,-2) -- (4,-2);
\draw[dashed,gray] (4,-2) -- (4,-3);
\node[gray,circle,fill,inner sep=0.8pt] at (-4.8,-3) {};
\node[below,gray] at (-4.8,-3) {$\xi_{r+1}$};
\node[gray,circle,fill,inner sep=0.8pt] at (-3.9,-3) {};
\node[below,gray] at (-3.9,-3) {$\xi_2$};
\node[gray,circle,fill,inner sep=0.8pt] at (-3.2,-3) {};
\node[below,gray] at (-3.2,-3) {$\xi_{1}$};
\draw[postaction={decorate}] (w0) to[bend right=90] (3.5,0) to[bend right=90] (w0); 
\draw[red] (3.5,0) -- (3.5,-2);
\draw[red] (-4.8,-3) -- (-4,-2);
\draw[red] (-3.9,-3) -- (1.75,-2);
\draw[red] (-3.2,-3) to[bend right=5] (3.5,-2);
\node[red] at (-4,-2.8) {$\ldots$};
\draw[gray] (-5,0) -- (-5,1);
\draw[gray] (-5,0) -- (-5,-3);
\draw[gray] (-5,-3) -- (4.5,-3) node[right] {$\partial D_n$};
\end{tikzpicture}}}
\end{equation*}

Applying $E$ to this class gives the part of its boundary lying in $X_r^-$. There are $r+1$ points embedded in this class, $r$ of them in the dashed arcs, and the last one in the plain arc. The part of the boundary lying in $X_r^-$ is the sum of:
\begin{itemize}
\item the leftmost point of dashed arcs going to $w_0$ (i.e. given by one boundary component from the simplex defined by the leftmost dashed arc, the image of $\lbrace t_1 = 0 \rbrace$ where $t_1$ is the first parameter of the latter simplex),
\item and of the two boundary parts corresponding to ``back and front faces'' parametrized by the plain arc (the image of $\lbrace t=
0\rbrace$ and $\lbrace t=1 \rbrace$ where $t$ is the coordinate sent to the plain arc)\end{itemize}
This corresponds to the following equality.

\begin{align*}
E \cdot \left(\vcenter{\hbox{
\begin{tikzpicture}
\node (w0) at (-2,0) {};
\node (w1) at (-0.5,0) {};
\node[gray] at (0.0,0.0) {\ldots};
\node (wn1) at (0.5,0) {};
\node (wn) at (1.5,0) {};  
\draw[dotted] (w0) -- (w1) node[midway] {\small $k_0$};
\draw[dotted] (wn1) -- (wn) node[midway] {\small $k_{n-1}$};
%
\node[gray] at (w1)[above=2pt] {\tiny $w_1$};
\node[gray] at (wn1)[above=2pt] {\tiny $w_{n-1}$};
\node[gray] at (wn)[above=2pt] {\tiny $w_n$};
\foreach \n in {w1,wn1,wn}
  \node at (\n)[gray,circle,fill,inner sep=2pt]{};
\node at (w0)[gray,circle,fill,inner sep=2pt]{};
\draw[double,thick,red] (-1.25,-0.1) -- (-1.25,-1.2);
\draw[double,thick,red] (1,-0.1) -- (1,-1.2);
\draw[postaction={decorate}] (w0) to[bend right=70] (2,0) to[bend right=70] (w0); 
\draw[red] (2,0) -- (2,-1.2);
\draw[gray] (-2,0) -- (-2,0.5);
\draw[gray] (-2,0) -- (-2,-1);
\end{tikzpicture}
}} \right) = & 
\left(\vcenter{\hbox{
\begin{tikzpicture}
\node (w0) at (-2,0) {};
\node (w1) at (-0.5,0) {};
\node[gray] at (0.0,0.0) {\ldots};
\node (wn1) at (0.5,0) {};
\node (wn) at (1.5,0) {};
\draw[dotted] (w0) -- (w1) node[midway] {\small $k_0-1$};
\draw[dotted] (wn1) -- (wn) node[midway] {\small $k_{n-1}$};
%
\node[gray] at (w1)[above=2pt] {\tiny $w_1$};
\node[gray] at (wn1)[above=2pt] {\tiny $w_{n-1}$};
\node[gray] at (wn)[above=2pt] {\tiny $w_n$};
\foreach \n in {w1,wn1,wn}
  \node at (\n)[gray,circle,fill,inner sep=2pt]{};
\node at (w0)[gray,circle,fill,inner sep=2pt]{};
\draw[double,thick,red] (-1.25,-0.1) -- (-1.25,-1.2);
\draw[double,thick,red] (1,-0.1) -- (1,-1.2);
\draw[postaction={decorate}] (w0) to[bend right=70] (2,0) to[bend right=70] (w0); 
\draw[red] (2,0) -- (2,-1.2);
\draw[gray] (-2,0) -- (-2,0.5);
\draw[gray] (-2,0) -- (-2,-1);
\end{tikzpicture}
}} \right) 
+C \times U\left( k_0 , \ldots , k_{n-1} \right)
\end{align*}
where the coefficient $C$ is the computation of the relative boundary part coming from the plain arc. One sees that:
\begin{equation*}
\left(\vcenter{\hbox{
\begin{tikzpicture}
\node (w0) at (-2,0) {};
\node (w1) at (-0.5,0) {};
\node[gray] at (0.0,0.0) {\ldots};
\node (wn1) at (0.5,0) {};
\node (wn) at (1.5,0) {};
\draw[dotted] (w0) -- (w1) node[midway] {\small $k_0-1$};
\draw[dotted] (wn1) -- (wn) node[midway] {\small $k_{n-1}$};
%
\node[gray] at (w1)[above=2pt] {\tiny $w_1$};
\node[gray] at (wn1)[above=2pt] {\tiny $w_{n-1}$};
\node[gray] at (wn)[above=2pt] {\tiny $w_n$};
\foreach \n in {w1,wn1,wn}
  \node at (\n)[gray,circle,fill,inner sep=2pt]{};
\node at (w0)[gray,circle,fill,inner sep=2pt]{};
\draw[double,thick,red] (-1.25,-0.1) -- (-1.25,-1.2);
\draw[double,thick,red] (1,-0.1) -- (1,-1.2);
\draw[postaction={decorate}] (w0) to[bend right=70] (2,0) to[bend right=70] (w0); 
\draw[red] (2,0) -- (2,-1.2);
\draw[gray] (-2,0) -- (-2,0.5);
\draw[gray] (-2,0) -- (-2,-1);
\end{tikzpicture}
}} \right) = F' \cdot \left( E \cdot U(k_0, \ldots , k_{n-1}) \right) 
\end{equation*}
using Example \ref{Ecodeseq}. We also mention that this term is zero if $k_0=0$. This gives:
\[
\left[ E, F' \right] \cdot U(k_0 , \ldots , k_{n-1}) = C \times U(k_0 , \ldots , k_{n-1}) 
\]
so that it remains to compute the coefficient $C$. The coefficient $C$ is the difference $C_1 - C_2$ where $C_1$ and $C_2$ satisfy the following equations:
\begin{align*}
\left(\vcenter{\hbox{
\begin{tikzpicture}
\node (w0) at (-2,0) {};
\node (w1) at (-0.5,0) {};
\node[gray] at (0.0,0.0) {\ldots};
\node (wn1) at (0.5,0) {};
\node (wn) at (1.5,0) {};  
\draw[dotted] (w0) -- (w1) node[midway] {\tiny $k_0$};
\draw[dotted] (wn1) -- (wn) node[midway] {\tiny $k_{n-1}$};
%
\node[gray] at (w1)[above=2pt] {\tiny $w_1$};
\node[gray] at (wn1)[above=2pt] {\tiny $w_{n-1}$};
\node[gray] at (wn)[above=2pt] {\tiny $w_n$};
\foreach \n in {w1,wn1,wn}
  \node at (\n)[gray,circle,fill,inner sep=2pt]{};
\node at (w0)[gray,circle,fill,inner sep=2pt]{};
\draw[double,thick,red] (-1.25,-0.1) -- (-1.25,-1.2);
\draw[double,thick,red] (1,-0.1) -- (1,-1.2);
\draw[postaction={decorate}] (w0) to[bend right=70] (2,0) to[bend right=70] (w0); 
\draw[red] (2,0) -- (2,-1.2);
\draw[gray] (-2,0) -- (-2,0.5);
\draw[gray] (-2,0) -- (-2,-1);
\end{tikzpicture}
}} \right) = C_1 
\left(\vcenter{\hbox{
\begin{tikzpicture}
\node (w0) at (-2,0) {};
\node (w1) at (-0.5,0) {};
\node[gray] at (0.0,0.0) {\ldots};
\node (wn1) at (0.5,0) {};
\node (wn) at (1.5,0) {};  
\draw[dotted] (w0) -- (w1) node[midway] {\tiny $k_0$};
\draw[dotted] (wn1) -- (wn) node[midway] {\tiny $k_{n-1}$};
%
\node[gray] at (w1)[above=2pt] {\tiny $w_1$};
\node[gray] at (wn1)[above=2pt] {\tiny $w_{n-1}$};
\node[gray] at (wn)[above=2pt] {\tiny $w_n$};
\foreach \n in {w1,wn1,wn}
  \node at (\n)[gray,circle,fill,inner sep=2pt]{};
\node at (w0)[gray,circle,fill,inner sep=2pt]{};
\draw[double,thick,red] (-1.25,-0.1) -- (-1.25,-1.2);
\draw[double,thick,red] (1,-0.1) -- (1,-1.2);
\draw[postaction={decorate}] (w0) to[bend right=70] (2,0) to[bend right=70] (w0); 
\draw[red] (-1.6,-0.65) -- (-1.6,-1.2);
\draw[gray] (-2,0) -- (-2,0.5);
\draw[gray] (-2,0) -- (-2,-1);
\end{tikzpicture}
}} \right) \\
\left(\vcenter{\hbox{
\begin{tikzpicture}
\node (w0) at (-2,0) {};
\node (w1) at (-0.5,0) {};
\node[gray] at (0.0,0.0) {\ldots};
\node (wn1) at (0.5,0) {};
\node (wn) at (1.5,0) {};  
\draw[dotted] (w0) -- (w1) node[midway] {\tiny $k_0$};
\draw[dotted] (wn1) -- (wn) node[midway] {\tiny $k_{n-1}$};
%
\node[gray] at (w1)[above=2pt] {\tiny $w_1$};
\node[gray] at (wn1)[above=2pt] {\tiny $w_{n-1}$};
\node[gray] at (wn)[above=2pt] {\tiny $w_n$};
\foreach \n in {w1,wn1,wn}
  \node at (\n)[gray,circle,fill,inner sep=2pt]{};
\node at (w0)[gray,circle,fill,inner sep=2pt]{};
\draw[double,thick,red] (-1.25,-0.1) -- (-1.25,-1.2);
\draw[double,thick,red] (1,-0.1) -- (1,-1.2);
\draw[postaction={decorate}] (w0) to[bend right=70] (2,0) to[bend right=70] (w0); 
\draw[red] (2,0) -- (2,-1.2);
\draw[gray] (-2,0) -- (-2,0.5);
\draw[gray] (-2,0) -- (-2,-1);
\end{tikzpicture}
}} \right) = C_2 
\left(\vcenter{\hbox{
\begin{tikzpicture}
\node (w0) at (-2,0) {};
\node (w1) at (-0.5,0) {};
\node[gray] at (0.0,0.0) {\ldots};
\node (wn1) at (0.5,0) {};
\node (wn) at (1.5,0) {};  
\draw[dotted] (w0) -- (w1) node[midway] {\tiny $k_0$};
\draw[dotted] (wn1) -- (wn) node[midway] {\tiny $k_{n-1}$};
%
\node[gray] at (w1)[above=2pt] {\tiny $w_1$};
\node[gray] at (wn1)[above=2pt] {\tiny $w_{n-1}$};
\node[gray] at (wn)[above=2pt] {\tiny $w_n$};
\foreach \n in {w1,wn1,wn}
  \node at (\n)[gray,circle,fill,inner sep=2pt]{};
\node at (w0)[gray,circle,fill,inner sep=2pt]{};
\draw[double,thick,red] (-1.25,-0.1) -- (-1.25,-1.2);
\draw[double,thick,red] (1,-0.1) -- (1,-1.2);
\draw[postaction={decorate}] (w0) to[bend right=70] (2,0) to[bend right=70] (w0); 
\draw[red] (-1.6,0.65) -- (-1.6,-1.2);
\draw[gray] (-2,0) -- (-2,0.5);
\draw[gray] (-2,0) -- (-2,-1);
\end{tikzpicture}
}} \right).
\end{align*}
This from the same handle argument as in Example \ref{Ecodeseq} (application of path $\varphi^{-1}$ from Definition \ref{addingpoint}). We compute these coefficients using the handle rule. The coefficient $C_1$ corresponds to the local system coefficient of the following braid:

\begin{figure}[h!]
\begin{center}
\def\svgwidth{0.5\columnwidth}
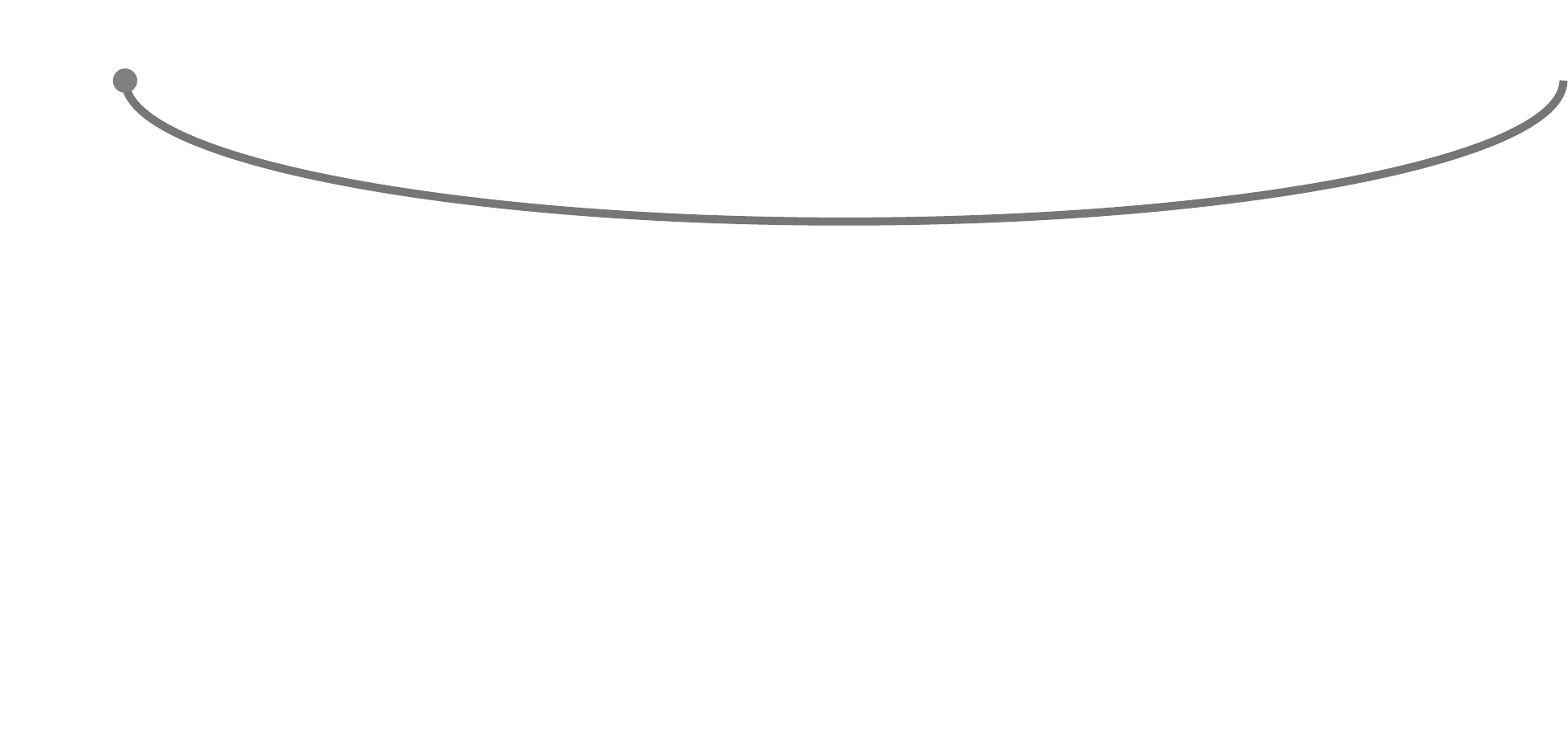
\end{center}
\end{figure}
while $C_2$ corresponds to the same braid but with the red front strand passing in the back. We emphasize that in the braid picture we got rid of parts of red handles lying outside parentheses. Outside the parentheses, the paths consist in going to the base point without crossing each other staying in front of the $w_i$'s, so that upper and lower the box, the contributions to the handle local system coefficient balance each other. Then it is straightforward to compute the local system coefficient of these braids, we get:
\begin{align*}
C_1= \mt^{\left. \sum_{i=0}^{n-1} k_i \right.} = \mt^r  \text{ }  & \text{, } C_2= \mt^{-r} q^{-2 \left. \sum_{i=1}^{n} \alpha_i \right.} 
\end{align*}
so that:
\[
\left[ E, F' \right] \cdot U(k_0 , \ldots , k_{n-1}) = \left({\mt}^r - {\mt}^{-r} q^{-2 \left. \sum_{i=1}^{n} \alpha_i \right.}\right) \times U(k_0 , \ldots , k_{n-1}) .
\]
We recall that:
\[
\left[ E, F^{(1)}  \right]  = q^{\sum \alpha_i} \left[ E, F' \right]
\]
which concludes the proof. 
\end{proof}

\begin{thm}\label{homologicalHabiro}
Let $q^{-2}=\mt$. The infinite module $\CH$ together with the above described action of $E, F^{(1)}, K^{\pm 1 }$ and $F^{(k)}$ for $k \ge 2$ yields a representation of the integral algebra $\UqhL$.
\end{thm}
\begin{proof}
The algebra $\UqhL$ is presented in Section \ref{halfLusztigversion}, Definition \ref{Halflusztig}. We use same notations (from Section \ref{halfLusztigversion}) for generators and we recover the same relations. Namely, the relations between $E$, $F^{(1)}$ and $K^{\pm 1 }$ are recovered using Proposition \ref{relationeFK}, while the fact that $F^{(k)}$ are the so called divided powers of $F^{(1)} $, see Proposition \ref{dividedpower}, ensures that the relations involving them hold. 
\end{proof}

\begin{rmk}
Even if it is not necessary to prove them knowing Proposition \ref{dividedpower} (divided power property), we can check homologically the relations involving the divided powers of $F^{(1)}$ (relations introduced in Remark \ref{relationsUqhL}). Namely:
\[
\left[ E, F^{(n+1)}  \right] = F^{(n)} \left( q^{-n} K - q^n K^{-1}  \right)
\]
is a simple computation of the relative boundary of a class as in the proof of Proposition \ref{relationeFK}. While:
\[
F^{(n)} F^{(m)} = \qbin{n}{n+m}_q F^{(n+m)}
\]
is a direct consequence of the homological Corollary \ref{dashedtodashed}. 

We have a complete homological description of the relations holding in $\UqhL$. 
\end{rmk}

\begin{rmk}
Using Proposition \ref{relationeFK}, one sees that we have a representation of the {\em simply connected} rational version of $\Uq$, for which are introduced generators that correspond to square roots of $K$ and $K^{-1}$. See \cite[\S~9]{DCP}, Remark 2.2 of \cite{Bas}, or \cite[\S~9.1]{C-P} for information about this version of $\Uq$. 
\end{rmk}


\subsection{Computation of the $\UqhL$-action}\label{computationUqaction}

In this section we compute the action of the operators $E, F^{(1)} $ and $K$ in the basis of multi arcs, in order to recognize the representation of $\Uq$ obtained over $\CH$. First we define a normalized version of the multi-arc basis.

\begin{defn}[Normalized multi-arcs.]\label{NormMultiA}
Let ${\bf k} \in  \EZnr$, and $A(k_0,\ldots , k_{n-1})$ be the following element of $\Hrelm_r$:
\[
A({\bf k}) = q^{\alpha_1(k_1 + \cdots + k_{n-1} ) + \alpha_2 ( k_2+ \cdots +k_{n-1}) + \cdots + \alpha_{n-1} k_{n-1} }A'({\bf k}). 
\]
Let $\CA = \left( A({\bf k}) \right)_{{\bf k} \in \EZnr}$ be the corresponding family indexed by $\EZnr$. 
By convention, $A(k_0 , \ldots , k_{n-1})$ is defined to be $0 \in \Hrelm_r$ when ever $k_i = -1$ for some $i \in \lbrace 0,\ldots , n-1 \rbrace$.
\end{defn}

\begin{rmk}
The family $\CA$ is obtained from $\CA'$ by a diagonal matrix of invertible coefficients in $\Laurentmax$ so that $\CA$ is still a basis of $\Hrelm_r$ as an $\Laurentmax$-module. As for the definition of divided powers of $F$, we could have chosen to avoid the normalization coefficient but to draw more complicated handles. In following computations we will work with $\CA'$ drawings and add the coefficient ad-hoc to work with the family $\CA$. 
\end{rmk}

We are going to compute the action of operators in this basis, and we will see that it recovers the basis of $\UqhL$ Verma-modules. 

\subsubsection{Action of $E$.}

First we need a lemma to reorganize handles.

\begin{lemma}
Let $A'(k_0 ,\ldots , k_{n-1})$ be the standard multi-arc associated with $(k_0 ,\ldots,k_{n-1}) \in \EZnr$. For $i=1,\ldots ,n$, there is the following relation holding in $\Hrelm_r$:
\begin{align*}
\left(\vcenter{\hbox{
\begin{tikzpicture}[decoration={
    markings,
    mark=at position 0.5 with {\arrow{>}}}
    ]
\node (w0) at (-1,0) {};
\node (w1) at (0,0) {};
\node[gray] at (0.5,0) {\tiny $\ldots$};
\node (wn) at (2,0) {};
\node[gray] at (1.6,0) {\tiny $\ldots$};
\node (wi) at (1,0) {};
\coordinate (a) at (-2,-1.3);
\draw[dashed] (w0) to node[above,pos=0.5] (km) {\small $k_0$} (w1);
\draw[dashed] (w0) to[bend right=50] node[pos=0.85] (k0) {\small $k_{n-1}$} (wn);
\draw[dashed] (w0) to[bend right=30] node[pos=0.7] (k1) {\small $k_{i-1}$} (wi);
\node[gray] at (w0)[left=5pt] {$w_0$};
\node[gray] at (wn)[above=5pt] {$w_n$};
\node[gray] at (wi)[above=5pt] {$w_{i}$};
\foreach \n in {w0,wn,wi,w1}
  \node at (\n)[gray,circle,fill,inner sep=3pt]{};
\draw[double,red,thick] (km) -- (km|-a);
\draw[double,red,thick] (k0) -- (k0|-a);
\draw[double,red,thick] (k1) -- (k1|-a);
\draw[gray,thick] (-1,0.3) -- (w0) -- (-1,-1.6);
\draw[dashed,gray] (-1,-1.3) -- (2,-1.3) -- (2,-1.6);
\end{tikzpicture}
}} \right)
& =  \mt^{k_0+ \cdots + k_{i-2}}
\left(\vcenter{\hbox{
\begin{tikzpicture}[decoration={
    markings,
    mark=at position 0.5 with {\arrow{>}}}
    ]
\node (w0) at (-1,0) {};
\node (w1) at (0,0) {};
\node[gray] at (0.5,0) {\tiny $\ldots$};
\node (wn) at (2,0) {};
\node[gray] at (1.6,0) {\tiny $\ldots$};
\node (wi) at (1,0) {};
\coordinate (a) at (-2,-1.3);
\draw[dashed] (w0) to node[above,pos=0.8] (km) {\small $k_0$} (w1);
\draw[dashed] (w0) to[bend right=50] node[pos=0.85] (k0) {\small $k_{n-1}$} (wn);
\draw[dashed] (w0) to[bend right=30] node[pos=0.7] (k1) {\small $k_{i-1}$} (wi);
\node[gray] at (w0)[left=5pt] {$w_0$};
\node[gray] at (wn)[above=5pt] {$w_n$};
\node[gray] at (wi)[above=5pt] {$w_{i}$};
\foreach \n in {w0,wn,wi,w1}
  \node at (\n)[gray,circle,fill,inner sep=3pt]{};
\draw[double,red,thick] (km) -- (km|-a);
\draw[double,red,thick] (k0) -- (k0|-a);
\node (solo) at (-0.5,-0.1) {};
\draw[red] (solo) -- (solo|-a);
\draw[double,red,thick] (k1) -- (k1|-a);
\draw[gray,thick] (-1,0.3) -- (w0) -- (-1,-1.6);
\draw[dashed,gray] (-1,-1.3) -- (2,-1.3) -- (2,-1.6);
\end{tikzpicture}
}} \right)
\end{align*}
where, in the right term, only one component of the red tube indexed by $k_i$ had been moved to the extreme left of other red handles. Namely only the leftmost handle composing the $(k_i)$-handle (tube of $k_i$ parallel handles) had been pushed to the left of the $(k_0)$-handle.  Down the parentheses, red handles are joining the base point following a usual dashed box, without crossing with each other. The left class follows this box:
\begin{equation*}
\begin{tikzpicture}[scale=0.65]
\node (w0) at (-5,1) {};
\node (w1) at (-1,1) {};
\node (w2) at (1,1) {};
\node[gray] at (2.0,1.0) {\ldots};
\node (wn1) at (3,1) {};
\node (wn) at (5,1) {};


\draw[dashed,gray] (-5,0) -- (5,0) -- (5,-3);
\draw[gray] (-5,1.5) -- (-5,-3) -- (6,-3) node[right] {$\partial D_n$};

\node[above,red] at (-3.5,0) {$k_0$};
\node[above,red] at (0,0) {$k_1$};
\node[above,red] at (4,0) {$k_{n-1}$};

\draw[red] (-4,0.3) -- (-4,0) -- (-4.8,-3) node[below] {$\xi_r$};
\draw[red] (-3,0.3) -- (-3,0) -- (-4.3,-3);
\draw[red] (-0.5,0.3) -- (-0.5,0) -- (-4,-3);
\draw[red] (0.5,0.3) -- (0.5,0) -- (-3.8,-3);
\draw[red] (3.5,0.3) -- (3.5,0)-- (-3.2,-3);
\draw[red] (4.5,0.3) -- (4.5,0)-- (-2.5,-3) node[below] {$\xi_1$};

\node[red] at (-3.8,-1) {$\ldots$};
\node[red] at (-1.2,-0.9) {$\ldots$};
\node[red] at (0.25,-0.9) {$\ldots$};
\node[red] at (2.9,-0.5) {$\ldots$};

\node[gray] at (w0)[left=5pt] {$w_0$};
\node[gray] at (w1)[above=5pt] {$w_1$};
\node[gray] at (w2)[above=5pt] {$w_2$};
\node[gray] at (wn1)[above=5pt] {$w_{n-1}$};
\node[gray] at (wn)[above=5pt] {$w_n$};
\foreach \n in {w1,w2,wn1,wn}
  \node at (\n)[gray,circle,fill,inner sep=3pt]{};
\node at (w0)[gray,circle,fill,inner sep=3pt]{};
\end{tikzpicture}
\end{equation*}
while the right one has the leftmost single handle following the leftmost path of the above dashed box. All other handles are right shifted. 
\end{lemma}
\begin{proof}
It is a straightforward consequence of the handle rule. The braid involved is drawn in Figure \ref{braidactionE}, so that one sees the local system coefficient (we did not draw the punctures as they don't play any role). 

\begin{figure}[h!]
\begin{center}
\def\svgscale{0.4}
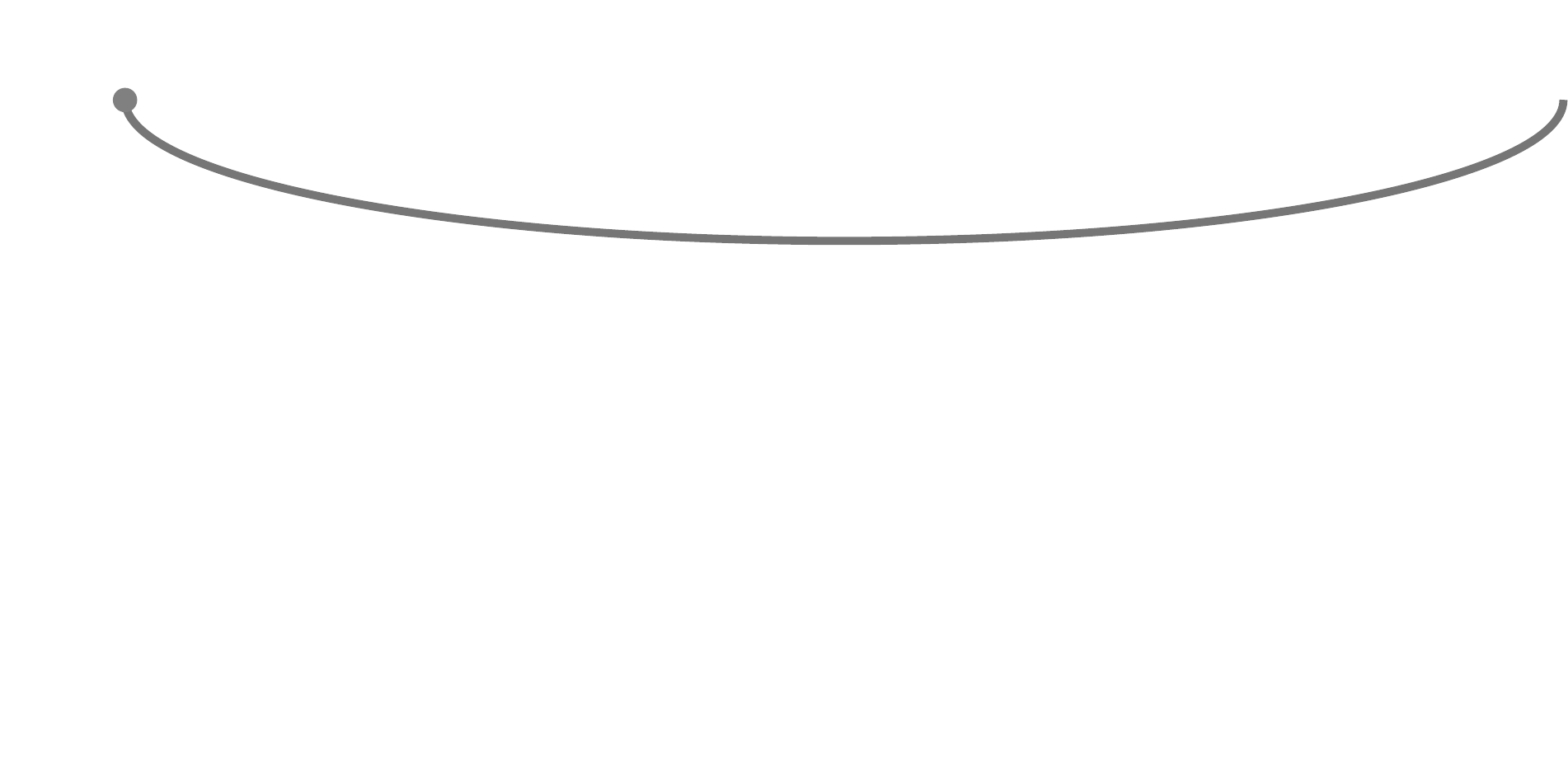
\caption{Handle rule. \label{braidactionE}}
\end{center}
\end{figure}
\end{proof}

\begin{lemma}
For any ${\bf k} = (k_0 ,\ldots , k_{n-1}) \in \EZnr$, the action of $E$ over the standard multi arcs is the following:
\[
E\cdot A'(k_0 ,\ldots , k_{n-1}) = \sum_{i=0}^{n-1} \mt^{k_0 + \cdots + k_{i-1}} A'(k_0 , \ldots , k_{i-1} , k_i-1 , k_{i+1}, \ldots , k_{n-1}) .
\] 
\end{lemma}
\begin{proof}
Every dashed component of $A'(k_0 , \ldots , k_{n-1})$ has its leftmost component having one end in $w_0$. For $i= 1,\ldots , n-1$ we have from the above lemma: 
\begin{align*}
A'(k_0 , \ldots , k_{n-1})  = \left(\vcenter{\hbox{
\begin{tikzpicture}[decoration={
    markings,
    mark=at position 0.5 with {\arrow{>}}}
    ]
\node (w0) at (-1,0) {};
\node (w1) at (0,0) {};
\node[gray] at (0.5,0) {\tiny $\ldots$};
\node (wn) at (2,0) {};
\node[gray] at (1.6,0) {\tiny $\ldots$};
\node (wi) at (1,0) {};
\coordinate (a) at (-2,-2);
\draw[dashed] (w0) to node[above,pos=0.5] (km) {\small $k_0$} (w1);
\draw[dashed] (w0) to[bend right=50] node[pos=0.85] (k0) {\small $k_{n-1}$} (wn);
\draw[dashed] (w0) to[bend right=30] node[pos=0.7] (k1) {\small $k_{i-1}$} (wi);
\node[gray] at (w0)[left=5pt] {$w_0$};
\node[gray] at (wn)[above=5pt] {$w_n$};
\node[gray] at (wi)[above=5pt] {$w_{i}$};
\foreach \n in {w0,wn,wi,w1}
  \node at (\n)[gray,circle,fill,inner sep=3pt]{};
\draw[double,red,thick] (km) -- (km|-a);
\draw[double,red,thick] (k0) -- (k0|-a);
\draw[double,red,thick] (k1) -- (k1|-a);
\draw[gray,thick] (-1,0.3) -- (w0) -- (-1,-2.5);
\end{tikzpicture}
}} \right)
= \mt^{k_0+ \cdots + k_{i-2}}
\left(\vcenter{\hbox{
\begin{tikzpicture}[decoration={
    markings,
    mark=at position 0.5 with {\arrow{>}}}
    ]
\node (w0) at (-1,0) {};
\node (w1) at (0,0) {};
\node[gray] at (0.5,0) {\tiny $\ldots$};
\node (wn) at (2,0) {};
\node[gray] at (1.6,0) {\tiny $\ldots$};
\node (wi) at (1,0) {};
\coordinate (a) at (-2,-2);
\draw[dashed] (w0) to node[above,pos=0.8] (km) {\small $k_0$} (w1);
\draw[dashed] (w0) to[bend right=50] node[pos=0.85] (k0) {\small $k_{n-1}$} (wn);
\draw[dashed] (w0) to[bend right=30] node[pos=0.7] (k1) {\small $k_{i-1}$} (wi);
\node[gray] at (w0)[left=5pt] {$w_0$};
\node[gray] at (wn)[above=5pt] {$w_n$};
\node[gray] at (wi)[above=5pt] {$w_{i}$};
\foreach \n in {w0,wn,wi,w1}
  \node at (\n)[gray,circle,fill,inner sep=3pt]{};
\draw[double,red,thick] (km) -- (km|-a);
\draw[double,red,thick] (k0) -- (k0|-a);
\node (solo) at (-0.5,-0.1) {};
\draw[red] (solo) -- (solo|-a);
\draw[double,red,thick] (k1) -- (k1|-a);
\draw[gray,thick] (-1,0.3) -- (w0) -- (-1,-2.5);
\end{tikzpicture}
}} \right)
\end{align*}

Using exactly same arguments from Example \ref{Ecodeseq}, we have:
\begin{align*}
\partial_* \left(\vcenter{\hbox{
\begin{tikzpicture}[decoration={
    markings,
    mark=at position 0.5 with {\arrow{>}}}
    ]
\node (w0) at (-1,0) {};
\node (w1) at (0,0) {};
\node[gray] at (0.5,0) {\tiny $\ldots$};
\node (wn) at (2,0) {};
\node[gray] at (1.6,0) {\tiny $\ldots$};
\node (wi) at (1,0) {};
\coordinate (a) at (-2,-2);
\draw[dashed] (w0) to node[above,pos=0.8] (km) {\small $k_0$} (w1);
\draw[dashed] (w0) to[bend right=50] node[pos=0.85] (k0) {\small $k_{n-1}$} (wn);
\draw[dashed] (w0) to[bend right=30] node[pos=0.7] (k1) {\small $k_{i-1}$} (wi);
\node[gray] at (w0)[left=5pt] {$w_0$};
\node[gray] at (wn)[above=5pt] {$w_n$};
\node[gray] at (wi)[above=5pt] {$w_{i}$};
\foreach \n in {w0,wn,wi,w1}
  \node at (\n)[gray,circle,fill,inner sep=3pt]{};
\draw[double,red,thick] (km) -- (km|-a);
\draw[double,red,thick] (km) -- (km|-a);
\node (solo) at (-0.5,-0.1) {};
\draw[red] (solo) -- (solo|-a);
\draw[double,red,thick] (k1) -- (k1|-a);
\draw[gray,thick] (-1,0.3) -- (w0) -- (-1,-2.5);
\end{tikzpicture}
}} \right) &
= \left(\vcenter{\hbox{
\begin{tikzpicture}[decoration={
    markings,
    mark=at position 0.5 with {\arrow{>}}}
    ]
\node (w0) at (-1,0) {};
\node (w1) at (0,0) {};
\node[gray] at (0.5,0) {\tiny $\ldots$};
\node (wn) at (2.4,0) {};
\node[gray] at (1.8,0) {\tiny $\ldots$};
\node (wi) at (1.2,0) {};
\coordinate (a) at (-2,-2);
\draw[dashed] (w0) to node[above,pos=0.8] (km) {\small $k_0$} (w1);
\draw[dashed] (w0) to[bend right=50] node[pos=0.85] (k0) {\small $k_{n-1}$} (wn);
\draw[dashed] (w0) to[bend right=40] node[pos=0.8] (k1) {\small $(k_{i-1}-1)$} (wi);
\node[gray] at (w0)[left=5pt] {$w_0$};
\node[gray] at (wn)[above=5pt] {$w_n$};
\node[gray] at (wi)[above=5pt] {$w_{i}$};
\foreach \n in {w0,wn,wi,w1}
  \node at (\n)[gray,circle,fill,inner sep=3pt]{};
\draw[double,red,thick] (km) -- (km|-a);
\draw[double,red,thick] (k0) -- (k0|-a);
\draw[double,red,thick] (k1) -- (k1|-a);
\draw[gray,thick] (-1,0.3) -- (w0) -- (-1,-2.5);
\end{tikzpicture}
}} \right) + \cdots
\end{align*}
where the rest of the terms concerns boundary terms coming from other arcs (different from the $k_{i-1}$ indexed one). Every dashed arc indexed by $k_i$ for $i=0 , \ldots , n-1$ can be treated the same way.  
The boundary of $A(k_0 , \ldots , k_{n-1})$ relative to $w_0$ is then the sum of these terms, and one gets the statement of the lemma.
\end{proof}

One has the following action over the normalized multi-arcs.

\begin{prop}[Action of $E$ over multi-arcs]\label{Eonarcs}
For any ${\bf k} = (k_0 ,\ldots , k_{n-1}) \in \EZnr$, the action of $E$ over the (normalized) multi-arc is the following:
\[
E\cdot A(k_0 ,\ldots , k_{n-1}) = \sum_{i=0}^{n-1} q^{\alpha_1 + \cdots + \alpha_i} \mt^{k_0 + \cdots + k_{i-1}} A(k_0 , \ldots , k_{i-1} , k_i-1 , k_{i+1}, \ldots , k_{n-1}) .
\]
\end{prop}
\begin{proof}
It is a simple computation:
\begin{align*}
E\cdot A(k_0 ,\ldots , k_{n-1}) & = q^{\alpha_1(k_1 + \cdots + k_{n-1} ) + \cdots + \alpha_{n-1} k_{n-1} } \sum_{i=0}^{n-1}  \mt^{k_0 + \cdots + k_{i-1}} A'(k_0 , \ldots , k_i-1 , \ldots , k_{n-1}) \\
& = \sum_{i=0}^{n-1} q^{\alpha_1 + \cdots + \alpha_i} \mt^{k_0 + \cdots + k_{i-1}} A(k_0 , \ldots , k_{i-1} , k_i-1 , k_{i+1}, \ldots , k_{n-1}) .
\end{align*}
\end{proof}

We emphasize the action in the case of one puncture.

\begin{coro}[$n=1$]\label{en=1}
Let $n=1$, $k \in \BN$, and $A(k)$ be the associated element of $\CH$. Then:
\[
E \cdot A(k) = A(k-1). 
\]
\end{coro}

\subsubsection{Action of $F^{(k)}$.}

Let $i \in \{ 1,\ldots , n \}$, and $S_i$ be the following class:
\begin{equation*}
S_i = \left(\vcenter{\hbox{
\begin{tikzpicture}[decoration={
    markings,
    mark=at position 0.5 with {\arrow{>}}}
    ]
\node (w0) at (-1,0) {};
\node (w1) at (0,0) {};
\node[gray] at (0.5,0) {\tiny $\ldots$};
\node (wn) at (2,0) {};
\node (wi) at (1,0) {};
\node[gray] at (2.5,0) {\tiny $\ldots$};
\node (wnn) at (3,0) {};
\coordinate (a) at (-2,-2);
\draw[dashed] (w0) to node[above,pos=0.5] (km) {\tiny $k_0$} (w1);
\draw[dashed] (w0) to[bend right=50] node[pos=0.75] (k0) {\small $k_{i}$} (wn);
\draw[dashed] (w0) to[bend right=30] node[pos=0.7] (k1) {\small $k_{i-1}$} (wi);
\draw[dashed] (w0) to[bend right=60] node[pos=0.85] (knn) {\small $k_{n-1}$} (wnn);
\draw (w0) to[bend left=50] node[above,pos=0.9] (solo) {} (wn);
\node[gray] at (w0)[left=5pt] {$w_0$};
\node[gray] at (wn)[above=5pt] {$w_{i+1}$};
\node[gray] at (wi)[above=5pt] {$w_{i}$};
\node[gray] at (wnn)[above=5pt] {$w_{n}$};
\foreach \n in {w0,wn,wi,w1,wnn}
  \node at (\n)[gray,circle,fill,inner sep=3pt]{};
\draw[double,red,thick] (km) -- (km|-a);
\draw[double,red,thick] (k0) -- (k0|-a);
\draw[double,red,thick] (k1) -- (k1|-a);
\draw[double,red,thick] (knn) -- (knn|-a);
\draw[red] (solo) -- (solo|-a);
\draw[gray,thick] (-1,0.3) -- (w0) -- (-1,-2.5);
\end{tikzpicture}
}} \right) \in \Hrelm_r.
\end{equation*}
Namely one recognizes a standard $(k_0 , \ldots , k_{n-1})$-multi arc to which a plain arc as in the picture has been added. 
To compute the action of $F^{(1)} $ we need the following lemma allowing us to deal with $S_i$ by recursion.

\begin{lemma}\label{recursionSi}
For $i \in \{ 2,\ldots , n \}$, the following equality holds in $\Hrelm_r$:
\begin{align*}
S_i = & (k_{i}+1)_{\mt^{-1}} A'(k_0 , \ldots , k_{i}+1,k_{i+1}, \ldots ,k_{n-1}) \\ & - \mt^{-k_{i}}(k_{i-1}+1)_{\mt} A'(k_0 , \ldots , k_{i-1}+1 , k_{i},\ldots ,k_{n-1}) \\ & + q^{- 2\alpha_i} \mt^{-k_i} S_{i-1} .
\end{align*}
\end{lemma}
\begin{proof}
By a breaking of a plain arc (see Example \ref{breakingplain}), one gets the following decomposition for $S_i$:
\begin{align*}
S_i= \left(\vcenter{\hbox{
\begin{tikzpicture}[decoration={
    markings,
    mark=at position 0.5 with {\arrow{>}}}
    ]
\node (w0) at (-1,0) {};
\node (w1) at (0,0) {};
\node[gray] at (0.5,0) {\tiny $\ldots$};
\node (wn) at (2,0) {};
\node (wi) at (1,0) {};
\node[gray] at (2.5,0) {\tiny $\ldots$};
\node (wnn) at (3,0) {};
\coordinate (a) at (-2,-2);
\draw[dashed] (w0) to node[above,pos=0.5] (km) {\small $k_0$} (w1);
\draw[dashed] (w0) to[bend right=50] node[pos=0.75] (k0) {\small $k_{i}$} (wn);
\draw[dashed] (w0) to[bend right=30] node[pos=0.7] (k1) {\small $k_{i-1}$} (wi);
\draw[dashed] (w0) to[bend right=60] node[pos=0.85] (knn) {\small $k_{n-1}$} (wnn);
\draw (w0) to[bend left=50] node[pos=0.9] (solo) {} (wi);
\node[gray] at (w0)[left=5pt] {$w_0$};
\node[gray] at (wn)[above=5pt] {$w_{i+1}$};
\node[gray] at (wi)[above=5pt] {$w_{i}$};
\node[gray] at (wnn)[above=5pt] {$w_{n}$};
\foreach \n in {w0,wn,wi,w1,wnn}
  \node at (\n)[gray,circle,fill,inner sep=3pt]{};
\draw[double,red,thick] (km) -- (km|-a);
\draw[double,red,thick] (k0) -- (k0|-a);
\draw[double,red,thick] (k1) -- (k1|-a);
\draw[double,red,thick] (knn) -- (knn|-a);
\draw[red] (solo) arc (180:0:0.4) node[above] (solop) {};
\draw[red] (solop) -- (solop|-a);
\draw[gray,thick] (-1,0.3) -- (w0) -- (-1,-2.5);
\end{tikzpicture}
}} \right) +
\left(\vcenter{\hbox{
\begin{tikzpicture}[decoration={
    markings,
    mark=at position 0.5 with {\arrow{>}}}
    ]
\node (w0) at (-1,0) {};
\node (w1) at (0,0) {};
\node[gray] at (0.5,0) {\tiny $\ldots$};
\node (wn) at (2,0) {};
\node (wi) at (1,0) {};
\node[gray] at (2.5,0) {\tiny $\ldots$};
\node (wnn) at (3,0) {};
\coordinate (a) at (-2,-2);
\draw[dashed] (w0) to node[above,pos=0.5] (km) {\small $k_0$} (w1);
\draw[dashed] (w0) to[bend right=50] node[pos=0.75] (k0) {\small $k_{i}$} (wn);
\draw[dashed] (w0) to[bend right=30] node[pos=0.7] (k1) {\small $k_{i-1}$} (wi);
\draw[dashed] (w0) to[bend right=60] node[pos=0.85] (knn) {\small $k_{n-1}$} (wnn);
\draw (wi) to node[above,pos=0.6] (solo) {} (wn);
\node[gray] at (w0)[left=5pt] {$w_0$};
\node[gray] at (wn)[above=5pt] {$w_{i+1}$};
\node[gray] at (wi)[above=5pt] {$w_{i}$};
\node[gray] at (wnn)[above=5pt] {$w_{n}$};
\foreach \n in {w0,wn,wi,w1,wnn}
  \node at (\n)[gray,circle,fill,inner sep=3pt]{};
\draw[double,red,thick] (km) -- (km|-a);
\draw[double,red,thick] (k0) -- (k0|-a);
\draw[double,red,thick] (k1) -- (k1|-a);
\draw[double,red,thick] (knn) -- (knn|-a);
\draw[red] (solo) -- (solo|-a);
\draw[gray,thick] (-1,0.3) -- (w0) -- (-1,-2.5);
\end{tikzpicture}
}} \right)
\end{align*}
We treat both right terms independently. From the first one we get:
\begin{align*}
\left(\vcenter{\hbox{
\begin{tikzpicture}[decoration={
    markings,
    mark=at position 0.5 with {\arrow{>}}}
    ]
\node (w0) at (-1,0) {};
\node (w1) at (0,0) {};
\node[gray] at (0.5,0) {\tiny $\ldots$};
\node (wn) at (2,0) {};
\node (wi) at (1,0) {};
\node[gray] at (2.5,0) {\tiny $\ldots$};
\node (wnn) at (3,0) {};
\coordinate (a) at (-2,-2);
\draw[dashed] (w0) to node[above,pos=0.5] (km) {\small $k_0$} (w1);
\draw[dashed] (w0) to[bend right=50] node[pos=0.75] (k0) {\small $k_{i}$} (wn);
\draw[dashed] (w0) to[bend right=30] node[pos=0.7] (k1) {\small $k_{i-1}$} (wi);
\draw[dashed] (w0) to[bend right=60] node[pos=0.85] (knn) {\small $k_{n-1}$} (wnn);
\draw (w0) to[bend left=50] node[pos=0.9] (solo) {} (wi);
\node[gray] at (w0)[left=5pt] {$w_0$};
\node[gray] at (wn)[above=5pt] {$w_{i+1}$};
\node[gray] at (wi)[above=5pt] {$w_{i}$};
\node[gray] at (wnn)[above=5pt] {$w_{n}$};
\foreach \n in {w0,wn,wi,w1,wnn}
  \node at (\n)[gray,circle,fill,inner sep=3pt]{};
\draw[double,red,thick] (km) -- (km|-a);
\draw[double,red,thick] (k0) -- (k0|-a);
\draw[double,red,thick] (k1) -- (k1|-a);
\draw[double,red,thick] (knn) -- (knn|-a);
\draw[red] (solo) arc (180:0:0.4) node[above] (solop) {};
\draw[red] (solop) -- (solop|-a);
\draw[gray,thick] (-1,0.3) -- (w0) -- (-1,-2.5);
\end{tikzpicture}
}} \right) = q^{- 2\alpha_i} \mt^{-k_{i}} S_{i-1}
\end{align*}
which follows from the handle rule.

Again, to treat the second term, breaking the plain arc (Example \ref{breakingplain}) leads to the following equality:

$
\begin{array}{l}
\left(\vcenter{\hbox{
\begin{tikzpicture}[decoration={
    markings,
    mark=at position 0.5 with {\arrow{>}}}
    ]
\node (w0) at (-1,0) {};
\node (w1) at (0,0) {};
\node[gray] at (0.5,0) {\tiny $\ldots$};
\node (wn) at (2,0) {};
\node (wi) at (1,0) {};
\node[gray] at (2.5,0) {\tiny $\ldots$};
\node (wnn) at (3,0) {};
\coordinate (a) at (-2,-2);
\draw[dashed] (w0) to node[above,pos=0.5] (km) {\small $k_0$} (w1);
\draw[dashed] (w0) to[bend right=50] node[pos=0.75] (k0) {\small $k_{i}$} (wn);
\draw[dashed] (w0) to[bend right=30] node[pos=0.7] (k1) {\small $k_{i-1}$} (wi);
\draw[dashed] (w0) to[bend right=60] node[pos=0.85] (knn) {\small $k_{n-1}$} (wnn);
\draw (wi) to node[above,pos=0.6] (solo) {} (wn);
\node[gray] at (w0)[left=5pt] {$w_0$};
\node[gray] at (wn)[above=5pt] {$w_{i+1}$};
\node[gray] at (wi)[above=5pt] {$w_{i}$};
\node[gray] at (wnn)[above=5pt] {$w_{n}$};
\foreach \n in {w0,wn,wi,w1,wnn}
  \node at (\n)[gray,circle,fill,inner sep=3pt]{};
\draw[double,red,thick] (km) -- (km|-a);
\draw[double,red,thick] (k0) -- (k0|-a);
\draw[double,red,thick] (k1) -- (k1|-a);
\draw[double,red,thick] (knn) -- (knn|-a);
\draw[red] (solo) -- (solo|-a);
\draw[gray,thick] (-1,0.3) -- (w0) -- (-1,-2.5);
\end{tikzpicture}
}} \right) =
\end{array}
$
\begin{align*}
 \left(\vcenter{\hbox{
\begin{tikzpicture}[decoration={
    markings,
    mark=at position 0.5 with {\arrow{>}}}
    ]
\node (w0) at (-1,0) {};
\node (w1) at (0,0) {};
\node[gray] at (0.5,0) {\tiny $\ldots$};
\node (wn) at (2,0) {};
\node (wi) at (1,0) {};
\node[gray] at (2.5,0) {\tiny $\ldots$};
\node (wnn) at (3,0) {};
\coordinate (a) at (-2,-2);
\draw[dashed] (w0) to node[above,pos=0.5] (km) {\small $k_0$} (w1);
\draw[dashed] (w0) to[bend right=50] node[pos=0.75] (k0) {\small $k_{i}$} (wn);
\draw[dashed] (w0) to[bend right=30] node[pos=0.65] (k1) {\small $k_{i-1}$} (wi);
\draw[dashed] (w0) to[bend right=60] node[pos=0.85] (knn) {\small $k_{n-1}$} (wnn);
\draw (w0) to[bend right=35] node[above,pos=0.9] (solo) {} (wn);
\node[gray] at (w0)[left=5pt] {$w_0$};
\node[gray] at (wn)[above=5pt] {$w_{i+1}$};
\node[gray] at (wi)[above=5pt] {$w_{i}$};
\node[gray] at (wnn)[above=5pt] {$w_{n}$};
\foreach \n in {w0,wn,wi,w1,wnn}
  \node at (\n)[gray,circle,fill,inner sep=3pt]{};
\draw[double,red,thick] (km) -- (km|-a);
\draw[double,red,thick] (k0) -- (k0|-a);
\draw[double,red,thick] (k1) -- (k1|-a);
\draw[double,red,thick] (knn) -- (knn|-a);
\draw[red] (solo) -- (solo|-a);
\draw[gray,thick] (-1,0.3) -- (w0) -- (-1,-2.5);
\end{tikzpicture}
}} \right) -
 \left(\vcenter{\hbox{
\begin{tikzpicture}[decoration={
    markings,
    mark=at position 0.5 with {\arrow{>}}}
    ]
\node (w0) at (-1,0) {};
\node (w1) at (0,0) {};
\node[gray] at (0.5,0) {\tiny $\ldots$};
\node (wn) at (2,0) {};
\node (wi) at (1,0) {};
\node[gray] at (2.5,0) {\tiny $\ldots$};
\node (wnn) at (3,0) {};
\coordinate (a) at (-2,-2);
\draw[dashed] (w0) to node[above,pos=0.5] (km) {\small $k_0$} (w1);
\draw[dashed] (w0) to[bend right=50] node[pos=0.75] (k0) {\small $k_{i}$} (wn);
\draw[dashed] (w0) to[bend right=30] node[pos=0.65] (k1) {\small $k_{i-1}$} (wi);
\draw[dashed] (w0) to[bend right=60] node[pos=0.85] (knn) {\small $k_{n-1}$} (wnn);
\draw (w0) to[bend right=50] node[pos=0.9] (solo) {} (wi);
\node[gray] at (w0)[left=5pt] {$w_0$};
\node[gray] at (wn)[above=5pt] {$w_{i+1}$};
\node[gray] at (wi)[above=5pt] {$w_{i}$};
\node[gray] at (wnn)[above=5pt] {$w_{n}$};
\foreach \n in {w0,wn,wi,w1,wnn}
  \node at (\n)[gray,circle,fill,inner sep=3pt]{};
\draw[double,red,thick] (km) -- (km|-a);
\draw[double,red,thick] (k0) -- (k0|-a);
\draw[double,red,thick] (k1) -- (k1|-a);
\draw[double,red,thick] (knn) -- (knn|-a);
\draw[red] (solo) arc (90:80:4) node[above] (solop) {};
\draw[red] (solop) -- (solop|-a);
\draw[gray,thick] (-1,0.3) -- (w0) -- (-1,-2.5);
\end{tikzpicture}
}} \right).
\end{align*}
To decompose these two terms in the standard multi-arc basis, one must apply Lemma \ref{plaintodashed} to crash a plain arc over a dashed one, after a simple application of the handle rule to reorganize the handles of the right term. This recovers the lemma.
\end{proof}
We use this lemma to compute the action of $F^{(1)}$ in the multi-arcs basis. 
\begin{lemma}
Let ${\bf k}=(k_0 , \ldots , k_{n-1}) \in \EZnr$, the action of $F^{(1)}$ over the associated standard multi arc is the following:
\begin{equation*}
F^{(1)}  \cdot A'({\bf k}) = \sum_{i=0}^{n-1} q^{\sum_{j=1}^{i+1} \alpha_j} q^{-  \sum_{j=i+2}^n \alpha_j} \mt^{-\sum_{j=i+1}^{n-1} k_j}(k_i+1)_{\mt} (1-q^{-2\alpha_{i+1}} \mt^{-k_i} ) A'({\bf k})_i
\end{equation*}where $A'({\bf k})_i$ means $A'(k_0 , \ldots , k_{i-1} , k_i+1 , k_{i+1} , \ldots , k_{n-1})$.
\end{lemma}
\begin{proof}
First, we compute the element $F' \cdot A'(k_0 , \ldots , k_{n-1})$ of $\Hrelm_r$. It corresponds to the following class for which we give a decomposition in $\Hrelm_r$:
\[
\begin{array}{ccc}
\left(\vcenter{\hbox{
\begin{tikzpicture}[decoration={
    markings,
    mark=at position 0.5 with {\arrow{>}}}
    ]
\node (w0) at (-1,0) {};
\node (w1) at (0,0) {};
\node[gray] at (0.5,0) {\tiny $\ldots$};
\node (wn) at (2,0) {};
\node[gray] at (1.6,0) {\tiny $\ldots$};
\node (wi) at (1,0) {};
\coordinate (a) at (-2,-2);
\draw[dashed] (w0) to node[above,pos=0.5] (km) {\small $k_0$} (w1);
\draw[dashed] (w0) to[bend right=50] node[pos=0.85] (k0) {\small $k_n$} (wn);
\draw[dashed] (w0) to[bend right=30] node[pos=0.7] (k1) {\small $k_i$} (wi);
\draw[postaction={decorate}] (w0) to[bend right=70] (2.5,0) node[above] (solo) {};
\draw[postaction={decorate}] (2.5,0) to[bend right=70] (w0);
\draw[red] (solo) -- (solo|-a); 
\node[gray] at (w0)[left=5pt] {$w_0$};
\node[gray] at (wn)[above=5pt] {$w_n$};
\node[gray] at (wi)[above=5pt] {$w_{i}$};
\draw[gray,dashed] (-1,-2) -- (3,-2) -- (3,-2.5);
\foreach \n in {w0,wn,wi,w1}
  \node at (\n)[gray,circle,fill,inner sep=3pt]{};
\draw[double,red,thick] (km) -- (km|-a);
\draw[red] (km|-a) -- (-0.8,-2.5);
\node[red] at (-0.3,-2.3) {$\ldots$};
\draw[double,red,thick] (k0) -- (k0|-a);
\draw[double,red,thick] (k1) -- (k1|-a);
\draw[red] (k0|-a) -- (-0.3,-2.5);
\draw[red] (solo|-a) -- (-0.2,-2.5);
\draw[gray,thick] (-1,0.3) -- (w0) -- (-1,-2.5) -- (3.2,-2.5);
\end{tikzpicture}
}} \right) & = &\\
\left(\vcenter{\hbox{
\begin{tikzpicture}[decoration={
    markings,
    mark=at position 0.5 with {\arrow{>}}}
    ]
\node (w0) at (-1,0) {};
\node (w1) at (0,0) {};
\node[gray] at (0.5,0) {\tiny $\ldots$};
\node (wn) at (2,0) {};
\node[gray] at (1.6,0) {\tiny $\ldots$};
\node (wi) at (1,0) {};
\coordinate (a) at (-2,-2);
\draw[dashed] (w0) to node[above,pos=0.5] (km) {\small $k_0$} (w1);
\draw[dashed] (w0) to[bend right=50] node[above,pos=0.85] (k0) {\small $k_n$} (wn);
\draw[dashed] (w0) to[bend right=30] node[pos=0.7] (k1) {\small $k_i$} (wi);
\draw[postaction={decorate}] (w0) to[bend right=70] node[above,pos=0.9] (solo) {} (wn);
\draw[red] (solo) -- (solo|-a); 
\node[gray] at (w0)[left=5pt] {$w_0$};
\node[gray] at (wn)[above=5pt] {$w_n$};
\node[gray] at (wi)[above=5pt] {$w_{i}$};
\foreach \n in {w0,wn,wi,w1}
  \node at (\n)[gray,circle,fill,inner sep=3pt]{};
\draw[double,red,thick] (km) -- (km|-a);
\draw[double,red,thick] (k0) -- (k0|-a);
\draw[double,red,thick] (k1) -- (k1|-a);
\draw[gray,thick] (-1,0.3) -- (w0) -- (-1,-2.2);
\end{tikzpicture}
}} \right) & - & 
\left(\vcenter{\hbox{
\begin{tikzpicture}[decoration={
    markings,
    mark=at position 0.5 with {\arrow{>}}}
    ]
\node (w0) at (-1,0) {};
\node (w1) at (0,0) {};
\node[gray] at (0.5,0) {\tiny $\ldots$};
\node (wn) at (2,0) {};
\node[gray] at (1.6,0) {\tiny $\ldots$};
\node (wi) at (1,0) {};
\coordinate (a) at (-2,-2);
\draw[dashed] (w0) to node[above,pos=0.5] (km) {\small $k_0$} (w1);
\draw[dashed] (w0) to[bend right=50] node[pos=0.85] (k0) {\small $k_n$} (wn);
\draw[dashed] (w0) to[bend right=30] node[pos=0.7] (k1) {\small $k_i$} (wi);
\draw[postaction={decorate}] (w0) to[bend left=70]  node[pos=0.9] (solo) {} (wn);
\draw[red] (solo) arc (90:-90:0.5) -- (solo|-a); 
\node[gray] at (w0)[left=5pt] {$w_0$};
\node[gray] at (wn)[above right] {$w_n$};
\node[gray] at (wi)[above=5pt] {$w_{i}$};
\foreach \n in {w0,wn,wi,w1}
  \node at (\n)[gray,circle,fill,inner sep=3pt]{};
\draw[double,red,thick] (km) -- (km|-a);
\draw[double,red,thick] (k0) -- (k0|-a);
\draw[double,red,thick] (k1) -- (k1|-a);
\draw[gray,thick] (-1,0.3) -- (w0) -- (-1,-2.2);
\end{tikzpicture}
}} \right)
\end{array}.
\]
This decomposition follows from a breaking of a plain arc (Example \ref{breakingplain}). The minus sign is due to the reverse of the orientation of right term's plain arc. The first term of the decomposition is in position to apply Lemma \ref{plaintodashed}, while after a handle rule one recognizes $S_{n-1}$ in the second term. Finally we get the following formula:
\[
F' \cdot A'(k_0 , \ldots , k_{n-1}) = (k_{n-1}+1)_{\mt} A'(k_0 , \ldots , k_{n-1}+1) - q^{- 2\alpha_n} S_{n-1} .
\]
Thank to the recursive property of $S_{n-1}$ the proof is achieved using Lemma \ref{recursionSi}, so that one gets:
\[
\begin{array}{l}
F' \cdot A'({\bf k}) = \sum_{i=0}^{n-1} q^{- 2 \sum_{j=i+2}^n \alpha_j} \mt^{-\sum_{j=i+1}^{n-1} k_j}(k_i+1)_{\mt} (1-q^{-2\alpha_{i+1}} \mt^{-k_i}) A'({\bf k})_i.
\end{array}
\]
By multiplication of the action by $q^{\sum \alpha_i}$, one obtains the expected action for $F^{(1)}$ over the multi arc basis. 


\end{proof}

\begin{prop}[Action of $F^{(1)}$ over multi-arcs] \label{Fonarcs}
Let ${\bf k}=(k_0 , \ldots , k_{n-1}) \in \EZnr$, the action of $F^{(1)}$ over the associated standard (normalized) multi-arc is the following
\[
F^{(1)} \cdot A({\bf k}) = \sum_{i=0}^{n-1} q^{-  \sum_{j=i+2}^n \alpha_j} \mt^{-\sum_{j=i+1}^{n-1} k_j} q^{\alpha_{i+1}} (k_i+1)_{\mt} (1-q^{-2\alpha_{i+1}} \mt^{-k_i} ) A({\bf k})_i .
\]
\end{prop}
\begin{proof}
It is a straightforward consequence of previous lemma and of the normalization sending the family $\CA'$ to $\CA$. 
\end{proof}

We emphasize again the case $n=1$.

\begin{coro}[$n=1$]\label{fn=1}
Let $n=1$, $k \in \BN$, and $A(k)$ be the associated element of $\CH$. Then:
\[
F^{(1)} \cdot A(k) = q^{\alpha_1} (k+1)_{\mt} (1-q^{-2 \alpha_1}\mt^{-k}) A(k+1) .
\]
\end{coro}

We end this section by giving the action of the divided powers $F^{(l)}$ but only in the case of one puncture. We need the following remark.
\begin{rmk}\label{looptodashed}
We observe the following relations between homology classes:

\[
\begin{array}{rl}
\left(\vcenter{\hbox{ \begin{tikzpicture}[decoration={
    markings,
    mark=at position 0.5 with {\arrow{>}}}
    ]
\coordinate (w0) at (-1,0) {};
\coordinate (w1) at (1,0) {};
\coordinate (x0) at (-1,-1) {};
\coordinate (x1) at (1,-1) {};

\draw[dashed,postaction={decorate}] (w0) -- node[above,pos=0.3] (k1) {$k$} (w1);
\draw[postaction={decorate}] (w0) to[bend right=90] node[above,pos=0.75] (k0) {} (1.4,0);
\draw[postaction={decorate}] (1.4,0) to[bend right=90] (w0);

\draw[red] (k0) -- node[pos=0.66] (b) {} (k0|-x0);
\draw[red] (k1) -- (k1|-x1);

\node[gray] at (w0)[left=5pt] {$w_0$};
\node[gray] at (w1)[right=5pt] {$w_i$};
\foreach \n in {w0,w1}
  \node at (\n)[gray,circle,fill,inner sep=3pt]{};
  
\draw[gray] (w0) -- (-1,0.5);
\draw[gray] (w0) -- (-1,-0.5);

\end{tikzpicture} }}\right) 
& =
\left(\vcenter{\hbox{ \begin{tikzpicture}[decoration={
    markings,
    mark=at position 0.5 with {\arrow{>}}}
    ]
\coordinate (w0) at (-1,0) {};
\coordinate (w1) at (1,0) {};
\coordinate (x0) at (-1,-1) {};
\coordinate (x1) at (1,-1) {};

\draw[dashed,postaction={decorate}] (w0) -- node[above,pos=0.3] (k1) {$k$} (w1);
\draw[postaction={decorate}] (w0) to[bend right=30] node[above,pos=0.75] (k0) {} (w1);

\draw[red] (k0) -- node[pos=0.66] (b) {} (k0|-x0);
\draw[red] (k1) -- (k1|-x1);

\node[gray] at (w0)[left=5pt] {$w_0$};
\node[gray] at (w1)[right=5pt] {$w_i$};
\foreach \n in {w0,w1}
  \node at (\n)[gray,circle,fill,inner sep=3pt]{};
  
\draw[gray] (w0) -- (-1,0.5);
\draw[gray] (w0) -- (-1,-0.5);

\end{tikzpicture} }}\right) 
-\left(\vcenter{\hbox{ \begin{tikzpicture}[decoration={
    markings,
    mark=at position 0.5 with {\arrow{>}}}
    ]
\coordinate (w0) at (-1,0) {};
\coordinate (w1) at (1,0) {};
\coordinate (x0) at (-1,-1) {};
\coordinate (x1) at (1,-1) {};

\draw[dashed,postaction={decorate}] (w0) -- node[pos=0.3] (k0) {$k$} (w1);
\draw[postaction={decorate}] (w0) to[bend left=30] node[pos=0.8] (k1) {} (w1);

\draw[red] (k0) -- node[midway] (a) {} (k0|-x0);
\draw[red] (k1) arc (180:90:0.4) arc (90:-90:0.4) -- (k1|-x1);

\draw[gray] (w0) -- (-1,0.5);
\draw[gray] (w0) -- (-1,-0.5);

\node[gray] at (w0)[left=5pt] {$w_0$};
\node[gray] at (w1)[right=5pt] {$w_i$};
\foreach \n in {w0,w1}
  \node at (\n)[gray,circle,fill,inner sep=3pt]{};

\end{tikzpicture} }}\right) \\
& = \left( (k+1)_{\mt} - q^{-2 \alpha_i} (k+1)_{\mt^{-1}} \right)  
\left(\vcenter{\hbox{ \begin{tikzpicture}[decoration={
    markings,
    mark=at position 0.5 with {\arrow{>}}}
    ]
\coordinate (w0) at (-1,0) {};
\coordinate (w1) at (1,0) {};
\coordinate (x0) at (-1,-1) {};
\coordinate (x1) at (1,-1) {};

\draw[dashed] (w0) -- node[above,pos=0.5] (k1) {\small $k+1$} (w1);

\draw[red] (k1) -- (k1|-x1);

\node[gray] at (w0)[left=5pt] {$w_0$};
\node[gray] at (w1)[right=5pt] {$w_i$};
\foreach \n in {w0,w1}
  \node at (\n)[gray,circle,fill,inner sep=3pt]{};
  
\draw[gray] (w0) -- (-1,0.5);
\draw[gray] (w0) -- (-1,-0.5);

\end{tikzpicture} }}\right) \\
& = (k+1)_{\mt}  \left( 1 - q^{-2 \alpha_i} \mt^{-k} \right)  
\left(\vcenter{\hbox{ \begin{tikzpicture}[decoration={
    markings,
    mark=at position 0.5 with {\arrow{>}}}
    ]
\coordinate (w0) at (-1,0) {};
\coordinate (w1) at (1,0) {};
\coordinate (x0) at (-1,-1) {};
\coordinate (x1) at (1,-1) {};

\draw[dashed] (w0) -- node[above,pos=0.5] (k1) {\small $k+1$} (w1);

\draw[red] (k1) -- (k1|-x1);

\node[gray] at (w0)[left=5pt] {$w_0$};
\node[gray] at (w1)[right=5pt] {$w_i$};
\foreach \n in {w0,w1}
  \node at (\n)[gray,circle,fill,inner sep=3pt]{};
  
\draw[gray] (w0) -- (-1,0.5);
\draw[gray] (w0) -- (-1,-0.5);

\end{tikzpicture} }}\right) 
\end{array}
\]
where everything stands inside a small neighborhood of the picture, without perturbating the rest of the class contained outside of it. The first equality comes from a breaking of plain arc, see Example \ref{breakingplain}. The second one is a consequence first of the application of a handle rule to get vertical handles, and then relations from Lemma \ref{plaintodashed}. 
\end{rmk}

\begin{prop}[Action of $F^{(l)}$, $n=1$]\label{Fln=1}
Let $n=1$, $k \in \BN$, and $A(k)$ be the associated element of $\CH$. Let $l \in \BN$, then:
\[
F^{(l)} \cdot A(k) = q^{\frac{-l(l-1)}{2}} q^{l \alpha_1} {{k+l}\choose{k}}_{\mt} \prod_{m=0}^l (1-q^{-2\alpha_1}\mt^{-m}) A(k+l). 
\]
\end{prop}
\begin{proof}
\[
\begin{array}{rl}
(l)_{\mt}! F'^{(l)}\cdot A(k) & =
(l)_{\mt}!\left(\vcenter{\hbox{ \begin{tikzpicture}[decoration={
    markings,
    mark=at position 0.5 with {\arrow{>}}}
    ]
\coordinate (w0) at (-1,0) {};
\coordinate (w1) at (1,0) {};
\coordinate (x0) at (-1,-1) {};
\coordinate (x1) at (1,-1) {};

\draw[dashed,postaction={decorate}] (w0) -- node[above,pos=0.3] (k1) {$k$} (w1);
\draw[dashed,postaction={decorate}] (w0) to[bend right=90] node[pos=0.75] (k0) {$l$} (1.4,0);
\draw[dashed,postaction={decorate}] (1.4,0) to[bend right=90] (w0);

\draw[double,red] (k0) -- node[pos=0.66] (b) {} (k0|-x0);
\draw[double,red] (k1) -- (k1|-x1);

\node[gray] at (w0)[left=5pt] {$w_0$};
\node[gray] at (w1)[right=5pt] {$w_1$};
\foreach \n in {w0,w1}
  \node at (\n)[gray,circle,fill,inner sep=3pt]{};
  
\draw[gray] (w0) -- (-1,0.5);
\draw[gray] (w0) -- (-1,-0.5);

\end{tikzpicture} }}\right) \\
& =
\left(\vcenter{\hbox{ \begin{tikzpicture}[decoration={
    markings,
    mark=at position 0.5 with {\arrow{>}}}
    ]
\coordinate (w0) at (-1,0) {};
\coordinate (w1) at (1,0) {};
\coordinate (x0) at (-1,-1) {};
\coordinate (x1) at (1,-1) {};

\draw[dashed,postaction={decorate}] (w0) -- node[above,pos=0.3] (k1) {$k$} (w1);
\draw[postaction={decorate}] (w0) to[bend right=90] node[above,pos=0.75] (k0) {} (1.4,0);
\draw[postaction={decorate}] (1.4,0) to[bend right=90] (w0);
\draw[postaction={decorate}] (w0) to[bend right=90] node[above,pos=0.75] (k2) {} (2,0);
\draw[postaction={decorate}] (2,0) to[bend right=90] (w0);
\node at (1.7,0) {$\ldots$};
\node at (1.7,0.2) {$l$};

\draw[red] (k0) -- node[pos=0.66] (b) {} (k0|-x0);
\draw[red] (k2) -- node[pos=0.66] (b) {} (k2|-x0);
\draw[double,red] (k1) -- (k1|-x1);

\node[gray] at (w0)[left=5pt] {$w_0$};
\node[gray] at (w1)[above=3pt] {$w_1$};
\foreach \n in {w0,w1}
  \node at (\n)[gray,circle,fill,inner sep=3pt]{};
  
\draw[gray] (w0) -- (-1,0.5);
\draw[gray] (w0) -- (-1,-0.5);

\end{tikzpicture} }}\right) \\
& = \prod_{m=0}^l (k+m)_{\mt}!  \left( 1 - q^{-2 \alpha_1} \mt^{-m} \right)  
\left(\vcenter{\hbox{ \begin{tikzpicture}[decoration={
    markings,
    mark=at position 0.5 with {\arrow{>}}}
    ]
\coordinate (w0) at (-1,0) {};
\coordinate (w1) at (1,0) {};
\coordinate (x0) at (-1,-1) {};
\coordinate (x1) at (1,-1) {};

\draw[dashed] (w0) -- node[above,pos=0.5] (k1) {\small $k+l$} (w1);

\draw[red] (k1) -- (k1|-x1);

\node[gray] at (w0)[left=5pt] {$w_0$};
\node[gray] at (w1)[right=5pt] {$w_1$};
\foreach \n in {w0,w1}
  \node at (\n)[gray,circle,fill,inner sep=3pt]{};
  
\draw[gray] (w0) -- (-1,0.5);
\draw[gray] (w0) -- (-1,-0.5);

\end{tikzpicture} }}\right) 
\end{array}
\]
The second equality comes from Corollary \ref{1forkto1code} and the last one is an iteration of the relations from Remark \ref{looptodashed}. Finally we have:
\[
F'^{(l)}\cdot A(k) = {{k+l}\choose{k}}_{\mt} \prod_{m=0}^l (1-q^{-2\alpha_1}\mt^{-m}) A(k+l) .
\]
The proposition is proved after the normalization passing from $F'^{(l)}$ to $F^{(l)}$.
\end{proof}

\subsubsection{Recovering monoidality of Verma modules for $\UqhL$.}\label{sectionmonoidality}

Since in this section $n$ (the number of punctures) is particularly important, we denote by $\CH^{\alpha_1,\ldots,\alpha_n}$ the module $\CH$ built from $X_r(w_0 , \ldots , w_n)$ with coefficients in $\Laurentmax = \BZ \left[ \mt^{\pm 1}, q^{\pm \alpha_1} , \ldots ,  q^{\pm \alpha_n} \right]$.  

\begin{rmk} We recall the action of $K$. We distinguish the cases whether $n$ is greater than $1$ or not. 
\begin{itemize}
\item[($n=1$)] Let $n=1$ so that $\Laurentmax = \BZ \left[ \mt^{\pm 1}, q^{\pm \alpha_1}  \right]$. Let $k \in \BN$, and $A(k)$ be the associated element of $\CH^{\alpha_1}$. Then:
\[
K \cdot A(k) = q^{\alpha_1} \mt^{k} A(k). 
\]
\item[($n>1$)] Let $n>1$, ${\bf k} \in \EZnr$ and $A({\bf k})$ be the associated element of $\Hrelm_r \in \CH^{\alpha_1 , \ldots , \alpha_n}$. Then:
\[
K \cdot A({\bf k}) = q^{\sum_{i=1}^n \alpha_i} \mt^{r} A({\bf k}). 
\]
\end{itemize}
\end{rmk}

\begin{prop}
Let $\mt=q^{-2}$. The module $\CH^{\alpha_1}$ is a Verma module for $\UqhL$ of highest weight $\alpha_1$. 
\end{prop}
\begin{proof}
The presentation of the action over a Verma-module, is given in \cite{JK} (see relations (18)) and is recalled in Definition \ref{GoodVerma}. Using Corollaries \ref{en=1} and \ref{fn=1} and the above remark in the case $n=1$, one recognizes the presentation of the Verma module. Namely, let $\mt=q^{-2}$, and let $s=q^{\alpha_1}$. Then:
\[
K \cdot A({k}) = q^{\alpha_1} \mt^{k} A({k}) = s q^{-2k} A(k)
\]
\[
E \cdot A(k) = A(k-1)
\]
and
\[
F^{(1)} \cdot A(k) = q^{\alpha_1} (k+1)_{\mt} (1-q^{-2 \alpha_1}\mt^{-k}) A(k+1) = \left[ k+1 \right]_q(sq^{-k} - s^{-1} q^k) A(k+1).
\]
The last equality uses Remark \ref{quantumtq}. 

These expressions ensure that the isomorphism of $\BZ \left[ s^{\pm 1} , q^{\pm 1} \right]$-modules:
\[
\bfct
\CH^{\alpha_1} & \to & V^{\alpha_1} \\
A(k) & \mapsto & v_k \text{ for } k \in \BN
\efct
\]
is $\UqhL$ equivariant. 
\end{proof}

\begin{rmk}
There is an isomorphism of $\Laurentmax$-modules:
\[
\tens: \bfct
\CH^{\alpha_1 , \ldots , \alpha_n} & \to & \CH^{\alpha_1} \otimes \cdots \otimes \CH^{\alpha_n} \\
A(k_0 , \ldots , k_{n-1}) & \mapsto & A(k_0) \otimes \cdots \otimes A(k_{n-1}).
\efct
\]
\end{rmk}

\begin{thm}[Monoidality of Verma-modules.]\label{monoidality}
For $\mt=q^{-2}$, the morphism:
\[
\tens: \CH^{\alpha_1 , \ldots , \alpha_n} \to \CH^{\alpha_1} \otimes \cdots \otimes \CH^{\alpha_n}
\]
is an isomorphism of $\UqhL$-modules. 
\end{thm}
\begin{proof}
From Proposition \ref{Fonarcs} one remarks that the formulae satisfy:
\[
\begin{array}{rl}
\tens \left( F^{(1)} \cdot A({\bf k}) \right) & = \tens \left( \sum_{i=0}^{n-1} q^{-  \sum_{j=i+2}^n \alpha_j} \mt^{-\sum_{j=i+1}^{n-1} k_j} q^{\alpha_{i+1}} (k_i+1)_{\mt} (1-q^{-2\alpha_{i+1}} \mt^{-k_i} ) A({\bf k})_i \right) \\
& = \sum_{i=0}^{n-1} A(k_0) \otimes \cdots \otimes \left( q^{\alpha_{i+1}} (k_i+1)_{\mt} (1-q^{-2\alpha_{i+1}} \mt^{-k_i} ) \right) A(k_i+1)\\
& \otimes q^{-\alpha_{i+2}} \mt^{-k_{i+1}} A(k_{i+1}) \otimes \cdots \otimes q^{-\alpha_{n}} \mt^{-k_{n-1}} A(k_{n-1})\\
& = \sum_{i=0}^{n-1} \left( 1\otimes 1 \otimes \cdots \otimes F^{(1)} \otimes K^{-1} \otimes \cdots \otimes K^{-1} \right) A(k_0) \otimes \cdots \otimes A(k_{n-1}) 
\end{array}
\]
where the $F^{(1)}$ in the sum is in the $(i+1)^{\text{st}}$ position, one recognizes the expression of $\Delta^n(F^{(1)})$.

We do the same for $E$, from Proposition \ref{Eonarcs} we have:
\[
\begin{array}{rl}
\tens \left( E\cdot A(k_0 ,\ldots , k_{n-1}) \right) & = \tens \left( \sum_{i=0}^{n-1} q^{\alpha_1 + \cdots + \alpha_i} \mt^{k_0 + \cdots + k_{i-1}} A(k_0 , \ldots , k_{i-1} , k_i-1 , k_{i+1}, \ldots , k_{n-1}) \right) \\
& = \sum_{i=0}^{n-1} (K \otimes \cdots \otimes K \otimes E \otimes 1 \otimes \cdots \otimes 1 ) A(k_0) \otimes \cdots \otimes A(k_{n-1}) 
\end{array}
\]
which proves that the action of $E$ over $\CH^{\alpha_1 , \ldots , \alpha_n}$ corresponds to the action of $\Delta^n(E)$ over the tensor product. The same proof works for the action of $K$ so that the theorem holds. 
\end{proof}

\begin{rmk}
The above theorem suggests that there should exist a homological interpretation of the $\UqhL$ coproduct. Probably in terms of gluing once punctured disks along arcs of their boundary. The morphism $\tens$ should then be the involved homological operator. 

We remark that:
\[
\tens (A'({\bf k})) = q^{\alpha_1(k_1 + \cdots + k_{n-1} ) + \alpha_2 ( k_2+ \cdots +k_{n-1}) + \cdots + \alpha_{n-1} k_{n-1} } A'(k_0) \otimes \cdots \otimes A'(k_{n-1})
\]
so that multi-arcs are divided into tensor products of single arcs, with coefficients appearing from the gluing operation. If one is able to draw the handles corresponding to the normalization coefficient, one should know how to glue once-punctured disks. 

%
\end{rmk}

\begin{rmk}
Theorem \ref{monoidality} answers Conjecture 6.2 of \cite{FW}. In fact the isomorphism was suggested by the conjecture while the topological basis was not the one that fits with the integral coefficients setting. See Corollary \ref{answerConjecture61} for precisions. 
\end{rmk}

\subsection{Homological braid action}\label{homologicalbraid}

\subsubsection{Definition of the action.}

In this section we present an extension of Lawrence representations (\cite{Law}) for braid groups, following her ideas. The starting point is the fact that the braid group is the mapping class group of $D_n$.
\begin{recall}
The braid group on $n$ strands is the mapping class group of $D_n$.
\[
\Bn = \Mod(D_n) = \Homeo(D_n, \partial D) \big/ {\Homeo}_0(D_n, \partial D).
\]
This definition is isomorphic to the Artin presentation of the braid group (Definition \ref{Artinpres}) by sending generator $\sigma_i$ to the mapping class of the half Dehn twist swapping punctures $w_i$ and $w_{i+1}$. The {\em pure braid group} $\PBn$ is defined to be braids fixing the punctures pointwise. 
\end{recall}
The action of a homeomorphism of $D_n$ can be generalized to $X_r$ as homeomorphisms extend to the configuration space coordinate by coordinate. Namely, if $\phi$ is a homeomorphism of $D_n$, the application : 
\[
\bfct
X_r & \to & X_r\\
\lbrace z_1, \ldots , z_r \rbrace & \mapsto & \lbrace \phi(z_1) , \ldots , \phi(z_r) \rbrace
\efct
\]
is a homeomorphism. We show that the action of half Dehn twists passes to homology with local coefficients in $L_r$, treating separately the unicolored case ($\alpha_1 = \cdots = \alpha_n$) from the general one. In the unicolored case, we get a representation of the standard braid group.

\begin{lemma}[Representation of the braid group]\label{RepBn}
Let $\alpha= \alpha_1 = \cdots = \alpha_n$ so that $\Laurentmax = \BZ\left[ \mt^{\pm 1} , q^{\pm \alpha} \right]$. Let $\beta \in \CB_n$ be a braid, and $\widehat{\beta}$ a homeomorphism representing $\beta$. The action of $\widehat{\beta}$ on $X_r$ described above lifts to $\Hrelm_r$, so that it yields a homological representation of the braid group:
\[
\Rhom: \bapp
\Laurentmax \left[ \CB_n \right] & \to & \End_{\Laurentmax} (\CH^{\alpha , \ldots, \alpha})
\eapp .
\]
\end{lemma}
\begin{proof}
Let $\sigma \in \CB_n$ be one of the Artin braid generator, the lemma is a direct consequence of the invariance of the local system representation under the braid action. Namely, the commutativity of the following diagram:
\begin{equation*}
\begin{tikzcd}[column sep=large]
\pi_1 (X_r , {\pmb \xi^r} ) \arrow[r,"\widehat{\sigma}_*"] \arrow[d,"\rho_r"]
   &  \pi_1 (X_r , {\pmb \xi^r} )  \arrow[d,"\rho_{r}"]\\ 
\BZ^{2} = \BZ \langle q^{\alpha} \rangle \oplus \BZ \langle t \rangle \arrow[r,"\Id"]
   & \BZ \langle q^{\alpha} \rangle \oplus \BZ \langle t \rangle
\end{tikzcd} 
\end{equation*}
where $\widehat{\sigma}$ is a half Dehn twist of $X_r$ associated with $\sigma$ and $\widehat{\sigma}_*$ its lift to the fundamental group. 
It is easy to see that for $l \in \lbrace 1 ,\ldots , r-1 \rbrace$ and $k \in \lbrace 1 , \ldots , n \rbrace$, the following relations hold:
\[
\rho_r \left( \widehat{\sigma}_*(\sigma_l) \right) = \rho_r(\sigma_l ) \text{ and } \rho_r \left( \widehat{\sigma}_*(B_{r,k}) \right) = \rho_r(B_{r,k} )
\]
considering generators $\sigma_l$ and $B_{r,k}$ of $\pi_1 (X_r , {\pmb \xi^r} )$ introduced in Remark \ref{pi_1X_r}. This ensures that the action lifts to the maximal abelian cover, and that it commutes with deck transformations, so that the action on homology with local coefficients is well defined. The action is invariant under isotopies, so that the action of $\Bn$ is well defined. 
\end{proof}
%

To deal with different colors, we need a morphism to follow the change of colors in $\Laurentmax$. 

\begin{defn}
Let $s \in \Sk_n$ be a permutation. We define the following morphism:
\[
\hat{s}: \bfct
\Laurentmax & \to & \Laurentmax \\
q^{\alpha_i} & \mapsto & q^{\alpha_{s(i)}} \\
t & \mapsto & t
\efct .
\]
\end{defn}

\begin{lemma}[Representation of the pure braid group]
In the general case, let $\sigma_i$ be an Artin generator of $\Bn$, with $i \in \lbrace 1 , \ldots  , n-1 \rbrace$ and $s \in \Sk_n$. There exists a well defined action of $\sigma_i$ lifted to homology:
\[
\Rhom(\sigma_i ) \in \Hom_{\Laurentmax}\left( \CH^{s(\alpha_1)} \otimes \cdots \otimes \CH^{s(\alpha_n)} , \CH^{s \tau_i(\alpha_1)} \otimes \cdots \otimes \CH^{s \tau_i(\alpha_n)}  \right)
\]
where $\tau_i = (i,i+1) \in \Sk_n$. There is a well defined action of $\PBn$ over $\CH^{\alpha_1 , \ldots ,\alpha_n}$:
\[
\Rhom: \bapp
\Laurentmax \left[ \PBn \right] & \to & \End_{\Laurentmax} (\CH^{\alpha_1 , \ldots, \alpha_n})
\eapp .
\]
This action commutes with the $\Laurentmax$-structure. 
\end{lemma}
\begin{proof}
The proof is almost the same as the one for Lemma \ref{RepBn}. Namely, it is a consequence of the fact that the following diagram commutes:
\begin{equation*}
\begin{tikzcd}[column sep=large]
\pi_1 (X_r , {\pmb \xi^r} ) \arrow[r,"\widehat{\sigma_i}_*"] \arrow[d,"\rho_r"]
   &  \pi_1 (X_r , {\pmb \xi^r} )  \arrow[d,"\rho_{r}"]\\ 
\BZ \langle q^{\pm \alpha_i}, t^{\pm 1} \rangle_{i \in \lbrace 1 , \ldots , n \rbrace} \arrow[r,"\widehat{\tau_k}"]
   & \BZ \langle q^{\pm \alpha_i}, t^{\pm 1} \rangle_{i \in \lbrace 1 , \ldots , n \rbrace} 
\end{tikzcd} .
\end{equation*}
The fact that this diagram commutes comes from the following remark:
\[
\widehat{\sigma_i}_* \left(B_{r,k} \right) = B_{r,k+1}
\]
while other generators of $\pi_1 (X_r , {\pmb \xi^r} )$ are not perturbed by the action of $\sigma_i$. For a pure braid $\beta$, we have:
\[
\Rhom(\beta ) \in \End_{\Laurentmax}\left( \CH^{\alpha_1} \otimes \cdots \otimes \CH^{\alpha_n}   \right) ,
\]
and as $\beta$ can be written as a composition of generators $\sigma_i$'s, by composition of diagrams, one obtains the following commutative diagram:
\begin{equation*}
\begin{tikzcd}[column sep=large]
\pi_1 (X_r , {\pmb \xi^r} ) \arrow[r,"\widehat{\beta}_*"] \arrow[d,"\rho_r"]
   &  \pi_1 (X_r , {\pmb \xi^r} )  \arrow[d,"\rho_{r}"]\\ 
\BZ \langle q^{\pm \alpha_i}, t^{\pm 1} \rangle_{i \in \lbrace 1 , \ldots , n \rbrace} \arrow[r,"\Id"]
   & \BZ \langle q^{\pm \alpha_i}, t^{\pm 1} \rangle_{i \in \lbrace 1 , \ldots , n \rbrace} 
\end{tikzcd} .
\end{equation*}
with the identity morphism on the second line coming from the pureness of $\beta$. This ends the proof as for previous lemma, Lemma \ref{RepBn}. 
\end{proof}
%

\subsubsection{Computation of the action.}\label{sectionbraidhabiro}

In the case of two punctures $w_1 , w_2$, we can perform the computation of the action of the single braid generator of $\CB_2$, and first we recall classical operators necessary to define the $R$-matrix of $\Uq$.

\begin{defn}
Let $q^{\frac{H\otimes H}{2}}$ be the following $\Laurentmax$-linear map:
\[
q^{\frac{H\otimes H}{2}} :
\bfct
\CH^{\alpha_1} \otimes \CH^{\alpha_2} & \to & \CH^{\alpha_1} \otimes \CH^{\alpha_2}\\
A^{\alpha_1}(k)\otimes A^{\alpha_2}(k') & \mapsto & q^{(\alpha_1- 2k)(\alpha_2 - 2k')/2} A^{\alpha_1}(k)\otimes A^{\alpha_2}(k')
\efct
\]
and $T$ be the following one:
\[
T: \bfct
\CH^{\alpha_1} \otimes \CH^{\alpha_2} & \to & \CH^{\alpha_2} \otimes \CH^{\alpha_1}\\
A^{\alpha_1}(k)\otimes A^{\alpha_2}(k') & \mapsto &  A^{\alpha_2}(k')\otimes A^{\alpha_1}(k)
\efct .
\]
where $A(k')^{\alpha_1}$, and $A(k)^{\alpha_2}$ are multi-arcs of $\CH^{\alpha_1}$, and $\CH^{\alpha_2}$ respectively.
\end{defn} 

\begin{lemma}[Braid action with two punctures]\label{braidedn=2}
Let $\mt=q^{-2}$. Let $k,k' \in \BN$, $\sigma_1$ be the standard generator of the braid group on two strands. Its action can be expressed as follows:
\[
\tens \left( \Rhom(\sigma_1) \left( A(k',k)^{\alpha_1 , \alpha_2} \right) \right) = \left[ q^{\frac{-\alpha_1 \alpha_2}{2}} q^{\frac{H \otimes H}{2}} \circ \sum_{l=0}^k \left( q^{\frac{l(l-1)}{2}} E^l \otimes F^{(l)} \right) \circ T \right] A(k')^{\alpha_1} \otimes A(k)^{\alpha_2} . 
\]
where $A(k',k)^{\alpha_1 , \alpha_2}$, $A(k')^{\alpha_1}$, and $A(k)^{\alpha_2}$ are vectors of $\CH^{\alpha_1,\alpha_2}$, $\CH^{\alpha_1}$, and $\CH^{\alpha_2}$ respectively. 
\end{lemma}
\begin{proof}
We have the following relations between homology classes:
\begin{align*}
\Rhom(\sigma_1) \left( A'(k',k)^{\alpha_1 , \alpha_2} \right) & = 
\Rhom(\sigma_1) \left(\vcenter{\hbox{
\begin{tikzpicture}[decoration={
    markings,
    mark=at position 0.5 with {\arrow{>}}}
    ]
\node (w0) at (-1,0) {};
\node (wn) at (1,0) {};
\node (wn1) at (0,0) {};
\coordinate (a) at (-2,-2);
\draw[dashed] (w0) to[bend right=50] node[pos=0.5] (k0) {\small $k$} (wn);
\draw[dashed] (w0) to[bend right=40] node[pos=0.3] (k1) {\small $k'$} (wn1);
\node[gray] at (w0)[left=5pt] {$w_0$};
\node[gray] at (wn)[above=5pt] {$w_2$};
\node[gray] at (wn1)[above=5pt] {$w_1$};
\foreach \n in {w0,wn,wn1}
  \node at (\n)[gray,circle,fill,inner sep=3pt]{};
\draw[double,red,thick] (k0) -- (k0|-a);
\draw[double,red,thick] (k1) -- (k1|-a);
\draw[gray,thick] (-1,0.3) -- (w0) -- (-1,-2.5);
\end{tikzpicture}
}} \right)
& =
\left(\vcenter{\hbox{
\begin{tikzpicture}[decoration={
    markings,
    mark=at position 0.5 with {\arrow{>}}}
    ]
\node (w0) at (-1,0) {};
\node (wn) at (1,0) {};
\node (wn1) at (0,0) {};
\coordinate (a) at (-2,-2);
\draw[dashed] (w0) to[bend right=30] node[pos=0.3] (k1) {\small $k'$} (wn);
\draw[dashed] (w0) to[bend right=90] node[pos=0.7] (k0) {\small $k$} (1.5,0);
\draw[dashed] (1.5,0) to[bend right=90] (wn1);
\node[gray] at (w0)[left=5pt] {$w_0$};
\node[gray] at (wn)[above=5pt] {$w_1$};
\node[gray] at (wn1)[above=5pt] {$w_2$};
\foreach \n in {w0,wn,wn1}
  \node at (\n)[gray,circle,fill,inner sep=3pt]{};
\draw[double,red,thick] (k0) -- (k0|-a);
\draw[double,red,thick] (k1) -- (k1|-a);
\draw[gray,thick] (-1,0.3) -- (w0) -- (-1,-2.5);
\end{tikzpicture}
}} \right)\\ 
\end{align*}
\begin{align*}
= 
\sum_{l=0}^k \left(\vcenter{\hbox{
\begin{tikzpicture}[decoration={
    markings,
    mark=at position 0.5 with {\arrow{>}}}
    ]
\node (w0) at (-1,0) {};
\node (wn) at (1.2,0) {};
\node (wn1) at (0,0) {};
\coordinate (a) at (-2,-2);
\draw[dashed] (w0) to[bend right=50] node[pos=0.3] (k0) {\small $k'$} (wn);
\draw[dashed] (w0) to[bend right=90] node[pos=0.7] (k1) {\small $l$} (1.5,0);
\draw[dashed] (1.5,0) arc (0:180:0.3) to[bend left=40] (w0);
\draw[dashed] (w0) to node[above,pos=0.5] (k2) {\small $k-l$} (wn1);
\node[gray] at (w0)[left=5pt] {$w_0$};
\node[gray] at (wn)[above=5pt] {$w_1$};
\node[gray] at (wn1)[above=5pt] {$w_2$};
\foreach \n in {w0,wn,wn1}
  \node at (\n)[gray,circle,fill,inner sep=3pt]{};
\draw[double,red,thick] (k0) -- (k0|-a);
\draw[double,red,thick] (k1) -- (k1|-a);
\draw[double,red] (-0.5,0) to[bend right=70] (0.8,0.5);
\draw[double,red] (0.8,0.5) arc (180:0:0.5) node (k3) {} to (k3|-a);
\draw[gray,thick] (-1,0.3) -- (w0) -- (-1,-2.5);
\end{tikzpicture}
}} \right)
& = 
\sum_{l=0}^k \mt^{-k'(k-l)} \mt^{-l(k-l)} q^{-2(k-l) \alpha_1} \left(\vcenter{\hbox{
\begin{tikzpicture}[decoration={
    markings,
    mark=at position 0.5 with {\arrow{>}}}
    ]
\node (w0) at (-1,0) {};
\node (wn) at (1.2,0) {};
\node (wn1) at (0,0) {};
\coordinate (a) at (-2,-2);
\draw[dashed] (w0) to[bend right=50] node[pos=0.5] (k0) {\small $k'$} (wn);
\draw[dashed] (w0) to[bend right=90] node[pos=0.7] (k1) {\small $l$} (1.5,0);
\draw[dashed] (1.5,0) arc (0:180:0.3) to[bend left=40] (w0);
\draw[dashed] (w0) to node[above,pos=0.5] (k2) {\small $k-l$} (wn1);
\node[gray] at (w0)[left=5pt] {$w_0$};
\node[gray] at (wn)[above=5pt] {$w_1$};
\node[gray] at (wn1)[above=5pt] {$w_2$};
\foreach \n in {w0,wn,wn1}
  \node at (\n)[gray,circle,fill,inner sep=3pt]{};
\draw[double,red,thick] (k0) -- (k0|-a);
\draw[double,red,thick] (k1) -- (k1|-a);
\draw[double,red,thick] (k2) -- (k2|-a);
\draw[gray,thick] (-1,0.3) -- (w0) -- (-1,-2.5);
\end{tikzpicture}
}} \right) .
\end{align*}
The second equality comes from a breaking of a dashed arc (Example \ref{breakingdashed}), the last one is a handle rule, for which we draw the corresponding braid in Figure \ref{braidedrule}.

\begin{figure}[h!]
\begin{center}
\def\svgscale{0.4}
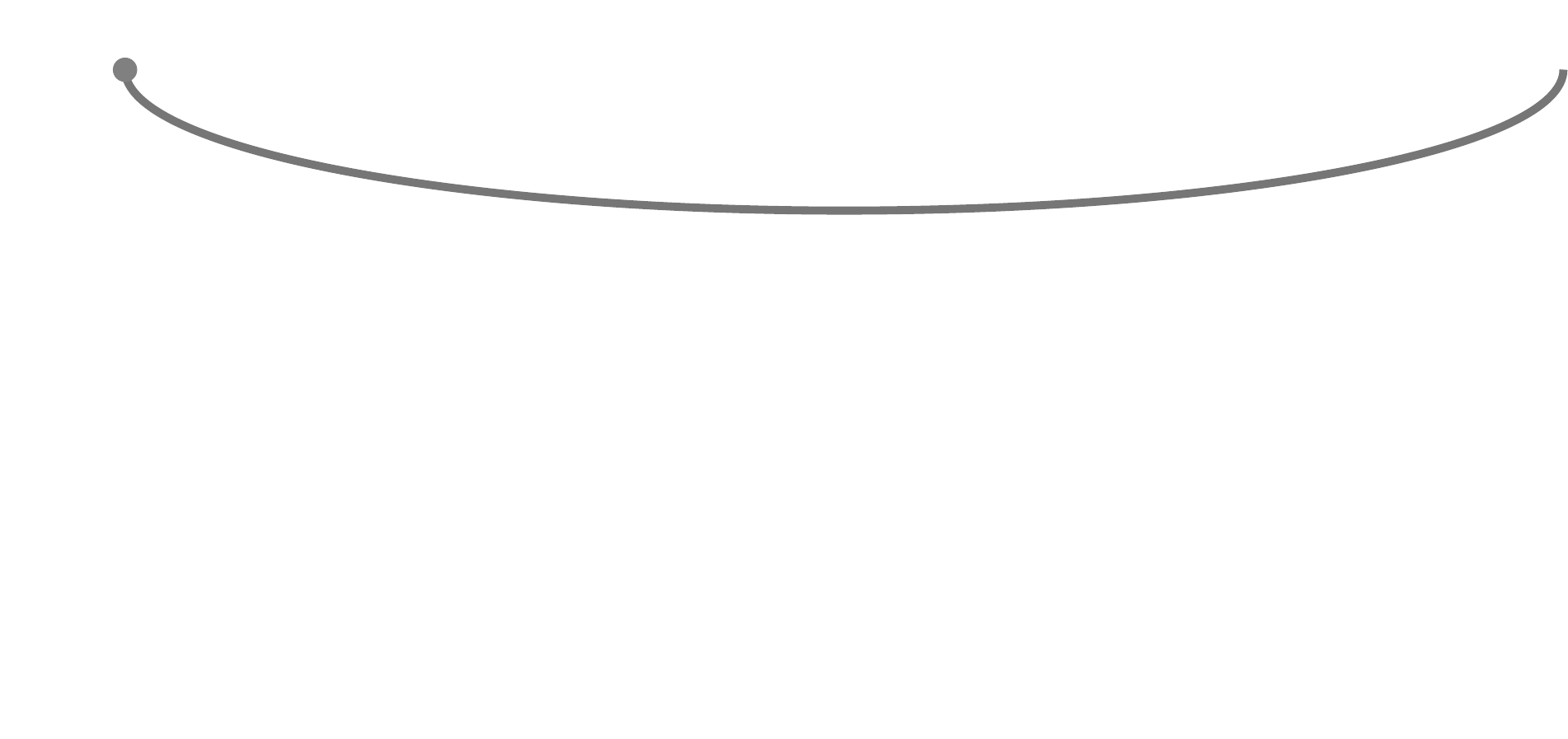
\caption{Braided handle rule. \label{braidedrule}}
\end{center}
\end{figure}

The bands represent a $(k-l)$-handle, a $(l)$-handle, and a $(k')$-handle. On the top and on the bottom of this box there is the part of the path corresponding to the dashed box. Namely red arcs are going back to ${\pmb \xi}$ without crossing themselves, passing in the front of $w_1$ and $w_2$. As this braid has $(k-l)$ strands passing successively in the back of $k'$ strands, $l$ strands and finally $w_2$, its local coefficient is $\mt^{-(k-l)(k'+l)}q^{-2 \alpha_2}$. From the local coefficient of this braid we deduce the coefficient appearing in the last term.

Finally, applying the proof of Proposition \ref{Fln=1} to crash a dashed loop on the indexed $k$ dashed arc, we get:
\[
\Rhom(\sigma_1) \left( A'(k',k)^{\alpha_1 , \alpha_2} \right) = \sum_{l=0}^k \mt^{-(k'+l)(k-l)} q^{-2(k-l) \alpha_1} {{k'+l}\choose{l}}_{\mt} \prod_{m=0}^l (1-q^{-2\alpha_1}\mt^{-m}) A'(k-l,k'+l)^{\alpha_2 , \alpha_1} .
\]
So that:
\[
\Rhom(\sigma_1) \left( A(k',k)^{\alpha_1 , \alpha_2} \right) = \sum_{l=0}^k \mt^{-(k'+l)(k-l)} q^{(-k+2l) \alpha_1} q^{-(k'+l)\alpha_2} {{k'+l}\choose{l}}_{\mt} \prod_{m=0}^l (1-q^{-2\alpha_1}\mt^{-m}) A(k-l,k'+l)^{\alpha_2 , \alpha_1} .
\]
Let $\mt=q^{-2}$, passing the above expression to $\tens$, we get for $\tens \left( \Rhom(\sigma_1^{\tau_1}) \left( A(k',k)^{\alpha_1 , \alpha_2} \right) \right)$ the following expression. 
\[
\sum_{l=0}^k q^{2(k'+l)(k-l) + (-k+2l) \alpha_1 -(k'+l)\alpha_2} {{k'+l}\choose{l}}_{\mt} \prod_{m=0}^l (1-q^{-2\alpha_1}\mt^{-m}) A(k-l)^{\alpha_2} \otimes A(k'+l)^{\alpha_1} .
\]
By use of the expression of the action of $F^{(l)}$ in Proposition \ref{Fln=1}, one recognizes:
\[
\left( \sum_{l=0}^k q^{2(k'+l)(k-l) + (-k+2l) \alpha_1 -(k'+l)\alpha_2} E^l \otimes F'^{(l)} \right) A(k)^{\alpha_2} \otimes A(k')^{\alpha_1}
\]
Finally, passing from $F'^{(l)}$ to $F^{(l)}$ 
we get:
\begin{align*}
\tens \left( \Rhom(\sigma_1) \left( A(k',k)^{\alpha_1 , \alpha_2} \right) \right) & = \left( \sum_{l=0}^k q^{2(k'+l)(k-l) - (k-l) \alpha_1 -(k'+l)\alpha_2} q^{\frac{l(l-1)}{2}} E^l \otimes F^{(l)} \right) A(k)^{\alpha_2} \otimes A(k')^{\alpha_1} \\
& = \left[ q^{\frac{-\alpha_1 \alpha_2}{2}} q^{\frac{H \otimes H}{2}} \circ \sum_{l=0}^k \left( q^{\frac{l(l-1)}{2}} E^l \otimes F^{(l)} \right) \circ T \right] A(k')^{\alpha_1} \otimes A(k)^{\alpha_2}
\end{align*}

\end{proof}
%
%
%
%
%

%

\begin{thm}[Recovering the $R$-Matrix of $\Uq$]\label{homoquantumbraided}
The representation of the pure braid group over $\CH^{\alpha_1 , \ldots , \alpha_n}$ (resp. the one of $\Bn$ over $\CH^{\alpha, \ldots , \alpha}$) is isomorphic to the $R$-matrix representation over the product of $\UqhL$ Verma-modules $V^{\alpha_1} \otimes \cdots \otimes V^{\alpha_n}$ (resp. over the product $(V^{\alpha})^{\otimes n}$) from Proposition \ref{goodRmatrix} (and Remark \ref{coloredquantum} for the colored version).
\end{thm}
\begin{proof}
From Lemma \ref{braidedn=2}, the following diagram:
\begin{equation*}
\begin{tikzcd}[column sep=large]
\CH^{\alpha_1,\alpha_2} \arrow[r,"\tens"] \arrow[d,"\Rhom \left( \sigma_1 \right)"]
   &  \CH^{\alpha_1} \otimes \CH^{\alpha_2}  \arrow[d,"q^{\frac{-\alpha_1 \alpha_2}{2}} R \circ T"]\\ 
\CH^{\alpha_2,\alpha_1} \arrow[r,"\tens"]
   & \CH^{\alpha_2} \otimes \CH^{\alpha_1}
\end{tikzcd} .
\end{equation*}
commutes. 
The action of a braid generator $\sigma_i$ over a multi-arc is contained in a disk that contains the dashed arcs reaching $w_i$ and $w_{i+1}$ and no other so that the action does not perturbate the other arcs. This last fact shows that the proof with two punctures guarantees the general case, and that the following diagram commutes:
\begin{equation*}
\begin{tikzcd}[column sep=large]
\CH^{\alpha_1,\ldots ,\alpha_n} \arrow[r,"\tens"] \arrow[d,"\Rhom \left( \sigma_i \right)"]
   &  \CH^{\alpha_1} \otimes \cdots \otimes \CH^{\alpha_n}  \arrow[d,"Q(\sigma_i)"]\\ 
\CH^{\tau_i \lbrace \alpha_1 , \ldots , \alpha_n \rbrace } \arrow[r,"\tens"]
   & \CH^{\alpha_{\tau_i(1)}} \otimes \cdots \otimes \CH^{\alpha_{\tau_i(n)}}
\end{tikzcd} .
\end{equation*}
where $Q(\sigma_i)=\Id^{\otimes i-1}\otimes q^{\frac{- \prod \alpha_i}{2}} R \circ T \otimes \Id^{\otimes n-i-1}$. Moreover all the morphisms involved commute with the $\Uq$ structure. This proves the theorem. 
\end{proof}

\section{Links with previous works}

\subsection{Integral version for Kohno's theorem}\label{integralKohno}

The following corollary recovers T. Kohno's theorem (\cite{Koh}, \cite[Theorem~4.5]{Itogarside}, recalled in Theorem \ref{existingtheorems} (iii)) in an integral version, namely with coefficients in $\Laurentmax$. It relates the action of the braid group over elements in the kernel of the action of $E$ with the action over absolute homology modules. 

\begin{coro} \label{kohnotheorem}
Under the condition $q^{-2}=\mt$, the restriction of the representation of the braid group $\Bn$ to the kernel of the homological action of $E$ yields a sub-representation of $\Bn$ in $\CH$ isomorphic to $\CH^E=\bigoplus_{r\in \BN^*} H^{\mathrm{BM}}_r \left( X_r ; L_r \right)$. 
\end{coro}
\begin{proof}
For $r \in \BN$, the relative long exact sequence of pairs gives this exact sequence of morphisms:
\begin{equation*}
\begin{tikzcd}
H_r(X_r^- ; L_r) \arrow[r] & H_r(X_r ; L_r) \arrow[r] & \Hrelm_r \arrow[r,"\partial_*"] & H_{r-1}(X_r^- ; L_r) \arrow[r] & H_{r-1}(X_r;L_r)
\end{tikzcd}
\end{equation*}
where we have avoided the notation $\mathrm{BM}$ as everything is Borel-Moore homology here.
Using Lemma 3.1 of \cite{Big0} one gets that $H_{r-1}(X_r;L_r)$ vanishes while Remark \ref{identificationHr1} above implies that $H_r(X_r^- ; L_r)$ vanishes. This provides a short exact sequence:
\[
\begin{tikzcd}
0 \arrow[r] & H_r(X_r ; L_r) \arrow[r] & \Hrelm_r \arrow[r,"\partial_*"] & H_{r-1}(X_r^- ; L_r) \arrow[r] & 0 .
\end{tikzcd}
\]
The kernel of the action of $E$ is exactly the kernel of the map $\partial_*$. This implies the corollary, as the kernel of the action of $E$ is isomorphic to the module of absolute homology. 
\end{proof}
%

Kohno's theorem \cite{Koh} holds only for generic choice of parameters, while in the above corollary all morphisms are defined over the ring of Laurent polynomials. Kohno's theorem in terms of bases, as it is stated for instance in \cite{Itogarside}, uses the basis of {\em multiforks} that we recall in the following notations.

\begin{Not}
For ${\bf k} \in \EZnr$ we let the multifork $F(k_0 , \ldots , k_{n-1})$ be the class in $\Hrelm_r$ assigned to the following picture:
\begin{equation*}
\begin{tikzpicture}[scale=0.8]
\node (w0) at (-5,0) {};
\node (w1) at (-3,0) {};
\node (w2) at (-1,0) {};
\node[gray] at (0.0,0.0) {\ldots};
\node (wn1) at (1,0) {};
\node (wn) at (3,0) {};
\draw (w0) to[bend right=30] node {\midarrow} (w1);
\draw (w0) to[bend right=-30] node {\midarrow} (w1);
\draw (w1) to[bend right=30] node {\midarrow} (w2);
\draw (w1) to[bend right=-30] node {\midarrow} (w2);
\draw (wn1) to[bend right=30] node {\midarrow} (wn);
\draw (wn1) to[bend right=-30] node {\midarrow} (wn);
\node at (w1)[above=3pt,left=23pt] {.};
\node at (w1)[left=23pt] {.};
\node at (w1)[below=4pt,left=23pt] {.};
\node at (w1)[above=20pt,left=20pt] {$k_0$};
\node at (w2)[above=3pt,left=23pt] {.};
\node at (w2)[left=23pt] {.};
\node at (w2)[below=4pt,left=23pt] {.};
\node at (w2)[above=20pt,left=20pt] {$k_1$};
\node at (wn)[above=3pt,left=23pt] {.};
\node at (wn)[left=23pt] {.};
\node at (wn)[below=4pt,left=23pt] {.};
\node at (wn)[above=20pt,left=18pt] {$k_{n-1}$};

\node[gray] at (w0)[left=5pt] {$w_0$};
\node[gray] at (w1)[above=5pt] {$w_1$};
\node[gray] at (w2)[above=5pt] {$w_2$};
\node[gray] at (wn1)[above=5pt] {$w_{n-1}$};
\node[gray] at (wn)[above=5pt] {$w_n$};
\foreach \n in {w1,w2,wn1,wn}
  \node at (\n)[gray,circle,fill,inner sep=3pt]{};
\node at (w0)[gray,circle,fill,inner sep=3pt]{};

\draw[red] (-4.5,0.25) -- (-4.5,-2);
\draw[red] (-2.5,0.25) -- (-2.5,-2);
\draw[red] (1.5,0.25) -- (1.5,-2);
\node[red] at (-4,-0.8) {$\ldots$};
\node[red] at (-2,-0.8) {$\ldots$};
\node[red] at (2,-0.8) {$\ldots$};
\draw[red] (-3.5,-0.25) -- (-3.5,-2);
\draw[red] (-1.5,-0.25) -- (-1.5,-2);
\draw[red] (2.5,-0.25) -- (2.5,-2);

\draw[dashed,gray] (-5,-2) -- (3,-2);
\draw[dashed,gray] (3,-2) -- (3,-3);

\node[gray,circle,fill,inner sep=0.8pt] at (-4.8,-3) {};
\node[below,gray] at (-4.8,-3) {$\xi_r$};
\node[below=5pt,gray] at (-4.2,-3) {$\ldots$};
\node[gray,circle,fill,inner sep=0.8pt] at (-3.5,-3) {};
\node[below,gray] at (-3.5,-3) {$\xi_1$};

\draw[red] (-4.8,-3) -- (-4.5,-2);
\draw[red] (-3.5,-3) -- (2.5,-2);
\node[red] at (-3,-2.4) {$\ldots$};

\draw[gray] (-5,0) -- (-5,1);
\draw[gray] (-5,0) -- (-5,-3);
\draw[gray] (-5,-3) -- (4,-3) node[right] {$\partial D_n$};
\end{tikzpicture}
\end{equation*}
\end{Not}

This recovers the consequences of Kohno's theorem that can be found in \cite{Itogarside}, stating that the family of multiforks is generically a basis of $H_{r} (X_r(w_0);L_r)$. We state precisely genericity conditions in the following corollary. 

\begin{prop}\label{genericFork}
Let ${\bf k} \in  \EZnr$, there is the following relation between the standards fork and code sequence associated with ${\bf k}$.
\[
F(k_0,\ldots , k_{n-1}) = \left( \prod_{i=0}^{n-1} {(k_i)_{\mt}!} \right)U(k_0,\ldots , k_{n-1}). 
\] 
This shows that the family $\CF$ is a basis of $\Hrelm_r$ whenever one works over a ring $R$ where all the $(i)_{\mt}!$ are invertible for $i$ an integer lower or equal to $r$. 
\end{prop}
\begin{proof}
The proof for the relation between multiforks and code sequences is a direct consequence of Corollary \ref{1forkto1code}.
\end{proof}

\begin{rmk}
Multi-forks in the kernel of the action of $E$ are multi-forks not reaching $w_0$, namely of type $F(0,k_1,\ldots,k_{n-1}$. This is the link with the multi-forks basis used in \cite[Theorem~4.5]{Itogarside}. 
\end{rmk}

\subsection{Felder-Wieczerkowski's conjectures}\label{FWconjectures}

In \cite{FW}, the authors use as a basis for their module elements called {\em $r$-loops}, for which we give a homological definition as follows.
\begin{Not}
For ${\bf k} \in \EZnr$ we call $L(k_0, \ldots , k_{n-1})$ a {\em $r$-loops}, the class in $\Hrelm_r$ assigned to the following drawing.
\begin{equation*}
\begin{tikzpicture}[scale=0.8,decoration={
    markings,
    mark=at position 0.5 with {\arrow{>}}}
    ] 
\node (w0) at (-5,0) {};
\node (w1) at (-3,0) {};
\node (w2) at (-0.5,0) {};
\node[gray] at (1.4,0.0) {\ldots};
\node (wn) at (3,0) {};

\draw[postaction={decorate}] (w0) to (-2.9,-0.25);
\draw[postaction={decorate}] (-2.9,-0.25) to[bend right=100] (-2.9,0.25);
\draw (-2.9,0.25) to (w0);
\node at (w1)[right=3pt] {$\ldots$};
\draw[postaction={decorate}] (w0) to (-2.3,-0.5);
\draw[postaction={decorate}] (-2.3,-0.5) to[bend right=100] (-2.3,0.5);
\draw (-2.3,0.5) to (w0);
\node at (w1)[above=20pt,right=5pt] {$k_0$};

\draw (w0)[postaction={decorate}] to[bend right=5] (-0.5,-1.8);
\draw[postaction={decorate}] (-0.5,-1.8) to[bend right=10] (-0.2,0.25) to[bend right=100] (-0.8,0.25) to[bend right=5] (-0.8,-1.5) to (w0);
\node at (w2)[right=5pt] {$\ldots$};
\draw[postaction={decorate}] (w0) to[bend right=5] (-0.3,-2.1) to[bend right=20] (0.3,0.4) to[bend right=100] (-1.3,0.4) to[bend right=5] (-1.1,-1.2) to (w0);
\node at (w2)[above=30pt,right=5pt] {$k_1$};

\draw[postaction={decorate}] (w0) to[->,bend right=15] (3,-3.2) to[bend right=5] (3.3,0.35) to[bend right=100] (2.7,0.35) to[bend right=5] (2.8,-2.9) to[bend left=15] (w0);
\node at (wn)[right=5pt] {$\ldots$};
\draw[postaction={decorate}] (w0) to[bend right=15] (3.4,-3.4) to[bend right=10] (4,0.5) to[bend right=100] (2,0.5) to[bend right=5] (2.4,-2.6) to[bend left=15] (w0);
\node at (wn)[above=18pt,right=2pt] {$k_{n-1}$};

\node[gray] at (w0)[left=5pt] {$w_0$};
\node[gray] at (w1)[left=5pt] {$w_1$};
\node[gray] at (w2)[below=5pt] {$w_2$};
\node[gray] at (wn)[below=5pt] {$w_n$};
\foreach \n in {w1,w2,wn}
  \node at (\n)[gray,circle,fill,inner sep=3pt]{};
\node at (w0)[gray,circle,fill,inner sep=3pt]{};

\draw[red] (-2.8,-0.22) -- (-2.8,-4);
\draw[red] (-2.1,-0.4) -- (-2.1,-4);
\node[red] at (-2.45,-3.7) {$\ldots$};
\draw[red] (-0.25,-0.9) -- (-0.25,-4);
\draw[red] (0.325,0.1) -- (0.325,-4);
\node[red] at (0.05,-3.7) {$\ldots$};
\draw[red] (3.2,-1.95) -- (3.2,-4);
\draw[red] (3.8,-2) -- (3.8,-4);
\node[red] at (3.5,-3.7) {$\ldots$};

\draw[dashed,gray] (-5,-4) -- (4,-4);
\draw[dashed,gray] (4,-4) -- (4,-5);

\node[gray,circle,fill,inner sep=0.8pt] at (-4.8,-5) {};
\node[below,gray] at (-4.8,-5) {$\xi_r$};
\node[below=5pt,gray] at (-4.2,-5) {$\ldots$};
\node[gray,circle,fill,inner sep=0.8pt] at (-3.5,-5) {};
\node[below,gray] at (-3.5,-5) {$\xi_1$};


\draw[red] (-4.8,-5) -- (-2.8,-4);
\draw[red] (-3.5,-5) -- (3.8,-4);
\node[red] at (-2,-4.4) {$\ldots$};

\draw[gray] (-5,0) -- (-5,1);
\draw[gray] (-5,0) -- (-5,-5);
\draw[gray] (-5,-5) -- (5,-5) node[right] {$\partial D_n$};
\end{tikzpicture}
\end{equation*}
\end{Not}

\begin{prop}\label{loopstoarcs}
Let ${\bf k} \in \EZnr$. There is the following relation between standards multi-arc and $r$-loop associated with ${\bf k}$.
\[
L(k_0 , \ldots , k_{n-1}) = \left( \prod_{i=0}^{n-1} (k_i)_{\mt} ! \prod_{k=0}^{k_i} \left(1 - q^{-2 \alpha_i} \mt^{-k} \right) \right) A'(k_0 , \ldots , k_{n-1}).
\] 
\end{prop}
\begin{proof}
To prove the proposition, one treats separately the loops winding around $w_1$, from those winding around $w_2$ etc. Every case is a straightforward recursion, using Remark \ref{looptodashed}, and leads to the formula of the proposition. 
\end{proof}

This answers Conjecture 6.1 from \cite{FW}. In fact it is a more precise statement saying exactly under which conditions the family of $r$-loops is a basis of the homology.

\begin{coro}[{\cite[Conjecture~6.1]{FW}}]\label{answerConjecture61}
If $R$ is a ring in which $(1 - q^{-2 \alpha_i} \mt^{-k})$ is invertible for all $i = 1 ,\ldots , n$ and so is $(k)_{\mt}!$ (for $k \leq r$), then $\Hrelm_r$ is a free $R$-module admitting the family $\CL$ of $r$-loops as basis.
\end{coro}

Actually, the lifts of the $r$-loops chosen in \cite{FW} are not exactly the same as ours, namely the handles we've chosen do not correspond to their choice of lift. As a change of lift corresponds to the multiplication by an invertible monomial of $\Laurentmax$, the conditions to be a basis are the same.

\section{Appendix}\label{appendix}

\subsection{Local coefficients}

\begin{rmk}
We recall that the representation $\rho_r$ defining the local system $L_r$ is canonically equivalent to the construction of a covering map over $X_r$. Namely, one can consider the universal cover $\widetilde{X_r}$ of $X_r$, upon which there is an action of $\pi_1(X_r)$. By making the quotient of $\widetilde{X_r}$ by the action of $\Ker{\rho_r} \subset \pi_1(X_r)$, one gets a cover $\widehat{X_r}$ of $X_r$. The group of deck transformations is then isomorphic to $Im( \rho_r) = \BZ^{n+1}$. There are three equivalent ways to build the chain complex with local coefficients in $L_r$:
\[
C_{\bullet}(X_r ; L_r) \simeq C_{\bullet}(\widetilde{X_r} , \BZ ) \otimes_{\pi_1(X_r)} \Laurentmax \simeq C_{\bullet} (\widehat{X_r}) . 
\]
The first one corresponds to complex with coefficients in a locally trivial bundle. In the middle one, the action of $\pi_1(X_r)$ is the one over the universal cover on the left, and given by $\rho_r$ on the right. The last one corresponds to singular chain complex of $\widehat{X_r}$ with the deck transformations action of $\Laurentmax$. 

We use $L_r$ or $\rho_r$ to designate both the representation of $\pi_1(X_r)$ or the covering $\widehat{X_r}$  together with the deck transformations group action. 
\end{rmk}

\subsection{Locally finite chains}\label{nonocompacthomologies}

In this work we use the locally finite version for singular homology which is isomorphic in our case to the Borel-Moore homology. This version controls the non-compact phenomena arising at punctures. We give general ideas and definitions of these homologies in this section. Let $X$ be a locally compact topological space. 

\begin{defn}[Locally finite homology]\label{lfhomologydef}
The {\em locally finite chain complex} associated with $X$ is the chain complex for which we allow infinite chains under the condition that their geometrical realization in $X$ is locally finite (for the topology of $X$). The latter guarantees that the boundary map is well defined.

Let $Y \subset X$. The {\em relative to $Y$} locally finite chain complex corresponds to the locally finite chain complex of $X$ mod out by the one of $Y$. The homology of locally finite chains is the homology complex corresponding to this definition of chain complexes. We use the notation $\Hlf_{\bullet}(X)$ to denote the locally finite homology complex. 
\end{defn}

\begin{rmk}[\cite{Big0}]\label{BMtoLF}
The homology of locally finite chains is isomorphic to the Borel-Moore homology that can be defined as follows:
\[
H_{\bullet}^{\mathrm{BM}} (X) = \varprojlim H_{\bullet} \left( X , X \setminus A \right) 
\]
where the inverse limit is taken over all compact subsets $A$ of $X$. The relative case is then the following:
\[
H_{\bullet}^{\mathrm{BM}} (X,Y) = \varprojlim H_{\bullet} \left( X , \left( X \setminus A \right) \cup Y \right) 
\]
for $Y \subset \partial X$. 
\end{rmk}

The above fact that Borel-Moore homology consists in a limit of homology complexes over compact spaces allows generalizations of many compact singular homology properties. 


Locally finite homology have very different properties than the usual ones when the space is non compact. We give first examples:

\begin{example} We give the example of the real line being a locally finite cycle, and a related example.
\begin{itemize}
\item[(Real line)] Any $0$-chain is null homologous (so that the $0$-homology does not encode connectedness). Let $p$ be a point, the chain:
\[
\sigma = \sum_{i=0}^{\infty} \left[ p+i , p+i+1 \right)
\]
has $p$ as Borel-Moore boundary. While the chain:
\[
\sum_{-\infty<k<\infty} \left[k,k+1 \right)
\]
has no boundary and hence is a cycle. This shows that $H^{\mathrm{BM}}_k(\mathbb{R}) = \BZ$ if $k=1$ and is $0$ otherwise and can be generalized to $H^{\mathrm{BM}}_k(\mathbb{R}^n) = \BZ$ if $k=n$ and is $0$ otherwise. 
\item[(Punctures)] Let $D_n$ be the punctured disk, and $c$ be a small circle running once around a puncture $p$. Then $c$ is a cycle using same kind of telescopic infinite chain as in the previous point. 
\end{itemize}
\end{example}

We emphasize that previous example generalizes.

\begin{rmk}\label{BMfirstcomputations}
\begin{itemize} There are the following facts.
\item[(Compact space)] If $X$ is compact, then the singular and locally finite homology are identical.
\item[(Submanifold)] In the spirit of previous example, any closed oriented submanifold defines a class in Borel–Moore homology, but not in ordinary homology unless the submanifold is compact. 
\end{itemize}
\end{rmk}

\end{document}